\documentclass[12pt]{amsart}
\usepackage[table,x11names]{xcolor}
\usepackage[left=2cm,right=2cm,
    top=2cm,bottom=2cm,bindingoffset=0cm]{geometry}

\usepackage{mathtools}

\usepackage{soul}
\usepackage{color}

\usepackage{nicefrac}
\usepackage{dirtytalk}

\usepackage{graphicx}
\graphicspath{ {./images/} }
\usepackage{subcaption}
\usepackage[export]{adjustbox}
\usepackage{wrapfig}

\usepackage{yfonts}

\usepackage{quiver}

\usepackage{multicol}
\usepackage{enumitem}

\usepackage{amsmath}
\usepackage{empheq}
\usepackage{epsfig}
\usepackage{wasysym}
\usepackage{extpfeil}
\usepackage{amsthm}
\usepackage{amssymb}
\usepackage{mathrsfs}
\usepackage{bm}
\usepackage{listings}
\usepackage{breqn}
\usepackage[usestackEOL]{stackengine}

\usepackage{float}
\numberwithin{equation}{section}
\usepackage{bm}

\usepackage{longtable}
\usepackage{makecell}
\usepackage[all]{xy}
\usepackage{longtable}
\usepackage{rotating}
\input xy
\xyoption{all}
\usepackage{multirow}
\usepackage{nicematrix}
\usepackage{xltabular}
\usepackage{tabularray}

\newtheorem{theorem}{Theorem}[section]
\newtheorem{lemma}{Lemma}[section]
\newtheorem*{lemma*}{Lemma}
\newtheorem*{corollary*}{Corollary}
\newtheorem{corollary}{Corollary}[section]
\theoremstyle{definition}
\newtheorem{definition}{Definition}[section]
\newtheorem*{proposition*}{Proposition}
\newtheorem{proposition}{Proposition}[section]

\newtheorem*{remark*}{Remark}

\usepackage{manfnt}

\usepackage{hyperref}
\hypersetup{
    colorlinks=true,
    linkcolor=blue,
    filecolor=magenta,      
    urlcolor=blue,
    citecolor=blue
}
\usepackage[nameinlink]{cleveref}
\usepackage{tikz-cd}
\usetikzlibrary{babel}

\DeclareMathOperator{\PGL}{PGL}

\DeclareMathOperator{\id}{id}

\DeclareMathOperator{\Img}{Im}
\DeclareMathOperator{\im}{im}

\DeclareMathOperator{\pr}{pr}

\DeclareMathOperator{\Aut}{Aut}
\DeclareMathOperator{\Norm}{N}

\DeclareMathOperator{\mult}{mult}
\DeclareMathOperator{\Bs}{Bs}
\DeclareMathOperator{\Bl}{Bl}
\DeclareMathOperator{\len}{length}

\DeclareMathOperator{\Pic}{Pic}
\DeclareMathOperator{\rank}{rk}

\DeclareMathOperator{\Stab}{Stab}

\newcommand{\ee}{\varepsilon}

\newcommand{\M}{\mathcal{M}}

\newcommand{\PPP}{{\mathbb{P}^1\times\mathbb{P}^1\times\mathbb{P}^1}}
\newcommand{\PP}{\mathbb{P}^1}

\makeatletter
\let\save@mathaccent\mathaccent
\newcommand*\if@single[3]{%
  \setbox0\hbox{${\mathaccent"0362{#1}}^H$}%
  \setbox2\hbox{${\mathaccent"0362{\kern0pt#1}}^H$}%
  \ifdim\ht0=\ht2 #3\else #2\fi
  }
\newcommand*\rel@kern[1]{\kern#1\dimexpr\macc@kerna}
\newcommand*\widebar[1]{\@ifnextchar^{{\wide@bar{#1}{0}}}{\wide@bar{#1}{1}}}
\newcommand*\wide@bar[2]{\if@single{#1}{\wide@bar@{#1}{#2}{1}}{\wide@bar@{#1}{#2}{2}}}
\newcommand*\wide@bar@[3]{%
  \begingroup
  \def\mathaccent##1##2{%
    \let\mathaccent\save@mathaccent
    \if#32 \let\macc@nucleus\first@char \fi
    \setbox\z@\hbox{$\macc@style{\macc@nucleus}_{}$}%
    \setbox\tw@\hbox{$\macc@style{\macc@nucleus}{}_{}$}%
    \dimen@\wd\tw@
    \advance\dimen@-\wd\z@
    \divide\dimen@ 3
    \@tempdima\wd\tw@
    \advance\@tempdima-\scriptspace
    \divide\@tempdima 10
    \advance\dimen@-\@tempdima
    \ifdim\dimen@>\z@ \dimen@0pt\fi
    \rel@kern{0.6}\kern-\dimen@
    \if#31
      \overline{\rel@kern{-0.6}\kern\dimen@\macc@nucleus\rel@kern{0.4}\kern\dimen@}%
      \advance\dimen@0.4\dimexpr\macc@kerna
      \let\final@kern#2%
      \ifdim\dimen@<\z@ \let\final@kern1\fi
      \if\final@kern1 \kern-\dimen@\fi
    \else
      \overline{\rel@kern{-0.6}\kern\dimen@#1}%
    \fi
  }%
  \macc@depth\@ne
  \let\math@bgroup\@empty \let\math@egroup\macc@set@skewchar
  \mathsurround\z@ \frozen@everymath{\mathgroup\macc@group\relax}%
  \macc@set@skewchar\relax
  \let\mathaccentV\macc@nested@a
  \if#31
    \macc@nested@a\relax111{#1}%
  \else
    \def\gobble@till@marker##1\endmarker{}%
    \futurelet\first@char\gobble@till@marker#1\endmarker
    \ifcat\noexpand\first@char A\else
      \def\first@char{}%
    \fi
    \macc@nested@a\relax111{\first@char}%
  \fi
  \endgroup
}
\makeatother

\definecolor{myblue}{rgb}{0.95,0.975,1}

\makeatother
\makeatletter

\begin{document}
\NewTblrTheme{mytheme}{ 
\SetTblrStyle{firstfoot}{\itshape} 
\SetTblrStyle{middlefoot}{\itshape} 
\SetTblrStyle{middlehead}{\itshape}
\SetTblrStyle{lasthead}{\itshape}
}
\title[$G$-birational rigidity of $\PPP$]{$G$-birational rigidity of the $G$-Fano variety $\PPP$}

\author[N. Stor]{Nikolai Stor}
\address{Department of Mathematics, University of Edinburgh, UK}
\email{N.Stor@sms.ed.ac.uk}

\date{\today}

\begin{abstract}
We determine the $G$-birational rigidity of the $G$-Fano variety $\PPP$ for all finite subgroups $G$ of $\Aut(\PPP)$.
\end{abstract}

\maketitle

\section{Introduction}
Fano varieties with many symmetries appear naturally in equivariant birational geometry. If a rationally connected variety is equipped with an action of a group $G$, then the $G$-equivariant Minimal Model Program ($G$-MMP) produces a $G$-Mori fiber space. The special case where the base is a point gives a $G$-Fano variety. Thus, $G$-Fano varieties are the basic objects whose equivariant birational geometry one wants to understand.

For a $G$-Fano variety $X$, this birational geometry is organized by the $G$-equivariant Sarkisov program. Namely, $G$-birational maps from $X$ to other $G$-Mori fiber spaces decompose into sequence of $G$-Sarkisov links. This leads to the following standard notions: 
\begin{enumerate}
    \item $X$ is $G$-birationally superrigid if no $G$-Sarkisov links start from $X$;
    \item $X$ is $G$-birationally rigid if every $G$-Sarkisov link starting from $X$ also ends at $X$ with the same $G$-action;
    \item $X$ is $G$-solid if it is not $G$-birational to a $G$-Mori fiber space with positive-dimensional base.
\end{enumerate}
These questions become particularly explicit in the toric setting. Let $X$ be a toric variety. We say that a subgroup $G\subset \Aut(X)$ is \emph{toric} if it normalizes a maximal torus $\mathbb{T}$ in $\Aut(X)$. Then $G$ fits into an exact sequence
\[
1\to T_G\to G\to W_G\to 1,
\]
where $T_G=G\cap \mathbb{T}$, and $W_G$ is the image of $G$ in the Weyl group of $X$.

Recall that we call a variety $X$ a $G$-Fano variety if the following conditions are satisfied:
\begin{enumerate}
    \item $X$ has terminal and $G\mathbb{Q}$-factorial singularities;
    \item $-K_X$ is ample;
    \item $X$ is $G$-minimal, that is $\rank\Pic(X)^G=1$.
\end{enumerate}
The paper~\cite{ToricGsolid} investigates $G$-birational rigidity of toric $G$-Fano threefolds for toric groups $G$. However, the finite toric subgroup case is not completely resolved there for two toric $G$-Fano varieties: a divisor $Y_{24}$ of degree $(1,1,1,1)$ in $(\mathbb{P}^1)^4$ and $\PPP$:
\begin{theorem}[\protect{\cite[Theorem 1.2]{ToricGsolid}}]
\label{GminimalGsolid}
Let $X = \PPP$ or $Y_{24}$, and let $G$ be a finite toric subgroup of $\Aut(X)$. Then 
\begin{enumerate}
    \item If $X$ is $G$-minimal and $G$-solid, then $W_G$ contains a subgroup isomorphic to $\mathfrak{A}_4$.
    \item If $W_G$ contains $\mathfrak{A}_4$, then $X$ is $G$-minimal.  
\end{enumerate}
Moreover, if $W_G$ contains $\mathfrak{A}_4$ and $|G|\geq 32 \cdot 24^4$, then $X$ is $G$-birationally superrigid.
\end{theorem}
The purpose of this paper is to complete the case of $\PPP$ and determine its $G$-birational behaviour for all finite toric groups $G$ and to investigate the non-toric groups as well.

Note, that for $X = \PPP$, the first two conditions of a $G$-Fano variety are automatic, since $\PPP$ is smooth and Fano. Hence, the only non-trivial condition is $G$-minimality. This condition is equivalent to the transitivity of the action of $G$ on the rulings of $\PPP$. For toric groups, $G$-minimality alone is not the only relevant constraint; indeed, $G$-solidity imposes an additional condition on the image $W_G$ of $G$ in the Weyl group of $X$: by Theorem~\ref{GminimalGsolid}, if $W_G$ does not contain $\mathfrak{A}_4$, then $\PPP$ is either not $G$-minimal, or it is $G$-minimal but not $G$-solid.

One of the obstructions to $G$-solidity of $\PPP$ is the existence of $G$-invariant pencils in $|\mathcal{O}_X(1,1,1)|$:
\begin{lemma}
\label{pencilsnotGsolid}
    Let $X = \PPP$ and let $G$ be a finite group in $\Aut(X)$ with $\rank\Pic(X)^G = 1$. Suppose there is a $G$-invariant pencil $W$ in $|\mathcal{O}_X(1,1,1)|$. Then $X$ is not $G$-solid.
\end{lemma} The proof follows the same strategy as in~\cite[Lemma 2.4]{GbirPnChelstov}.

The main result of this paper is the following:
\begin{theorem}
    \label{maintheorem}
    Let $X = \PPP$ and let $G$ be a finite subgroup of $\Aut(X)$ such that $X$ is $G$-minimal. Then $X$ is $G$-solid if and only if it is $G$-birationally superrigid. Moreover, $X$ is not $G$-solid if and only if one of the following is true: 
        \begin{enumerate}
            \item $G$ is toric and $W_G$ does not contain $\mathfrak{A}_4$.
            \item $G$ is toric with $W_G \supset \mathfrak{A}_4 $, and ${G \cong \mathfrak{A}_4\rtimes \boldsymbol{\mu}_4}$.
            \item there is a $G$-invariant $(1,1,1)$-curve.
            \item there is a $G$-invariant disjoint pair of $(1,1,1)$-curves.
            \item there is a $G$-invariant pencil in $|\mathcal{O}_X(1,1,1)|$.
        \end{enumerate}
\end{theorem}

The structure of the paper is as follows. In Section~\ref{Preliminaries}, we introduce the necessary definitions and notation. In Section~\ref{Groups}, we classify toric groups with $W_G = \mathfrak{A}_4$, as well as the groups with the other possible choices of $W_G$ and small toric kernels $T_G$. We also list the orbit decompositions that are needed later. In Section~\ref{Curves}, we first show that, for sufficiently large toric groups, the only invariant curve of small tri-degree is the union of torus-invariant curves. We then compute the invariant curves of small tri-degree for the remaining toric groups requiring separate analysis. In Section~\ref{Gbirrig}, we prove the main $G$-birational rigidity results for toric groups. In Section~\ref{nontoricgroups}, we classify non-toric groups and prove $G$-birational rigidity results for them. Finally, Appendix~\hyperref[appendixA]{A} contains the summary table of groups and the $G$-birational behaviour of $\PPP$ for each group.

All computer-assisted computations, including group-theoretic and orbit calculations, representation-theoretic decompositions, computations of ideals, linear systems and base loci, and verifications of irreducibility, smoothness, singularities and properties of explicit rational maps, were carried out in SageMath \cite{sagemath}. We omit the routine computational details for brevity.

\subsection*{Acknowledgements}
The author is very grateful to his supervisor, Ivan Cheltsov, for invaluable discussions and help during the preparation of this paper. The author’s work was supported by the UKRI Centre for Doctoral Training in Algebra, Geometry and Quantum Fields (AGQ), Grant Number EP/Y035232/1
\section{Preliminaries}
We work over the field of complex numbers $\mathbb{C}$.
\label{Preliminaries}
\subsection{$G$-varieties and their $G$-birational geometry}
\begin{definition}
    A variety $X$ together with a regular action $m_X$ of an algebraic group $G$ \[m_X\colon G\times X \rightarrow X\] is called a \emph{$G$-variety} and is denoted $(X,G)$.
\end{definition}
\begin{definition}
    Let $(X,G)$ and $(X',G')$ be two $G$-varieties. An \emph{equivariant rational map} is a pair consisting of a homomorphism of algebraic groups $\varphi\colon G \rightarrow G'$ and a rational map $\Phi\colon X \dashrightarrow X'$ such that the following diagram is commutative: 
    \[\begin{tikzcd}
	{G\times X} &&&& X \\
	\\
	{G'\times X'} &&&& {X'}
	\arrow["{m_{X}}", from=1-1, to=1-5]
	\arrow["{\varphi\times\Phi}"', dashed, from=1-1, to=3-1]
	\arrow["\Phi", dashed, from=1-5, to=3-5]
	\arrow["{m_{X'}}"', from=3-1, to=3-5]
\end{tikzcd}\]
We say that a rational map $\Phi$ is \emph{$\varphi$-equivariant}. If $G = G'$ and $\varphi$ is an automorphism, we say that $\Phi$ is a \emph{rational $G$-map}. If, moreover, $\varphi = \id_G$, we say that $\Phi$ is \emph{$G$-equivariant}.
\end{definition}
\begin{definition}
Let $X$ be a projective variety and let $G$ be a finite group acting on $X$. We call a morphism \[\varphi\colon X \rightarrow S\] to some variety $S$ with an action of $G$ a \emph{$G$-Mori fiber space} if it satisfies the following conditions:
\begin{enumerate}
    \item $\varphi$ is a surjective $G$-equivariant morphism with connected fibers;
    \item $X$ has terminal singularities, and each $G$-invariant Weil divisor on $X$ is a $\mathbb{Q}$-Cartier divisor (we call the second property $G\mathbb{Q}$-factoriality);
    \item $\dim S < \dim X$;
    \item $-K_X$ is $\varphi$-ample;
    \item $\rank \Pic(X)^G = \rank \Pic(S)^G+1$.
\end{enumerate}
If $S$ is a point, then $X$ is called a \emph{$G$-Fano variety}.
\end{definition}
$G$-Mori fiber spaces are one of the two possible outputs of the $G$-MMP; the other ones are $G$-minimal models. Recall the criterion that distinguishes those cases:
\begin{theorem}[\protect{\cite[3.1.1.]{GMMPProkhorov}}]
Let $X$ be a projective $G$-threefold with terminal and $G\mathbb{Q}$-factorial singularities. Then there exists a projective $G$-variety $Y$ and a $G$-equivariant birational map $\varphi\colon X \dashrightarrow Y$ such that $\varphi^{-1}$ does not contract any divisors, and one of the following is true: 
\begin{enumerate}
    \item if $K_X$ is pseudo-effective, then $Y$ is a $G$-minimal model, i.e., $K_Y$ is nef;
    \item if $K_X$ is not pseudo-effective, then there is a $G$-Mori fiber space structure $f\colon Y \rightarrow S$ on $Y$.
\end{enumerate}
\end{theorem}
The criterion also works relatively (with the necessary changes).

From~\cite{SarkisovProgram} we know that every $G$-equivariant birational map of threefold $G$-Mori fiber spaces can be decomposed into a composition of elementary $G$-Sarkisov links (and $G$-equivariant isomorphisms).
\begin{definition}
    Let $X$ be a $G$-Fano variety. We say that $X$ is:
    \begin{enumerate}
        \item \emph{$G$-birationally superrigid} if no $G$-Sarkisov links start at $X$;
        \item \emph{$G$-birationally rigid} if every $G$-Sarkisov link that starts at $X$ also ends at $X$ with the same $G$-action;
        \item \emph{$G$-solid} if $X$ is not $G$-birational to a $G$-Mori fiber space with a positive-dimensional base.
    \end{enumerate}
\end{definition}
It follows from the definition that for a $G$-Fano variety $X$ there is the following chain of implications: \[X \text{ is }G\text{-birationally superrigid} \Rightarrow X \text{ is }G\text{-birationally rigid} \Rightarrow X \text{ is }G\text{-solid}.\]
To prove that a $G$-Fano threefold is $G$-birationally superrigid it is enough to prove that it is $H$-birationally superrigid for some subgroup $H$ of $G$. More precisely:
\begin{lemma}[\protect{\cite[Remark 6.5]{FinColCheltsovShramov}}]
    \label{overgroupsGbirsuperrig}
    Let $G$ be a finite group in $\Aut(X)$ and let $H$ be a subgroup in $H$ such that $X$ is is $G$-Fano and $H$-Fano. and let $X$ be a $G$-Fano threefold. If $X$ is $H$-birationally superrigid, then $X$ is $G$-birationally superrigid.
\end{lemma}
Note that the same fact is not automatically true for $G$-birational rigidity, but there are no known counterexamples.
\begin{lemma}[\protect{\cite[Lemma~2.4]{ToricGsolid}}]
    \label{faithrep}
    Let $G$ be a reductive group acting generically free on $X$ and let $P$ be a smooth $G$-invariant point of $X$. Then the induced action of $G$ on the Zariski tangent space $T_P(X)$ is faithful.
\end{lemma}
\begin{lemma}[\protect{\cite[Corollary~2.5]{ToricGsolid}}]
    \label{singcyclic}
    A finite group of automorphisms of an irreducible curve fixing a smooth point is a cyclic group.
\end{lemma}
\begin{lemma}
    Every finite group $G$ acting effectively on a curve of genus $1$ is of the form: \begin{equation}
    \label{eq:Gellipticcurve}
        1 \rightarrow N \rightarrow G \rightarrow Q \rightarrow 1,
    \end{equation} where $N$ is a subgroup of translations and $Q =\boldsymbol{\mu}_1,\boldsymbol{\mu}_2,\boldsymbol{\mu}_3,\boldsymbol{\mu}_4$ or $\boldsymbol{\mu}_6$.
\end{lemma}
Suppose $X$ is a normal projective $G$-variety for a finite group $G$, and $\mathcal{M}\subset |\mathcal{L}|$ is a $G$-invariant linear system. Let $V\subset H^0(X,\mathcal{L})$ be the corresponding space of sections. Then we have a natural projective representation of $G$ on $\mathcal{M}=\mathbb{P}(V)$. However, it is not necessarily true that $G$ acts linearly on $V$. For that it might be necessary to consider a central extension of $G$. Note that, for a fixed space of sections, this convention does not affect the representation-theoretic data that we use below. Indeed, suppose that two finite central extensions $\widehat{G}_1$ and $\widehat{G}_2$ lift the same projective action of $G$ on $V$. Passing to the common central extension $\widehat{G}_1\times_G\widehat{G}_2$, we may compare the two linear actions on the same group. Since these two actions have the same projectivization, they differ by multiplication by a character of $\widehat{G}_1\times_G\widehat{G}_2$. In particular, the number of pairwise non-isomorphic irreducible summands, their dimensions, and their multiplicities are independent of the chosen lift. From now on, by abuse of notation, whenever we consider the representation of $G$ on a $G$-invariant space of sections $V\subset H^0(X,\mathcal{L})$ and its decomposition into irreducible $G$-modules, we always assume that we have chosen some finite central lift of $G$ first.

By a \emph{$G$-irreducible} variety we mean a, possibly reducible, variety such that $G$ acts transitively on its irreducible components. By a \emph{degenerate orbit} of a $G$-variety we mean an orbit of length strictly smaller than a generic orbit of this variety; equivalently, this is the orbit of a point whose stabilizer is strictly larger than the kernel of the action of $G$. 
\subsection{Singularities of log-pairs}
\label{singularities}
We also require some results about the singularities of log-pairs. Our main instrument in proving $G$-birational superrigidity is the following Noether--Fano criterion.
\begin{lemma}[\protect{\cite[Corollary~3.3.3]{IcosahedronCheltsovShramov}, see also \cite[Corollary~5.10]{RationalVarieties}}]
\label{mobilesystems}
   Let $G$ be a finite group and $X$ be a $G$-Fano variety. Suppose that for every non-empty $G$-invariant linear system $\M$ on the variety $X$ that has no fixed components, the log-pair $(X, \lambda \M)$ has canonical singularities, where $\lambda$ is a positive rational number such that \[K_X+\lambda\M \sim_{\mathbb{Q}}0.\] Then $X$ is $G$-birationally superrigid.
\end{lemma}
\begin{lemma}[\protect{\cite[Exercise~6.18]{RationalVarieties}}]
    \label{noncancurve}
    Let $C$ be an irreducible curve on $X$. Assume that \[\mult_{C}(\mathcal{M})\leq\frac{1}{\lambda}.\] Then $C$ is not a center of non-canonical singularities of the log-pair $(X,\lambda\M)$.
\end{lemma}
\begin{lemma}[\protect{\cite[Theorem~3.1]{SingCorti}}]
    \label{ineqCorti}
    Let $C$ be an irreducible curve on $X$. Assume that $C$ is a center of non-log canonical singularities of the log-pair $(X,\lambda\M)$. Then \[\mult_{C}(M_1 \cdot M_2)>\frac{4}{\lambda^2}\] for any two general surfaces $M_1$ and $M_2$ in the linear system $\M$. 
\end{lemma}
\begin{lemma}[\protect{\cite{BirAutFanoPukhlikov}}]
    \label{ineqPukhlikov}
    Let $P$ be a smooth point on $X$. Assume that $P$ is a center of non-canonical singularities of the log-pair $(X,\lambda\M)$. Then \[\mult_{P}(M_1 \cdot M_2)>\frac{4}{\lambda^2}\] for any two general surfaces $M_1$ and $M_2$ in the linear system $\M$. 
\end{lemma}
\begin{lemma}[\protect{\cite[Lemma~2.4]{DoubVerCones}}]
\label{irrepmult}
    Let $G$ be an algebraic group acting on $X$ and let $\mathcal{M}$ be $G$-invariant. Suppose that $G$ fixes a smooth point $P \in X$ and that its induced action on $T_P(X)$ is an irreducible representation. If $P$ is a non-canonical center of $(X,\lambda \mathcal{M})$ then $\mult_P(\mathcal{M})>\frac{2}{\lambda}$.
\end{lemma}
Note that if an irreducible variety $Z$ is a center of non-canonical/canonical/non-log canonical/log canonical singularities of a log-pair $(X,\lambda\M)$, where $\M$ is $G$-invariant, then any element $Z'$ of the $G$-orbit of $Z$ is so too, and the multiplicity of $\M$ at $Z'$ is the same. Thus, it makes sense to speak of $G$-irreducible varieties as centers of non-canonical/canonical/non-log canonical/log canonical singularities. We call such a $G$-irreducible subvariety a \emph{$G$-center} of the corresponding singularities.

We will also utilize the alpha-invariant. Recall the definition.
\begin{definition}
    Let $X$ be a $G$-Fano variety for a finite group $G \subset \Aut(X)$. Then the \emph{$\alpha$-invariant} $\alpha_G(X)$ of $(X,G)$ is \[\sup\left\{\lambda \in \mathbb{Q}\ \middle|\ (X,\lambda D) \text{ is log-canonical for every effective }G\text{-invariant }\mathbb{Q}\text{-divisor }D\sim_{\mathbb{Q}} -K_X\right\}.\] Equivalently, it can be defined using linear systems as: \[\sup\left\{\lambda \in \mathbb{Q}\ \middle|\ \left(X,\frac{\lambda}{n}\mathcal{M}\right)\text{ is log-canonical for every } n \text{ and every } G\text{-invariant }\mathcal{M} \subset |-nK_X|\right\}.\]
\end{definition}
\begin{theorem}[\protect{\cite[Lemma 2.16]{FiveEmbed}}]
\label{alphaPP}
Suppose $X \cong \PP \times \PP$ and $G$ is a finite subgroup of $\Aut(X)$. Then the following are equivalent:
\begin{enumerate}
    \item $\alpha_G(X) \geq 1$;
    \item There are no $G$-invariant curves in the linear systems \[|L_1|,\ |L_2|,\ |L_1+L_2|,\] where $L_1$ and $L_2$ are fibers of two distinct natural projections of the surface $X$ to $\mathbb{P}^1$. 
\end{enumerate}
\end{theorem}
\subsection{Notation and basic geometry of $\PPP$}
\label{Notationsandbasics}
Let $X = \PPP$. We will write coordinates on it as \[\left([x_1:y_1],[x_2:y_2],[x_3:y_3]\right).\] Then \[\Aut(X) \cong (\PGL_2)^3 \rtimes \mathfrak{S}_3,\] where each $\PGL_2$ acts on the corresponding $\mathbb{P}^1$ and $\mathfrak{S}_3$ permutes the factors. We will write elements of $\Aut(X)$ as \[\left(A_1,A_2,A_3;\sigma\right),\] where $A_i \in \PGL_2$, and $\sigma \in \mathfrak{S}_3$. Then the action of $\Aut(X)$ on $X$ is given as \[\left(A_1,A_2,A_3;\sigma\right)\cdot\left(p_1,p_2,p_3\right) = \left(A_1 p_{\sigma^{-1}(1)},A_2 p_{\sigma^{-1}(2)},A_3 p_{\sigma^{-1}(3)}\right),\] and the multiplication in $\Aut(X)$ is \[\left(A_1,A_2,A_3;\sigma\right) \cdot \left(B_1,B_2,B_3;\pi\right) = (A_1B_{\sigma^{-1}(1)},A_2B_{\sigma^{-1}(2)},A_3B_{\sigma^{-1}(3)};\sigma \circ \pi).\] We also fix the contragredient convention for the action of $G$ on functions/sections: \[g\cdot f(x) = f(g^{-1}\cdot x).\]

Let $\mathbb{T} \cong \mathbb{G}_m^3$ be a maximal torus in $\Aut(X)$. We pick the standard maximal torus, defined as \[\mathbb{T} = \left\{\left(
    \begin{pmatrix}
        t_1 & 0\\
        0 & 1 
    \end{pmatrix},
    \begin{pmatrix}
        t_2 & 0\\
        0 & 1 
    \end{pmatrix},
    \begin{pmatrix}
        t_3 & 0\\
        0 & 1 
    \end{pmatrix};\,\id \right)\ \middle|\ t_1,t_2,t_3 \in \mathbb{C}^*\right\}.\] We will write $t$ for the diagonal matrix \[
    \begin{pmatrix}
        t & 0\\
        0 & 1 
    \end{pmatrix},\] and we will write $\iota(t)$ for the anti-diagonal matrix \[\begin{pmatrix}
        0 & t\\
        1 & 0 
    \end{pmatrix}.\] We will also write \[(g_1,g_2,g_3)\] for the elements \[(g_1,g_2,g_3;\id).\] Thus, we denote elements of $\mathbb{T}$ as \[(t_1,t_2,t_3).\] 

Let $\mathbb{W}\cong \mathfrak{S}_4 \times \boldsymbol{\mu}_2$ be the Weyl group of $X$ as a toric variety. We choose the following lift of $\mathbb{W}$ in $\Norm_{\Aut(X)}(\mathbb{T})$:  \[\mathbb{W} = \left\langle \mathfrak{S}_3, \left(
   \iota(1),1,1 \right)\right\rangle,\] where $\mathfrak{S}_3$ acts as the pure permutations of the factors. Then there is an exact sequence: \[1\rightarrow \mathbb{T}\rightarrow \mathrm{N}_{\Aut(X)}(\mathbb{T})\rightarrow\mathbb{W}\rightarrow 1,\] where $\mathrm{N}_{\Aut(X)}(\mathbb{T})$ is the normalizer of $\mathbb{T}$ in $\Aut(X)$. By a \textit{toric group} we mean a subgroup of $\Aut(X)$ that normalizes a maximal torus, i.e. a subgroup that lies in $\mathrm{N}_{\Aut(X)}(\mathbb{T})$ up to conjugation. Every finite subgroup $G$ of $\mathrm{N}_{\Aut(X)}(\mathbb{T})$ fits into the exact sequence: \[1 \rightarrow T_G\rightarrow G \xrightarrow{\pi_{\mathbb{T}}} W_G \rightarrow 1.\]

    Let us describe the structure of $\mathbb{T}$-orbits on $\PPP$. 
    \begin{lemma} The $\mathbb{T}$-orbit structure of $\PPP$ is the following:
        \begin{enumerate}
            \item There are eight $\mathbb{T}$-invariant points in $\PPP$:
            \begin{align*}
                &\left([0:1],[0:1],[0:1]\right),\ \left([0:1],[1:0],[1:0]\right),\ \left([1:0],[0:1],[1:0]\right),\ \left([1:0],[1:0],[0:1]\right),\\
                &\left([1:0],[1:0],[1:0]\right),\ \left([1:0],[0:1],[0:1]\right),\ \left([0:1],[1:0],[0:1]\right),\ \left([0:1],[0:1],[1:0]\right).
            \end{align*}
            \item There are twelve $\mathbb{T}$-invariant irreducible curves in $\PPP$:
            \begin{align*}
                &\mathbb{P}^1 \times \{[0:1]\} \times \{[0:1]\},\ \mathbb{P}^1 \times \{[1:0]\} \times \{[0:1]\},\ \mathbb{P}^1 \times \{[0:1]\} \times \{[1:0]\},\\
                &\mathbb{P}^1 \times \{[1:0]\} \times \{[1:0]\},\ \{[0:1]\} \times \mathbb{P}^1 \times \{[0:1]\},\ \{[1:0]\} \times \mathbb{P}^1 \times \{[0:1]\},\\
                &\{[0:1]\} \times \mathbb{P}^1 \times \{[1:0]\},\ \{[1:0]\} \times \mathbb{P}^1 \times \{[1:0]\},\ \{[0:1]\} \times \{[0:1]\} \times \mathbb{P}^1,\\
                &\{[1:0]\} \times \{[0:1]\} \times \mathbb{P}^1,\ \{[0:1]\} \times \{[1:0]\} \times \mathbb{P}^1,\ \{[1:0]\} \times \{[1:0]\} \times \mathbb{P}^1.
            \end{align*}
            \item There are six $\mathbb{T}$-invariant prime divisors in $\PPP$:
            \begin{align*}
                &\PP \times \PP \times \{[0:1]\},\ \PP \times \PP \times \{[1:0]\},\ \PP \times \{[0:1]\} \times \PP,\\
                &\PP \times \{[1:0]\} \times \PP, \ \{[0:1]\} \times \PP \times \PP,\ \{[1:0]\} \times \PP \times \PP.
            \end{align*}
        \end{enumerate}
    \end{lemma}
    \begin{lemma}
        Let $G$ be a finite toric subgroup of $\Aut(\PPP)$ such that $W_G$ contains $\mathfrak{A}_4$. Then 
        \begin{enumerate}
            \item The $\mathbb{T}$-invariant points form a unique $G$-orbit if and only if $W_G$ is $\mathfrak{S}_4$, $\mathfrak{A}_4 \times \boldsymbol{\mu}_2$, or $\mathfrak{S}_4 \times \boldsymbol{\mu}_2$. Otherwise, the set of $\mathbb{T}$-invariant points forms two $G$-orbits of length 4: \[\boldsymbol{\mathcal{O}}_{4,0}:\left([0:1],[0:1],[0:1]\right),\ \left([0:1],[1:0],[1:0]\right),\ \left([1:0],[0:1],[1:0]\right),\ \left([1:0],[1:0],[0:1]\right),\] and \[\boldsymbol{\mathcal{O}}_{4,\infty}: \left([1:0],[1:0],[1:0]\right),\ \left([1:0],[0:1],[0:1]\right),\ \left([0:1],[1:0],[0:1]\right),\ \left([0:1],[0:1],[1:0]\right).\]
            \item The $\mathbb{T}$-invariant irreducible curves form a unique $G$-irreducible curve $\boldsymbol{\mathcal{C}}_{\mathbb{T}}$.
            \item The $\mathbb{T}$-invariant prime divisors form a unique $G$-irreducible divisor $\boldsymbol{\mathcal{D}}_{\mathbb{T}}$.
        \end{enumerate}
    \end{lemma}

    Let us also recall the intersection theory on $\PPP$. Denote by $\pr_i$ the projection of $\PPP$ onto the $i$-th factor $\PP$, and denote \[H_i = \pr_i^*\mathcal{O}_{\mathbb{P}^1}(1),\ i = 1,2,3.\] Also denote by $\pr_{i,j}$ the projection of $\PPP$ onto the product of the $i$-th and $j$-th factors. Then the Chow ring of $\PPP$ is \[A^*(\PPP) = \mathbb{Z}[H_1,H_2,H_3]/(H_1^2,H_2^2,H_3^2).\] In particular, 
    \begin{align*}
        &A^0(X) = \mathbb{Z},\\
        &A^1(X) \cong \Pic(X) = \mathbb{Z}H_1 \oplus \mathbb{Z}H_2 \oplus \mathbb{Z}H_3,\\
        &A^2(X) = \mathbb{Z}H_2\cdot H_3 \oplus \mathbb{Z}H_1\cdot H_3 \oplus \mathbb{Z}H_1\cdot H_2,\\
        &A^3(X) = \mathbb{Z}H_1\cdot H_2\cdot H_3,
    \end{align*} with the $H_1 \cdot H_2 \cdot H_3 = 1$ convention. The \emph{tri-degree of a divisor} $D$ is a triple \[(D\cdot H_2\cdot H_3,D\cdot H_1\cdot H_3,D\cdot H_1\cdot H_2),\] and the \emph{tri-degree of a curve} $C$ is a triple \[(C\cdot H_1, C\cdot H_2, C\cdot H_3).\]

Suppose $D_1$ and $D_2$ are divisors with tri-degrees $(a_1,b_1,c_1)$ and $(a_2,b_2,c_2)$, respectively. Then $D_1 \cdot D_2$ has a tri-degree \[(b_1c_2+b_2c_1,a_1c_2+a_2c_1,a_1b_2+a_2b_1).\] Now suppose that $D$ is a divisor of tri-degree $(a,b,c)$ and $C$ is a curve of tri-degree $(k,l,m)$. Then \[D\cdot C = ak+bl+cm.\] 

Denote \[H = H_1+H_2+H_3,\] and call the \emph{$H$-degree} of a curve $C$ the number $C\cdot H$. We also know that \[K_X = -2H.\]
\section{Groups}
\label{Groups}
This section is dedicated to the classification of finite toric groups on $\PPP$. However, we do not need all finite toric groups. Namely, by Theorem~\ref{GminimalGsolid}, if $X$ is $G$-minimal and $G$-solid, then $W_G$ contains (the unique) $\mathfrak{A}_4$; only those groups are relevant for our analysis. From now on, whenever we consider subgroups of $\Norm_{\Aut(X)}(\mathbb{T})$, we always assume that $W_G = \Img \pi_{\mathbb{T}}$ contains $\mathfrak{A}_4$. In $\mathbb{W}$, there are five subgroups that contain $\mathfrak{A_4}$: 
    \begin{align*}
    \mathfrak{A}_4 &= \left\langle\left(1,1,1;\,(1,2,3)\right),\left(
   \iota(1),1,\iota(1)\right)\right\rangle \\
    \widebar{\mathfrak{S}}_4 &= \left\langle\left(
    \iota(1),1,\iota(1);\,(2,3)\right), \left(
   1,1,1;\,(1,2)\right)\right\rangle \\
    \mathfrak{S}_4 &= \left\langle\left(
    1,\iota(1),1;\,(1,2)\right),\left(
    \iota(1),\iota(1),\iota(1);\,(2,3)\right)\right\rangle \\
    \mathfrak{A}_4 \times \boldsymbol{\mu}_2 &= \left\langle\left(1,1,1;\,(1,2,3)\right),\left(
   1,\iota(1),1\right)\right\rangle \\
    \mathfrak{S}_4 \times \boldsymbol{\mu}_2 &= \left\langle\left(
    \iota(1),\iota(1),\iota(1);\,(1,2)\right),\left(
    \iota(1),1,\iota(1);\,(2,3)\right)\right\rangle 
    \end{align*} with the following presentations on the chosen generators: 
    \begin{align*}
        \mathfrak{A}_4 &= \left\langle a,b \mid a^3 = b^2 = (ab)^3 = 1\right\rangle\\
        \widebar{\mathfrak{S}}_4 \cong \mathfrak{S}_4 &= \left\langle a,b \mid a^4 = b^2 = (ab)^3 = 1\right\rangle\\
        \mathfrak{A}_4 \times \boldsymbol{\mu}_2 &= \left\langle a,b \mid a^3 = b^2 = (ab)^6 = [a,(ab)^3] =1\right\rangle\\
        \mathfrak{S}_4 \times \boldsymbol{\mu}_2 &= \left\langle a,b \mid a^2 = b^4 = (ab)^6 = [b,(ba)^3] =1\right\rangle
    \end{align*}
    
We now want to understand which finite groups are possible. For that, we need to know which subgroups of $\mathbb{T}$ are possible candidates for finite toric kernels $T_G$. We denote by $\ee_{n}$ the primitive $n$-th root of unity.
\begin{lemma}[\protect{\cite[Lemma~7.1]{ToricGsolid}}]
    Let $W_G$ be a subgroup of $\mathbb{W}$ that contains $\mathfrak{A}_4$ and let $T_G$ be a $W_G$-invariant subgroup of $\mathbb{T}$. Then there exists $n \in \mathbb{N}$ such that one of the following three possibilities holds: 
    \begin{enumerate}
        \item $T_G \cong \boldsymbol{\mu}^3_n=\left\langle\right(\ee_n,1,1),(1,\ee_n,1),(1,1,\ee_n)\rangle$;
        \item $n$ is even, and $T_G \cong \boldsymbol{\mu}^2_n \times \boldsymbol{\mu}_{\frac{n}{2}}=\left\langle\right(\ee_n,\ee_n,1),(1,\ee_n,\ee_n),(1,1,\ee^2_n)\rangle$;
        \item $n$ is even, and $T_G \cong \boldsymbol{\mu}_n \times \boldsymbol{\mu}_{\frac{n}{2}}^2=\left\langle\right(\ee_n,\ee_n,\ee_n),(1,\ee_n^2,1),(1,1,\ee_n^2)\rangle$,
    \end{enumerate}
    where $\ee_n$ is a primitive $n$-th root of unity.
\end{lemma}

Let us classify groups with $W_G = \mathfrak{A}_4$.
\begin{proposition}
\label{groupsodd}
For each odd $n$ and $W_G = \mathfrak{A}_4$ there are, up to conjugation, two possible groups: 
\begin{enumerate}
    \item $G_{12n^3} = \left\langle (\varepsilon_n,1,1),(1,\varepsilon_n,1),(1,1,\varepsilon_n),\left(
    1,1,1;\,(1,2,3)\right), \left(
    \iota(1),1,\iota(1)\right)\right\rangle$.

    \item $G_{12n^3}' = \left\langle (\varepsilon_n,1,1),(1,\varepsilon_n,1),(1,1,\varepsilon_n),\left(
    i,-i,1;\,(1,2,3)\right), \left(
    \iota(1),-1,\iota(1)\right)\right\rangle$.
\end{enumerate}
\end{proposition}
\begin{proof}
    Let \[A = \left(
a_1,a_2,a_3;\,(1,2,3)\right),\ B = \left(
\iota(b_1),b_2,\iota(b_3)\right)\] be elements in $G$ such that $\pi_{\mathbb{T}}(A)$ and $\pi_{\mathbb{T}}(B)$ are the chosen generators of $\mathfrak{A}_4$. Conjugating by elements of $\mathbb{T}$, we may assume that $a_3 = b_1 = b_3 =1$. Then $A^3 = (a_1a_2,a_1a_2,a_1a_2) \in T_G$, $B^2 = (1, b_2^2,1) \in T_G$, and \[(AB)^3 = \left(\frac{a_1}{a_2b_2},\frac{a_2b_2}{a_1},\frac{a_2b_2}{a_1}\right) \in T_G.\] Since $n$ is odd, \[T_G = \boldsymbol{\mu}^3_n = \left\langle (\varepsilon_n,1,1),(1,\varepsilon_n,1),(1,1,\varepsilon_n)\right\rangle.\] Hence \[
\begin{cases}
a_1a_2 = \varepsilon^{i_1}_n\\
b_2^2 = \varepsilon^{i_2}_n\\
\frac{a_1}{a_2b_2} = \varepsilon^{i_3}_n
\end{cases},\] for some $i_1,i_2,i_3$. Solving and then conjugating by the elements of the normalizer of $\mathbb{T}$, we get the possibilities in the list.
\end{proof}

\begin{proposition}
\label{groupseven}
For each even $n$ and $W_G = \mathfrak{A}_4$ there are, up to conjugation, the following possibilities: 
\begin{enumerate}
    \item $G_{12n^3} = \left\langle (\varepsilon_n,1,1),(1,\varepsilon_n,1),(1,1,\varepsilon_n),\left(
    1,1,1;\,(1,2,3)\right), \left(
    \iota(1),1,\iota(1)\right)\right\rangle.$
    \item $G_{12n^3}' = \left\langle (\varepsilon_n,1,1),(1,\varepsilon_n,1),(1,1,\varepsilon_n),\left(
    \ee_{4n},\ee_{4n}^{-1},1;\,(1,2,3)\right), \left(
    \iota(1),\ee_{2n},\iota(1)\right)\right\rangle.$
    \item $G_{6n^3} = \left\langle (\varepsilon_n,\varepsilon_n,1),  (1,\varepsilon_n,\varepsilon_n),(1,1,\varepsilon^2_n),\left(
    1,1,1;\,(1,2,3)\right), \left(
    \iota(1),1,\iota(1)\right)\right\rangle$.
    \item $G_{3n^3} = \left\langle (\varepsilon_n,\varepsilon_n,\varepsilon_n),(1,\varepsilon^2_n,1),(1,1,\varepsilon^2_n),\left(
    1,1,1;\,(1,2,3)\right), \left(
    \iota(1),1,\iota(1)\right)\right\rangle$.
    \item $G_{3n^3}' = \left\langle (\varepsilon_n,\varepsilon_n,\varepsilon_n),(1,\varepsilon^2_n,1),(1,1,\varepsilon^2_n),\left(
    \ee_{2n},\ee_{2n}^{-1},1;\,(1,2,3)\right), \left(
    \iota(1),1,\iota(1)\right)\right\rangle$.
\end{enumerate}
\end{proposition}
\begin{proof}
As in the case of an odd $n$, we get \[A^3 = (a_1a_2,a_1a_2,a_1a_2),\ B^2 = (1, b_2^2,1),\ (AB)^3 = \left(\frac{a_1}{a_2b_2}, \frac{a_2b_2}{a_1},\frac{a_2b_2}{a_1}\right) \in T_G.\] Solving for each possible $T_G$ and then conjugating by the elements of the normalizer of $\mathbb{T}$, we get the possibilities in the list.
\end{proof}
Note that $G_{12\cdot(\nicefrac{n}{2})^3} \subset G_{6n^3}$ and $G_{12\cdot(\nicefrac{n}{2})^3} \subset G_{3n^3}$. Also note that there is a subgroup conjugate to $G_{12\cdot(\nicefrac{n}{2})^3}'$ inside $G_{3n^3}'$. Indeed, in the case of $T_G = \boldsymbol{\mu}_n \times \boldsymbol{\mu}^2_{\frac{n}{2}}$ we can always multiply $a_1$, $a_2$, or $b_2$ by $\ee_n$. In the case of an even $\frac{n}{2}$, we can assume that $b_2 = \ee_n$. But then a subgroup 
\[\left\langle (\varepsilon^2_{n},1,1),(1,\varepsilon^2_{n},1),(1,1,\varepsilon^2_{n}),\left(
\ee_{2n},\ee_{2n}^{-1},1;\,(1,2,3)\right), \left(
\iota(1),\ee_n,\iota(1)\right)\right\rangle\] is precisely $G_{12\cdot(\nicefrac{n}{2})^3}'$. On the other hand, in the case of an odd $\frac{n}{2}$, we can assume $a_1 = i$, $a_2 = -i$, and~$b_2 = -1$ (since $\ee^{\frac{n}{2}}_{2n} = i$ and $\frac{n}{2}-1$ is even). Hence, a subgroup 
\[\left\langle (\varepsilon^2_{n},1,1),(1,\varepsilon^2_{n},1),(1,1,\varepsilon^2_{n}),\left(
i,-i,1;\,(1,2,3)\right), \left(
\iota(1),-1,\iota(1)\right)\right\rangle\] is $G_{12\cdot(\nicefrac{n}{2})^3}'$.

Groups with small $n$ will require extra care in our analysis, so we list them separately to fix the notation. We denote groups in the lists by the twofold indices in GAP's SmallGroups library, where the first number is the order of the group and the second number is the internal index in the library's list of groups of that order.
\begin{corollary}
\label{groupsA4}
For $n \in \{1,2,3,4\}$ and $W_G = \mathfrak{A}_4$, the group $G$ is conjugate in $\Aut(\mathbb{P}^1 \times \mathbb{P}^1 \times \mathbb{P}^1)$ to one of the following groups: 
\begin{enumerate}
\item ${G_{12,3}}  = \left\langle\left(
    1,1,1;\,(1,2,3)\right), \left(
    \iota(1),1,\iota(1)\right)\right\rangle$.

\item ${G_{12,3}'}  = \left\langle\left(
    i,-i,1;\,(1,2,3)\right), \left(
    \iota(1),-1,\iota(1)\right)\right\rangle$.

\item ${G_{96,229}}  =\left\langle (-1,1,1),(1,-1,1),(1,1,-1),\left(
    1,1,1;\,(1,2,3)\right), \left(
    \iota(1),1,\iota(1)\right)\right\rangle$.

\item ${G_{96,3}}  = \left\langle (-1,1,1),(1,-1,1),(1,1,-1),\left(
    \ee_8,\ee_8^{-1},1;\,(1,2,3)\right), \left(
    \iota(1),i,\iota(1)\right)\right\rangle$.

\item ${G_{48,50}} = \left\langle (-1,-1,1),(1,-1,-1),\left(
1,1,1;\,(1,2,3)\right), \left(\iota(1),1,\iota(1)\right)\right\rangle$.

\item $G_{24,13}  = \left\langle (-1,-1,-1),\left(
    1,1,1;\,(1,2,3)\right), \left(
    \iota(1),1,\iota(1)\right)\right\rangle$.
    
\item $G_{24,13}'  = \left\langle (-1,-1,-1),\left(
    i,-i,1;\,(1,2,3)\right), \left(
    \iota(1),1,\iota(1)\right)\right\rangle$.

\item ${G_{324,160}}  = \left\langle (\varepsilon_3,1,1),(1,\varepsilon_3,1),(1,1,\varepsilon_3),\left(
    1,1,1;\,(1,2,3)\right), \left(
    \iota(1),1,\iota(1)\right)\right\rangle$.

\item ${G_{324,160}'}  = \left\langle (\varepsilon_3,1,1),(1,\varepsilon_3,1),(1,1,\varepsilon_3),\left(
    i,-i,1;\,(1,2,3)\right), \left(
    \iota(1),-1,\iota(1)\right)\right\rangle$.

\item ${G_{768, 1084559}}  = \left\langle (i,1,1),(1,i,1),(1,1,i),\left(
    1,1,1;\,(1,2,3)\right), \left(
    \iota(1),1,\iota(1)\right)\right\rangle$.

\item ${G_{768, 1083476}}  = \left\langle (i,1,1),(1,i,1),(1,1,i),\left(
    \ee_{16},\ee_{16}^{-1},1;\,(1,2,3)\right), \left(
    \iota(1),\ee_8,\iota(1)\right)\right\rangle$.

\item ${G_{384, 5863}} =  \left\langle (i,i,1),(1,i,i),(1,1,-1),\left(
1,1,1;\,(1,2,3)\right), \left(\iota(1),1,\iota(1)\right)\right\rangle$.

\item ${G_{192, 191}}  = \left\langle (i,i,i),(1,-1,1),(1,1,-1),\left(
    1,1,1;\,(1,2,3)\right), \left(
    \iota(1),1,\iota(1)\right)\right\rangle$.

\item ${G_{192, 194}}  = \left\langle (i,i,i),(1,-1,1),(1,1,-1),\left(
    \ee_8,\ee_8^{-1},1;\,(1,2,3)\right), \left(
    \iota(1),1,\iota(1)\right)\right\rangle$.
\end{enumerate}
\end{corollary}

Define a partial order on the subgroups of $\Aut(\PPP)$ in the following way: for two subgroups $G_1$, $G_2$, we say that $G_1 \preceq G_2$ if and only if there is a subgroup in $G_2$ conjugate to $G_1$ in $\Aut(\PPP)$. Here is the Hasse diagram for this partial order for groups in Corollary~\ref{groupsA4} for $n \in \{1,2,3\}$:
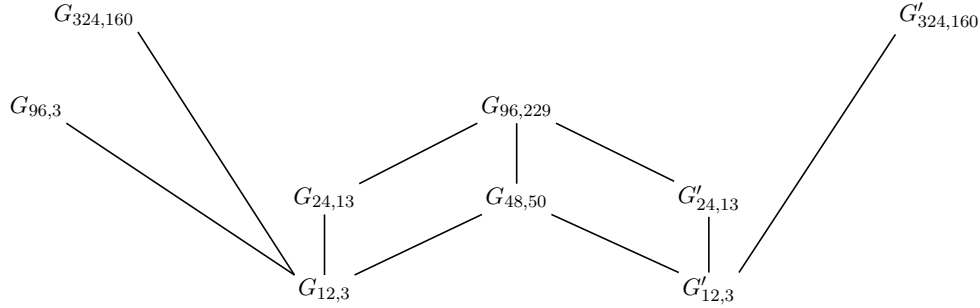
\begin{figure}[H]
    \centering
\begin{tikzpicture}[
    scale=0.82,
    transform shape,
    x=1.55cm,
    y=1.05cm,
    group/.style={
        font=\small,
        inner sep=1.5pt
    },
    hasseedge/.style={
        semithick
    }
]

\node[group] (g123)      at (-2,0)   {$G_{12,3}$};
\node[group] (g123p)     at ( 2,0)   {$G_{12,3}'$};

\node[group] (g2413)     at (-2,1.4) {$G_{24,13}$};
\node[group] (g4850)     at ( 0,1.4) {$G_{48,50}$};
\node[group] (g2413p)    at ( 2,1.4) {$G_{24,13}'$};

\node[group] (g96229)    at (0,2.8)  {$G_{96,229}$};

\node[group] (g963)      at (-5.0,2.8) {$G_{96,3}$};

\node[group] (g324160)   at (-4.4,4.2) {$G_{324,160}$};
\node[group] (g324160p)  at ( 4.4,4.2) {$G_{324,160}'$};

\draw[hasseedge] (g123.north)       -- (g2413.south);
\draw[hasseedge] (g123.north east)  -- (g4850.south west);
\draw[hasseedge] (g2413.north east) -- (g96229.south west);
\draw[hasseedge] (g4850.north)      -- (g96229.south);

\draw[hasseedge] (g123p.north)       -- (g2413p.south);
\draw[hasseedge] (g123p.north west)  -- (g4850.south east);
\draw[hasseedge] (g2413p.north west) -- (g96229.south east);

\draw[hasseedge] (g123.north west)   -- (g324160.south east);
\draw[hasseedge] (g123p.north east)  -- (g324160p.south west);

\draw[hasseedge] (g123.north west)       -- (g963.south east);

\end{tikzpicture}
    \label{hasse}
\end{figure}

We also need groups with other possibilities for $W_G$, but only for $n \in \{1,2\}$. Cases with $n = 4$ and $W_G \neq \mathfrak{A}_4$ are treated uniformly with cases with $n\geq 5$ and do not require further analysis, and cases with $n = 3$ and $W_G \neq \mathfrak{A}_4$ will follow from Lemma~\ref{overgroupsGbirsuperrig}.
\begin{proposition}
\label{groupsS4odd}
For $n \in \{1,2\}$ and $W_G = \widebar{\mathfrak{S}}_4$, the group $G$ is conjugate in $\Aut(\mathbb{P}^1 \times \mathbb{P}^1 \times \mathbb{P}^1)$ to one of the following groups:

\begin{enumerate}
\item ${\widebar{G}_{24,12}}  = \left\langle \left(
    \iota(1),1,\iota(1);\,(2,3)\right),\left(1,1,1;\,(1,2)\right)\right\rangle$.

\item ${\widebar{G}_{24,12}'}  = \left\langle \left(
    \iota(1),1,\iota(1);\,(2,3)\right),\left(i,-i,-1;\,(1,2)\right)\right\rangle$.

\item ${\widebar{G}_{192,1538}} = \left\langle (-1,1,1),(1,-1,1),(1,1,-1),\left(
    \iota(1),1,\iota(1);\,(2,3)\right),\left(
    1,1,1;\,(1,2)\right)\right\rangle$.

\item ${\widebar{G}_{192,1495}}  = \left\langle (-1,1,1),(1,-1,1),(1,1,-1),\left(
    \iota(1),1,\iota(1);\,(2,3)\right),\left(
    \ee_8,\ee_8^{-1},i;\,(1,2)\right)\right\rangle$.

\item ${\widebar{G}_{96,227}}  = \left\langle (-1,-1,1),(1,-1,-1),\left(
    \iota(1),1,\iota(1);\,(2,3)\right),\left(
    1,1,1;\,(1,2)\right)\right\rangle$.

\item ${\widebar{G}_{96,227}'}  = \left\langle (-1,-1,1),(1,-1,-1),\left(
    \iota(1),1,\iota(1);\,(2,3)\right),\left(
    i,-i,-1;\,(1,2)\right)\right\rangle$.

\item ${\widebar{G}_{48,48}}  = \left\langle (-1,-1,-1),\left(
    \iota(1),1,\iota(1);\,(2,3)\right),\left(
    1,1,1;\,(1,2)\right)\right\rangle$.

\item ${\widebar{G}_{48,30}}  = \left\langle (-1,-1,-1),\left(
    \iota(1),1,\iota(1);\,(2,3)\right),\left(
    \ee_8,\ee_8^3,i;\,(1,2)\right)\right\rangle$.
\end{enumerate}
\end{proposition}
\begin{proof}
Let \[A = \left(\iota(a_1),a_2,\iota(a_3);\,(2,3)\right),\ B = \left(b_1,b_2,b_3;\,(1,2)\right)\]
be elements in $G$ such that $\pi_{\mathbb{T}}(A)$ and $\pi_{\mathbb{T}}(B)$ are the chosen generators of $\widebar{\mathfrak{S}}_4$. Conjugating by elements of $\mathbb{T}$, we may assume that $a_1 = a_2 = a_3 =1$. Then, from the presentation, we get: \[B^2 = (b_1b_2,b_1b_2,b_3^2),\ (AB)^3 = \left(\frac{b_2}{b_1b_3},\frac{b_1b_3}{b_2},\frac{b_1b_3}{b_2}\right) \in T_G,\] and $A^4 = (1,1,1)$ automatically. Solving for each $n \in \{1,2\}$ and for each possible $T_G$ and then conjugating by the elements of the normalizer of $\mathbb{T}$, we get the possibilities in the list.
\end{proof}

\begin{proposition}
\label{groupsS4even}
For $n \in \{1,2\}$ and $W_G = \mathfrak{S}_4$, the group $G$ is conjugate in $\Aut(\mathbb{P}^1 \times \mathbb{P}^1 \times \mathbb{P}^1)$ to one of the following groups: 

\begin{enumerate}
\item ${G_{24,12}}  = \left\langle \left(1,\iota(1),1;\,(1,2)\right),\left(
    \iota(1),\iota(1),\iota(1);\,(2,3)\right)\right\rangle$.

\item ${G_{24,12}'}  = \left\langle \left(1,\iota(1),-1;\,(1,2)\right),\left(
    \iota(-1),\iota(1),\iota(1);\,(2,3)\right)\right\rangle$.

\item ${G_{24,12}''}  = \left\langle \left(1,\iota(1),i;\,(1,2)\right),\left(
    \iota(-i),\iota(1),\iota(1);\,(2,3)\right)\right\rangle$.

\item ${G_{192,1538}}  = \left\langle (-1,1,1),(1,-1,1),(1,1,-1),\left(
    1,\iota(1),1;\,(1,2)\right),\left(
    \iota(1),\iota(1),\iota(1);\,(2,3)\right)\right\rangle$.

\item ${G_{192,1538}'}  = \left\langle (-1,1,1),(1,-1,1),(1,1,-1),\left(
    1,\iota(1),i;\,(1,2)\right),\left(
    \iota(-i),\iota(1),\iota(1);\,(2,3)\right)\right\rangle$.

\item ${G_{192,181}} = \left\langle (-1,1,1),(1,-1,1),(1,1,-1),\left(
    1,\iota(1),\ee_8;\,(1,2)\right),\left(
    \iota(\ee_8^{-1}),\iota(1),\iota(1);\,(2,3)\right)\right\rangle$.

\item ${G_{96,227}} = \left\langle (-1,-1,1),(1,-1,-1),\left(
    1,\iota(1),1;\,(1,2)\right),\left(
    \iota(1),\iota(1),\iota(1);\,(2,3)\right)\right\rangle$.

\item ${G_{96,227}'} = \left\langle (-1,-1,1),(1,-1,-1),\left(
    1,\iota(1),i;\,(1,2)\right),\left(
    \iota(-i),\iota(1),\iota(1);\,(2,3)\right)\right\rangle$.

\item ${G_{48,48}} = \left\langle (-1,-1,-1),\left(
    1,\iota(1),1;\,(1,2)\right),\left(
    \iota(1),\iota(1),\iota(1);\,(2,3)\right)\right\rangle$.

\item ${G_{48,48}'} = \left\langle (-1,-1,-1),\left(
    1,\iota(1),i;\,(1,2)\right),\left(
    \iota(-i),\iota(1),\iota(1);\,(2,3)\right)\right\rangle$.
\end{enumerate}
\end{proposition}
\begin{proof}
Let \[A = \left(
a_1,\iota(a_2),a_3;\,(1,2)\right),\ B = \left(\iota(b_1),\iota(b_2),\iota(b_3);\,(2,3)\right)\] be elements in $G$ such that $\pi_{\mathbb{T}}(A)$ and $\pi_{\mathbb{T}}(B)$ are the chosen generators of $\mathfrak{S}_4$. Conjugating by elements of $\mathbb{T}$, we may assume that $a_1 = a_2 = b_3 =1$. Then, from the presentation, we get: \[A^4 = (1,1,a_3^4),\ B^2 = \left(1, b_2,\frac{1}{b_2}\right) ,\ (AB)^3 = \left(\frac{b_2}{a_3 b_1},\frac{b_2}{a_3 b_1},\frac{a_3 b_1}{b_2}\right) \in T_G.\] Solving for each $n \in \{1,2\}$ and for each possible $T_G$ and then conjugating by the elements of the normalizer of $\mathbb{T}$, we get the possibilities in the list.
\end{proof}

\begin{proposition}
\label{groupsA4C2}
For $n \in \{1,2\}$ and $W_G = \mathfrak{A}_4 \times \boldsymbol{\mu}_2$, the group $G$ is conjugate in $\Aut(\mathbb{P}^1 \times \mathbb{P}^1 \times \mathbb{P}^1)$ to one of the following groups: 

\begin{enumerate}
\item ${G_{24,13}''}  = \left\langle \left(1,1,1;\,(1,2,3)\right),\left(
    1,\iota(1),1\right)\right\rangle$.

\item ${G_{24,13}'''}  = \left\langle \left(1,1,1;\,(1,2,3)\right),\left(
    -1,\iota(1),1\right)\right\rangle$.

\item ${G_{24,13}^{IV}}  = \left\langle \left(1,1,1;\,(1,2,3)\right),\left(
    -1,\iota(1),-1\right)\right\rangle$.

\item ${G_{192,1540}}  = \left\langle (-1,1,1),(1,-1,1),(1,1,-1),\left(
    1,1,1;\,(1,2,3)\right),\left(
    1,\iota(1),1\right)\right\rangle$.

\item ${G_{192,197}}  = \left\langle (-1,1,1),(1,-1,1),(1,1,-1),\left(
    1,1,1;\,(1,2,3)\right),\left(
    i,\iota(1),1\right)\right\rangle$.

\item ${G_{192,1000}}  = \left\langle (-1,1,1),(1,-1,1),(1,1,-1),\left(
    1,1,1;\,(1,2,3)\right),\left(
    i,\iota(1),i\right)\right\rangle$.

\item ${G_{96,229}'}  = \left\langle (-1,-1,1),(1,-1,-1),\left(
    1,1,1;\,(1,2,3)\right),\left(
    1,\iota(1),1\right)\right\rangle$.

\item ${G_{96,70}} = \left\langle (-1,-1,1),(1,-1,-1),\left(
    1,1,1;\,(1,2,3)\right),\left(
    i,\iota(1),i\right)\right\rangle$.

\item ${G_{48,49}}  = \left\langle (-1,-1,-1),\left(
    1,1,1;\,(1,2,3)\right),\left(
    1,\iota(1),1\right)\right\rangle$.

\item ${G_{48,49}'}  = \left\langle (-1,-1,-1),\left(
    1,1,1;\,(1,2,3)\right),\left(
    -1,\iota(1),1\right)\right\rangle$.
\end{enumerate}
\end{proposition}
\begin{proof}
Let \[A = \left(
a_1,a_2,a_3;\,(1,2,3)\right),\ B = \left(
b_1,\iota(b_2),b_3\right)\] be elements in $G$ such that $\pi_{\mathbb{T}}(A)$ and $\pi_{\mathbb{T}}(B)$ are the chosen generators of $\mathfrak{A}_4 \times \boldsymbol{\mu}_2$. Conjugating by elements of $\mathbb{T}$, we may assume that $a_1 = a_3 = b_2 =1$. Then, from the presentation, we get: 
\[A^3 = \left(a_2,a_2,a_2\right),\ B^2 = \left(b_1^2,1,b_3^2\right),\ [A,(AB)^3]=\left(\frac{1}{b_3^2},\frac{1}{b_1^2},(a_2 b_1 b_3)^2\right) \in T_G,\] and $(AB)^6 = (1,1,1)$ automatically. Solving for each $n \in \{1,2\}$ and for each possible $T_G$ and then conjugating by the elements of the normalizer of $\mathbb{T}$, we get the possibilities in the list.
\end{proof}

\begin{proposition}
\label{groupsS4C2}
For $n \in \{1,2\}$ and $W_G = \mathfrak{S}_4 \times \boldsymbol{\mu}_2$, the group $G$ is conjugate in $\Aut(\mathbb{P}^1 \times \mathbb{P}^1 \times \mathbb{P}^1)$ to one of the following groups: 

\begin{enumerate}
\item ${G_{48,48}''}  = \left\langle \left(\iota(1),\iota(1),\iota(1);\,(1,2)\right),\left(
    \iota(1),1,\iota(1);\,(2,3)\right)\right\rangle$.

\item ${G_{48,48}'''}  = \left\langle \left(\iota(1),\iota(1),\iota(-1);\,(1,2)\right),\left(
    \iota(1),1,\iota(1);\,(2,3)\right)\right\rangle$.

\item ${G_{48,48}^{IV}}  = \left\langle \left(\iota(i),\iota(i),\iota(1);\,(1,2)\right),\left(
    \iota(1),1,\iota(1);\,(2,3)\right)\right\rangle$.

\item ${G_{48,48}^{V}}  = \left\langle \left(\iota(i),\iota(i),\iota(-1);\,(1,2)\right),\left(
    \iota(1),1,\iota(1);\,(2,3)\right)\right\rangle$.

\item ${G_{384,20163}}  = \left\langle (-1,1,1),(1,-1,1),(1,1,-1),\left(
    \iota(1),\iota(1),\iota(1);\,(1,2)\right),\left(
    \iota(1),1,\iota(1);\,(2,3)\right)\right\rangle$.

\item ${G_{384,17948}}  = \left\langle (-1,1,1),(1,-1,1),(1,1,-1),\left(
    \iota(1),\iota(1),\iota(i);\,(1,2)\right),\left(
    \iota(1),1,\iota(1);\,(2,3)\right)\right\rangle$.

\item ${G_{384,5657}}  = \left\langle (-1,1,1),(1,-1,1),(1,1,-1),\left(
    \iota(\ee_8),\iota(\ee_8),\iota(1);\,(1,2)\right),\left(
    \iota(1),1,\iota(1);\,(2,3)\right)\right\rangle$.

\item ${G_{384,20100}}  = \left\langle (-1,1,1),(1,-1,1),(1,1,-1),\left(
    \iota(\ee_8),\iota(\ee_8),\iota(i);\,(1,2)\right),\left(
    \iota(1),1,\iota(1);\,(2,3)\right)\right\rangle$.

\item ${G_{192,1538}''}  = \left\langle (-1,-1,1),(1,-1,-1),\left(
    \iota(1),\iota(1),\iota(1);\,(1,2)\right),\left(
    \iota(1),1,\iota(1);\,(2,3)\right)\right\rangle$.

\item ${G_{192,1538}'''}  = \left\langle (-1,-1,1),(1,-1,-1),\left(
    \iota(i),\iota(i),\iota(1);\,(1,2)\right),\left(
    \iota(1),1,\iota(1);\,(2,3)\right)\right\rangle$.

\item ${G_{192,955}}  = \left\langle (-1,-1,1),(1,-1,-1),\left(
    \iota(1),\iota(1),\iota(i);\,(1,2)\right),\left(
    \iota(1),1,\iota(1);\,(2,3)\right)\right\rangle$.

\item ${G_{192,955}'}  = \left\langle (-1,-1,1),(1,-1,-1),\left(
    \iota(i),\iota(i),\iota(i);\,(1,2)\right),\left(
    \iota(1),1,\iota(1);\,(2,3)\right)\right\rangle$.

\item ${G_{96,226}}  = \left\langle (-1,-1,-1),\left(
    \iota(1),\iota(1),\iota(1);\,(1,2)\right),\left(
    \iota(1),1,\iota(1);\,(2,3)\right)\right\rangle$.

\item ${G_{96,195}}  = \left\langle (-1,-1,-1),\left(
    \iota(\ee_8),\iota(\ee_8),\iota(i);\,(1,2)\right),\left(
    \iota(1),1,\iota(1);\,(2,3)\right)\right\rangle$.
\end{enumerate}
\end{proposition}
\begin{proof}
Let \[A = \left(
\iota(a_1),\iota(a_2),\iota(a_3);\, (1,2)\right),\ B = \left(
\iota(b_1),b_2,\iota(b_3);\,(2,3)\right)\] be elements in $G$ such that $\pi_{\mathbb{T}}(A)$ and $\pi_{\mathbb{T}}(B)$ are the chosen generators of $\mathfrak{S}_4 \times \boldsymbol{\mu}_2$. Conjugating by elements of $\mathbb{T}$, we may assume that $b_1 = b_2 = b_3 =1$. Then, from the presentation, we get: 
\[A^2 = \left(\frac{a_1}{a_2},\frac{a_2}{a_1},1\right),\ [B,(BA)^3] =\left(\left(\frac{a_1}{a_2a_3}\right)^2,\left(\frac{1}{a_1a_2}\right)^2,\left(\frac{1}{a_3}\right)^2\right) \in T_G,\] and $B^4 = (AB)^6 = (1,1,1)$ automatically. Solving for each $n \in \{1,2\}$ and for each possible $T_G$ and then conjugating by the elements of the normalizer of $\mathbb{T}$, we get the possibilities in the list.
\end{proof}

Note that any finite subgroup $G$ of $\mathbb{T}\rtimes \mathbb{W}$ with $W_G \supset\mathfrak{A}_4$ has a subgroup in $\mathbb{T}\rtimes \mathfrak{A}_4$ with equal toric parts and such that its projection is surjective onto $\mathfrak{A}_4$. More precisely, the following is true:

\begin{lemma}
\label{A4sub}
Let $G$ be a finite toric group with $W_G \supset \mathfrak{A}_4$. Then there exists a subgroup $H$ in $G$, such that $T_H \cong T_G$ and $W_H \cong \mathfrak{A}_4$.
\end{lemma}
\begin{proof}
Note, that there is a unique subgroup isomorphic to $\mathfrak{A}_4$ in $W_G$, since there is a unique subgroup isomorphic to $\mathfrak{A}_4$ in $\mathbb{W}$. Take its preimage in $G$ and denote it $H$. Then $H$ is exactly the subgroup in question.
\end{proof}

\subsection{Orbit descriptions}
Here we list the orbit decompositions for the relevant groups. The computation is uniform in all cases: we first treat the open torus orbit, then deal separately with the toric boundary. Concretely, we proceed as follows:
\begin{enumerate}
    \item Let
    \[
        U = X \setminus \boldsymbol{\mathcal{D}}_{\mathbb{T}}.
    \]
    The action of $\mathbb{T}$ on $U$ is free, and hence the action of $T_G$ on $U$ is free as well.
    
    \item We pass to the quotient $U/T_G$, which is described explicitly by invariant monomials.
    
    \item The quotient group $G/T_G \cong W_G$ acts naturally on $U/T_G$. We compute this action explicitly: for each element of $W_G$, we compute its fixed locus, and then use the intersections of these loci to determine the $W_G$-orbit decomposition of $U/T_G$.
    
    \item We then lift this decomposition to obtain the $G$-orbit decomposition of $X \setminus \boldsymbol{\mathcal{D}}_{\mathbb{T}}.$ It remains to analyze the toric boundary $\boldsymbol{\mathcal{D}}_{\mathbb{T}}$.
    
    \item Since $G$ acts transitively on the irreducible components of $\boldsymbol{\mathcal{D}}_{\mathbb{T}}$, it suffices to choose one component $S$ and compute the $\Stab_G(S)$-orbit decomposition on $S$. The $\Stab_G(S)$-orbits on $S$ are then computed explicitly, and their $G$-translates give the full $G$-orbit decomposition of $\boldsymbol{\mathcal{D}}_{\mathbb{T}}$.
\end{enumerate}
Combining the decompositions on $X \setminus \boldsymbol{\mathcal{D}}_{\mathbb{T}}$ and on $\boldsymbol{\mathcal{D}}_{\mathbb{T}}$ gives the required $G$-orbit decomposition of $X$.

First, let us list the orbit decompositions for all groups with $W_G = \mathfrak{A}_4$ and $T_G \cong \boldsymbol{\mu}^3_n$, as well as for the small groups relevant to our analysis.
\begin{proposition}
\label{orbitdescription}
For the following groups $G$, the orbit decompositions are as follows, where $G_{12n^3}$ and $G_{12n^3}'$ refer to both even and odd cases of $n$:
\begingroup\footnotesize
\begin{longtblr}[
    theme = mytheme,
    entry = none,
    label = none,
]{
    stretch = 1.25,
    colspec = {|Q[c,m,co=0.8]|Q[c,m,co=0.6]|Q[l,m,co=6.6]|},
    rowhead = 1,
}
    \hline
    Group & Orbit length & Orbits\\
    \hline
    \hline
    $G_{12n^3}$ & $4$ & {$G\cdot\left([0:1],[0:1],[0:1]\right)$, $G\cdot\left([1:0],[1:0],[1:0]\right)$}\\
    \cline{2-3}
    & $12n$ & $G\cdot\left([0:1],p,[0:1]\right)$, where $p\neq[0:1],[1:0]$\\
    \cline{2-3}
    & $6n^2$ & {$G\cdot\left([1:1],[0:1],[1:1]\right)$, $G\cdot\left([\varepsilon_{2n}:1],[0:1],[1:1]\right)$\\ $G\cdot\left([1:1],[0:1],[\varepsilon_{2n}:1]\right)$, $G\cdot\left([\varepsilon_{2n}:1],[0:1],[\varepsilon_{2n}:1]\right)$}\\
    \cline{2-3}
    & $12n^2$ & {$G\cdot\left(p,[0:1],q\right)$, where $p,q\neq[0:1],[1:0]$,\\ and at least one of $p,q$ is not equal to $[\varepsilon_{2n}^j:1]$ for some $j\in\mathbb{Z}$}\\
    \cline{2-3}
    & $n^3$ & {$G\cdot\left([1:1],[1:1],[1:1]\right)$, $G\cdot\left([\varepsilon_{2n}:1],[\varepsilon_{2n}:1],[\varepsilon_{2n}:1]\right)$}\\
    \cline{2-3}
    & $3n^3$ & {$G\cdot\left([1:1],[1:1],[\varepsilon_{2n}:1]\right)$, $G\cdot\left([1:1],[\varepsilon_{2n}:1],[\varepsilon_{2n}:1]\right)$}\\
    \cline{2-3}
    & $4n^3$ & {$G\cdot\left(p,p,p\right)$, where $p\neq[\varepsilon_{2n}^j:1],[0:1],[1:0]$ for some $j\in\mathbb{Z}$}\\
    \cline{2-3}
    & $6n^3$ & {$G\cdot\left([1:1],p,[1:1]\right)$, $G\cdot\left([\varepsilon_{2n}:1],p,[1:1]\right)$\\ $G\cdot\left([1:1],p,[\varepsilon_{2n}:1]\right)$, $G\cdot\left([\varepsilon_{2n}:1],p,[\varepsilon_{2n}:1]\right)$\\ where $p\neq[\varepsilon_{2n}^j:1],[0:1],[1:0]$ for some $j\in\mathbb{Z}$}\\
    \cline{2-3}
    & $12n^3$ & All other orbits\\
    \hline

    $G_{12n^3}'$ & $4$ & {$G\cdot\left([0:1],[0:1],[0:1]\right)$, $G\cdot\left([1:0],[1:0],[1:0]\right)$}\\
    \cline{2-3}
    & $12n$ & $G\cdot\left([0:1],p,[0:1]\right)$, where $p\neq[0:1],[1:0]$\\
    \cline{2-3}
    & $6n^2$ & {$G\cdot\left([1:1],[0:1],[1:1]\right)$, $G\cdot\left([\varepsilon_{2n}:1],[0:1],[1:1]\right)$\\ $G\cdot\left([1:1],[0:1],[\varepsilon_{2n}:1]\right)$, $G\cdot\left([\varepsilon_{2n}:1],[0:1],[\varepsilon_{2n}:1]\right)$}\\
    \cline{2-3}
    & $12n^2$ & {$G\cdot\left(p,[0:1],q\right)$, where $p,q\neq[0:1],[1:0]$,\\ and at least one of $p,q$ is not equal to $[\varepsilon_{2n}^j:1]$ for some $j\in\mathbb{Z}$}\\
    \cline{2-3}
    & $4n^3$ & {$G\cdot\left([\theta x:y],[x:y],[x:y]\right)$, where $\theta=\ee_{4n}$ for even $n$\\ and $\theta=i$ for odd $n$, and $[x:y]\neq[0:1],[1:0]$}\\
    \cline{2-3}
    & $12n^3$ & All other orbits\\
    \hline

    $G_{96,3}$ & $4$ & {$G\cdot\left([0:1],[0:1],[0:1]\right)$, $G\cdot\left([1:0],[1:0],[1:0]\right)$}\\
    \cline{2-3}
    & $24$ & {$G\cdot\left([1:1],[0:1],[1:1]\right)$, $G\cdot\left([i:1],[0:1],[1:1]\right)$\\ $G\cdot\left([1:1],[0:1],[i:1]\right)$, $G\cdot\left([i:1],[0:1],[i:1]\right)$\\ $G\cdot\left([0:1],p,[0:1]\right)$, where $p\neq[0:1],[1:0]$}\\
    \cline{2-3}
    & $32$ & {$G\cdot\left([\varepsilon_8x:y],[x:y],[x:y]\right)$, where $[x:y]\neq[0:1],[1:0]$}\\
    \cline{2-3}
    & $48$ & {$G\cdot\left(p,[0:1],q\right)$, where $p,q\neq[0:1],[1:0]$,\\ and at least one of $p,q$ is not equal to $[i^j:1]$ for some $j\in\mathbb{Z}$}\\
    \cline{2-3}
    & $96$ & All other orbits\\
    \hline

    $G_{48,50}$ & $4$ & {$G\cdot\left([0:1],[0:1],[0:1]\right)$, $G\cdot\left([1:0],[1:0],[1:0]\right)$\\ $G\cdot\left([1:1],[1:1],[1:1]\right)$, $G\cdot\left([1:1],[1:1],[-1:1]\right)$\\ $G\cdot\left([i:1],[i:1],[i:1]\right)$, $G\cdot\left([i:1],[i:1],[-i:1]\right)$}\\
    \cline{2-3}
    & $16$ & {$G\cdot\left(p,p,p\right)$, where $p\neq[0:1],[1:0],[i^j:1]$ for some $j\in\mathbb{Z}$}\\
    \cline{2-3}
    & $24$ & {$G\cdot\left([1:1],p,[1:1]\right)$, where $p\neq[\pm1:1]$\\ $G\cdot\left([i:1],p,[i:1]\right)$, where $p\neq[\pm i:1]$\\ $G\cdot\left([i:1],[0:1],[1:1]\right)$, $G\cdot\left([1:1],[0:1],[i:1]\right)$\\ $G\cdot\left([0:1],p,[0:1]\right)$, where $p\neq[0:1],[1:0]$}\\
    \cline{2-3}
    & $48$ & All other orbits\\
    \hline
\end{longtblr}
\endgroup
\end{proposition}

For the group $G_{48,49}'$ we require a more detailed orbit decomposition.
\begin{proposition}
\label{orbits4849}
For $G=G_{48,49}'$, the non-generic $G$-orbits are as follows:
\begingroup\footnotesize
\begin{longtblr}[
    theme = mytheme,
    entry = none,
    label = none,
]{
    stretch = 1.25,
    colspec = {|Q[c,m,co=0.8]|Q[l,m,co=3.6]|Q[l,m,co=3.6]|},
    rowhead = 1,
}
    \hline
    Group & Family or isolated orbit & Exceptional members\\
    \hline
    \hline
    $G_{48,49}'$ & {$\Gamma_{16}=G\cdot\left([x:y],[x:y],[x:y]\right)$} & {$\mathcal{O}_{8,1}=G\cdot\left([1:0],[1:0],[1:0]\right)$\\ $\mathcal{O}_{8,2}=G\cdot\left([1:1],[1:1],[1:1]\right)$\\ $\mathcal{O}_{8,3}=G\cdot\left([i:1],[i:1],[i:1]\right)$}\\
    \cline{2-3}
    & {$\Gamma_{24,1}=G\cdot\left([0:1],[-1:1],[x:y]\right)$\\ $\Gamma_{24,2}=G\cdot\left([x:y],[-i:1],[1:0]\right)$\\ $\Gamma_{24,3}=G\cdot\left([-i:1],[x:y],[-1:1]\right)$} & {$\mathcal{O}_{12,+}=G\cdot\left([0:1],[-1:1],[-i:1]\right)$\\ $\mathcal{O}_{12,-}=G\cdot\left([0:1],[-1:1],[i:1]\right)$}\\
    \cline{2-3}
    & {$\mathcal{O}_{24}=G\cdot\left([-1:1],[1:0],[-i:1]\right)$} & ---\\
    \hline
\end{longtblr}
\endgroup
\noindent
All other $G$-orbits have length $48$.
\end{proposition}

\section{Curves}
\label{Curves}
To determine $G$-birational behaviour of $\PPP$ in each case, it is important to understand which $G$-invariant curves of \say{small} tri-degrees are possible. For every group $G$ considered below, we have $W_G \supset \mathfrak{A}_4$, so $G$ permutes the three rulings of $\PPP$ transitively. Hence the tri-degree of any $G$-invariant curve is symmetric, i.e. of the form $(a,a,a)$ for some $a$. Here by \say{small} we mean curves with $a\leq 7$. The reason for this restriction is that only those curves can appear as $G$-invariant centers of non-canonical singularities of log-pairs in Lemma~\ref{mobilesystems}. Note that we can restrict ourselves to $G$-irreducible $G$-invariant curves. Also note that we always have a $G$-invariant $G$-irreducible curve $\boldsymbol{\mathcal{C}}_{\mathbb{T}}$ of tri-degree $(4,4,4)$. In further considerations we will omit it from our lists of curves and assume that our curves are distinct from $\boldsymbol{\mathcal{C}}_{\mathbb{T}}$, unless stated otherwise.
\subsection{Invariant curves for sufficiently large groups}
Here we prove that for groups with $n\geq 4$ and $W_G \supset \mathfrak{A}_4$, the only $G$-invariant curve of tri-degree $(a,a,a)$ with $a \leq 7$ is $\boldsymbol{\mathcal{C}}_{\mathbb{T}}$.
\begin{proposition}
\label{oddcurves}
For an odd $n \geq 5$ and $W_G = \mathfrak{A}_4$ there are no $G$-invariant $G$-irreducible curves of tri-degree $(a,a,a)$ with $a \leq 7$, other than $\boldsymbol{\mathcal{C}}_{\mathbb{T}}$.
\end{proposition}
\begin{proof}
For each $n$, both groups $G_{12n^3}$ and $G_{12n^3}'$ have a positive-dimensional base-point-free linear system $U$ of $G$-invariant divisors of tri-degree $(d, d, d)$, where $d = 2n$ for $G_{12n^3}$, and $d = 4n$ for $G_{12n^3}'$. Indeed, a direct computation gives a positive-dimensional base-point-free linear system $U_1$ of $G_{12,3}$-invariant $(2,2,2)$-divisors such that \[x_1y_1x_2y_2x_3y_3 \in U_1.\] Similarly, there is a positive-dimensional base-point-free linear system $U_1'$ of $G_{12,3}'$-invariant $(4,4,4)$-divisors such that \[(x_1y_1x_2y_2x_3y_3)^2 \in U_1'.\] For arbitrary $n$, define $U_n$ and $U_n'$ by applying the substitution \[x_i \mapsto x_i^n,\ y_i \mapsto y_i^n\] to the elements of $U_1$ and $U_1'$, respectively. Then $U_n$ and $U_n'$ are positive-dimensional base-point-free linear systems of $G_{12n^3}$-invariant $(2n,2n,2n)$-divisors and $G_{12n^3}'$-invariant $(4n,4n,4n)$-divisors, respectively.

Now, if $C$ were the required curve, it would either intersect the general member of the system $U$ in $3ad \leq 84n$ points or lie in the locus of degenerate orbits of $\PPP$. Since $n\geq5$, $12n^3\geq 84n$. Hence $C$ must lie in the locus of degenerate orbits. From the description of the orbits, it follows that the only possibility for $C$ is to lie in $\boldsymbol{\mathcal{D}}_{\mathbb{T}}$. For $G_{12n^3}$ the number of intersection points is no greater than $42n$, but the length of the general orbits in $\boldsymbol{\mathcal{D}}_{\mathbb{T}}$ is $12n^2$. Since for $n\geq5$ we have $12n^2 \geq 42n$, this curve cannot exist. As the actions of $G_{12n^3}$ and $G_{12n^3}'$ on $\mathbb{P}^1 \times \{[0:1]\} \times \mathbb{P}^1$ coincide, this curve does not exist for $G_{12n^3}'$ either.
\end{proof}

\begin{proposition}
For an even $n \geq 4$ and $W_G = \mathfrak{A}_4$ there are no $G$-invariant $G$-irreducible curves of tri-degree $(a,a,a)$ with $a \leq 7$, other than $\boldsymbol{\mathcal{C}}_{\mathbb{T}}$.
\end{proposition}
\begin{proof} Let $C$ be a $G$-invariant $G$-irreducible curve of tri-degree $(a,a,a)$ with $a\leq7$, other than $\boldsymbol{\mathcal{C}}_{\mathbb{T}}$. Consider cases:
\begin{enumerate}
    \item Let us first consider the case of $G_{12n^3}$ and $G_{12n^3}'$. Then the proof is essentially the same as in the case of an odd $n$.

    \item Then let us consider the case of $n=4$. For each group it is possible to find a positive-dimensional base-point-free linear system $U$ of $G$-invariant divisors of tri-degree $(d,d,d)$, where $d = 16$ for $G_{768,1083476}$, and $d = 8$ otherwise. If $C$ does not lie in the locus of degenerate orbits, then either $C$ intersects the general member $D$ of the system $U$ in points with trivial stabilizers or $C \subset D$. In the former case, $|C \cap D| \geq |G|$. But from the degree consideration, we get that $C\cdot D = 3ad \leq 21d$. This is a contradiction, since in every case $21d<|G|$. The only other option is for $C$ to lie in $D$, but since $D$ is general and $U$ is base-point-free, this is impossible. Hence $C$ lies in the locus of degenerate orbits of $\PPP$. This is also impossible by direct computations.
    
    \item Now let us consider the case of $n = 6$. Recall from Section~\ref{Groups} that each of the three remaining groups contains, up to conjugacy, either $G_{324,160}$ or $G_{324,160}'$ as a subgroup. Any $G$-invariant curve must also be invariant under such a subgroup. We can list all $G$-invariant curves that satisfy our constraints for $G = G_{324,160}$ and $G_{324,160}'$. Indeed, by the same argument as in Proposition~\ref{oddcurves} they must lie in $\boldsymbol{\mathcal{D}}_{\mathbb{T}}$. Such curves are computed directly; for both group, there are four curves of tri-degree $(6,6,6)$ of the following form:
    \begin{align*}
        &G \cdot \left(\mathbb{P}^1 \times \{[0:1]\} \times \{[1:1]\}\right),\\ 
        &G \cdot \left(\mathbb{P}^1 \times \{[0:1]\} \times \{[-1:1]\}\right),\\
        &G \cdot \left(\{[1:1]\} \times \{[0:1]\} \times \mathbb{P}^1\right),\\ 
        &G \cdot \left(\{[-1:1]\} \times \{[0:1]\} \times \mathbb{P}^1\right).
    \end{align*}
    Since in both cases the only possibility for $a$ is $6$, for the $G$-orbits of those curves to be of the required tri-degree, those curves must already be $G$-invariant. But they clearly are not. Hence there are no $G$-invariant curves for $n = 6$.
    
    \item Now, for $n\geq 8$, we always have $G_{12\cdot (\nicefrac{n}{2})^3}$ or $G_{12\cdot (\nicefrac{n}{2})^3}'$ as a subgroup in $G = G_{6n^3}$, $G_{3n^3}$, or $G_{3n^3}'$. Since $\frac{n}{2} \geq 4$, and since for such $n$ we know that there are no curves invariant under the action of the subgroup and satisfying our conditions, we conclude that there are no required $G$-invariant curves.
    \end{enumerate}
\end{proof}

Applying Lemma~\ref{A4sub}, we have the following:
\begin{corollary}
\label{nocurves}
For any finite subgroup $G$ of $\mathbb{T}\rtimes \mathbb{W}$ with the toric part $T_G \cong \boldsymbol{\mu}^3_n$, $\boldsymbol{\mu}^2_n \times \boldsymbol{\mu}_{\frac{n}{2}}$ or $\boldsymbol{\mu}_n \times \boldsymbol{\mu}^2_{\frac{n}{2}}$ with $n\geq 4$, there are no $G$-invariant $G$-irreducible curves of tri-degree $(a,a,a)$ with $a \leq 7$, other than $\boldsymbol{\mathcal{C}}_{\mathbb{T}}$.
\end{corollary}

\subsection{Invariant curves for small groups}
In this subsection, we compute the $G$-invariant $G$-irreducible curves for the small groups that require separate analysis. The relevant groups are
\begin{equation}
\label{eq:small-superrigidity-groups}
\left\{G_{96,3},\ G_{48,50},\ G_{324,160},\ G_{324,160}',\ G_{48,49}'\right\}.
\end{equation} The role of the list~\eqref{eq:small-superrigidity-groups} is as follows. The groups in this list are the remaining minimal groups, with respect to containment up to conjugacy in $\Aut(\PPP)$, for which the $G$-birational superrigidity argument has to be carried out. More precisely, if $G$ is a group with $n<4$ and $W_G\supset \mathfrak{A}_4$ for which $\PPP$ is to be proved $G$-birationally superrigid, then $G$ contains, up to conjugacy, one of the groups in the list above. After proving $G$-birational superrigidity for these minimal groups, we will apply Lemma~\ref{overgroupsGbirsuperrig} to pass to their overgroups.

We begin with the following lemma on curves and divisors containing them, which follows easily from the ideal exact sequence.
\begin{lemma}
    \label{222sections} 
    Let $C$ be an irreducible curve of tri-degree $(k,l,m)$ and consider the space of divisors of tri-degree $(a,b,c)$. Suppose that \[N = (a+1)(b+1)(c+1)-(ak+bl+cm)-1 >0.\] Then the linear system of divisors of tri-degree $(a,b,c)$ containing $C$ has projective dimension at least $N - 1$. In particular, the following are true:
    \begin{enumerate}
        \item An irreducible curve of tri-degree $(2,2,2)$ lies in a divisor of tri-degree $(1,1,1)$.
        \item An irreducible curve of tri-degree $(4,4,4)$ lies in a pencil of divisors of tri-degree $(2,2,2)$.
    \end{enumerate}
\end{lemma}

Next, we list several types of curves that are easy to compute. In the subsequent propositions, we will not mention their computations and just list them.
\begin{enumerate}
    \item Curves that decompose as a union of curves of tri-degrees $(m,0,0)$, $(0,m,0)$, and~$(0,0,m)$ for some positive $m$. A curve of tri-degree $(m,0,0)$ is a union of $m$ fibers of $\pr_{2,3}$. Thus, the $F_1$-invariant $F_1$-irreducible curves of this tri-degree correspond to $F_1$-orbits of length $m$ on the base $\PP\times\PP$. Taking the full $G$-orbit gives the corresponding curve of tri-degree $(m,m,m)$.
    
    \item Curves whose decomposition contains a component $C_0$ of tri-degree $(0,a,b)$ for some positive $a,b$; note that this case includes the curves that lie in $\boldsymbol{\mathcal{D}}_{\mathbb{T}}$. Put \[F_{C_0}=\Stab_G(C_0). \]Such a component is supported in fibers of $\pr_1$. To enumerate them, choose an $F_{C_0}$-orbit $O$ of length $m$ on the base of $\pr_1$, with $m\mid a$ and $m\mid b$, and a point $p\in O$. On the fiber \[S_p=\pr_1^{-1}(p)\cong\PP\times\PP,\] compute the $\Stab_{F_{C_0}}(S_p)$-invariant irreducible curves of bi-degree $\left(\frac{a}{m},\frac{b}{m}\right)$ from the invariant one-dimensional subspaces of \[H^0\left(S_p,\mathcal{O}_{S_p}\left(\frac{a}{m},\frac{b}{m}\right)\right).\] Taking the $F_{C_0}$-translates and then the full $G$-orbit gives all curves of this type.

    \item Irreducible curves of tri-degree $(1,1,1)$. These are of the form \[\{\left([x:y],\varphi([x:y]),\psi([x:y])\right)\mid [x:y] \in \mathbb{P}^1\},\] for some $\varphi,\psi\in \PGL_2$. To find the $G$-invariant curves among them, we write a curve in the normalized form above and impose that every generator of $G$ sends it to itself. This gives explicit equations on $\varphi$ and $\psi$, whose solutions give all $G$-invariant irreducible curves of tri-degree $(1,1,1)$.

    \item Irreducible curves of tri-degree $(2,2,2)$. By Lemma~\ref{222sections}, they must lie in a $(1,1,1)$-divisor. Thus, to compute such curves, we decompose $H^0(\mathcal{O}(1,1,1))$ into irreducible representations and examine the corresponding base loci. 
\end{enumerate}
\begin{lemma}
\label{rulingsstab}
Let $G$ be one of the groups from list~\eqref{eq:small-superrigidity-groups}. Let $F$ be the stabilizer of the rulings, and let $K_i$ be the subgroup of $F$ that stabilizes a general divisor in $|H_i|$. Let $C$ be a $G$-irreducible $G$-invariant curve. Then there is the following alternative:
\begin{enumerate}
    \item For any $i = 1,2,3$ and for each $F$-irreducible component $C'$ of $C$, $|K_i|$ divides $C'\cdot H_i$;
    \item $F$-components of $C$ are contained in the locus of degenerate $K_i$-orbits for some $i$, possibly depending on the component.
\end{enumerate}
Moreover,
\[
|K_1|=|K_2|=|K_3|.
\]
\end{lemma}
\begin{proof}
    Consider a general divisor $S$ of tri-degree $(1,0,0)$. Then, for each group in \eqref{eq:small-superrigidity-groups}, $S$ is stabilized by some subgroup $K_1 \subset F$, and $K_1$ acts effectively on $S$. Hence, a general $K_1$-orbit on $S$ has length $|K_1|$. Intersecting $S$ with $C$ and applying the same argument to general divisors of tri-degrees $(0,1,0)$ and $(0,0,1)$, we get the required alternative. Finally, note that \[|K_1| = |K_2| = |K_3|\] by transitivity of the action on the rulings.
\end{proof}
Let us now compute $G$-invariant curves in each required case. In the tables, \say{$12$ rulings} means four curves of each of the tri-degrees $(1,0,0)$, $(0,1,0)$, and $(0,0,1)$; \say{$r\times(a,b,c)$} means $r$ irreducible components of tri-degree $(a,b,c)$; and \say{$(a,b,c)$} means that the curve is irreducible of tri-degree $(a,b,c)$.
\begin{proposition}
\label{Gcurves963}
Let $G=G_{96,3}$. The complete list of $G$-invariant $G$-irreducible curves of tri-degree $(a,a,a)$ with $a\leq 7$, other than $\boldsymbol{\mathcal{C}}_{\mathbb{T}}$, is as follows:
\begingroup\footnotesize
\begin{longtblr}[
theme = mytheme,
entry = none,
label = none,
]{
stretch = 1.25,
colspec = {|Q[c,m,co=0.3]|Q[c,m,co=0.5]|Q[c,m,co=1]|Q[l,m,co=5.8]|},
rowhead = 1
}
\hline
$a$ & Curve & Components & Representative component\\
\hline
\hline
$4$ & $\boldsymbol{\mathcal{F}}_1$ & $12$ rulings & $\mathbb{P}^1\times\{[0:1]\}\times\{[1:1]\}$\\
\cline{2-4}
& $\boldsymbol{\mathcal{F}}_2$ & $12$ rulings & $\mathbb{P}^1\times\{[0:1]\}\times\{[i:1]\}$\\
\cline{2-4}
& $\boldsymbol{\mathcal{F}}_3$ & $12$ rulings & $\{[1:1]\}\times\{[0:1]\}\times\mathbb{P}^1$\\
\cline{2-4}
& $\boldsymbol{\mathcal{F}}_4$ & $12$ rulings & $\{[i:1]\}\times\{[0:1]\}\times\mathbb{P}^1$\\
\cline{2-4}
& $\boldsymbol{\mathcal{S}}_1$ & $(4,4,4)$ & {$\displaystyle\begin{cases}-x_1^2x_3^2-ix_1^2y_3^2+iy_1^2x_3^2+y_1^2y_3^2=0,\\ -ix_2^2x_3^2+ix_2^2y_3^2+x_3^2y_2^2+y_2^2y_3^2=0.\end{cases}$}\\
\cline{2-4}
& $\boldsymbol{\mathcal{S}}_2$ & $(4,4,4)$ & {$\displaystyle\begin{cases}-x_1^2x_3^2+ix_1^2y_3^2-iy_1^2x_3^2+y_1^2y_3^2=0,\\ -ix_2^2x_3^2-ix_2^2y_3^2-x_3^2y_2^2+y_2^2y_3^2=0.\end{cases}$}\\
\cline{2-4}
& $\boldsymbol{\mathcal{S}}_3$ & $(4,4,4)$ & {$\displaystyle\begin{cases}x_1^2x_3^2-ix_1^2y_3^2-iy_1^2x_3^2+y_1^2y_3^2=0,\\ ix_2^2x_3^2-ix_2^2y_3^2+x_3^2y_2^2+y_2^2y_3^2=0.\end{cases}$}\\
\cline{2-4}
& $\boldsymbol{\mathcal{S}}_4$ & $(4,4,4)$ & {$\displaystyle\begin{cases}x_1^2x_3^2+ix_1^2y_3^2+iy_1^2x_3^2+y_1^2y_3^2=0,\\ ix_2^2x_3^2+ix_2^2y_3^2-x_3^2y_2^2+y_2^2y_3^2=0.\end{cases}$}\\
\hline
\end{longtblr}
\endgroup
\end{proposition}
\begin{proof}
Let $C$ be a $G$-invariant $G$-irreducible curve of tri-degree $(a,a,a)$, $a\leq 7$, and $C \neq \boldsymbol{\mathcal{C}}_{\mathbb{T}}$. Assume that $C$ does not lie in either the locus of degenerate $G$-orbits or $\boldsymbol{\mathcal{D}}_{\mathbb{T}}$; the ones that do are the curves $\boldsymbol{\mathcal{F}}_i$ for $i \in \{1,2,3,4\}$. First, apply Lemma~\ref{rulingsstab}. We compute that $|K_i| = 4$; hence, either $C$ is contained in the union of degenerate loci of $K_i$ or every $F$-component of $C$ has tri-degree divisible by 4. The curves contained in these degenerate loci are precisely the curves $\boldsymbol{\mathcal{F}}_{i}$. Thus, we can assume that $C$ is not one of these curves and satisfies the divisibility condition. We get that the tri-degree of $C$ is $(4,4,4)$, and there are the following possibilities for the decomposition of $C$ into irreducible components:
\begin{enumerate}
    \item an irreducible $(4,4,4)$-curve;
    \item two irreducible $(2,2,2)$-curves: impossible since there are no index 2 subgroups in $G_{96,3}$;
    \item four irreducible $(1,1,1)$-curves: none;
    \item four $(1,0,0)$-, four $(0,1,0)$-, and four $(0,0,1)$-curves: $\boldsymbol{\mathcal{F}}_{i}$ for $i \in \{1,2,3,4\}$.
\end{enumerate}

Assume that $C$ is an irreducible curve of tri-degree $(4,4,4)$. By Lemma~\ref{222sections}, $C$ must lie in a pencil of divisors of tri-degree $(2,2,2)$. Note that there is a linear system $V_{4,4,4}$ of $G$-invariant divisors of tri-degree $(4,4,4)$, generated by 
    \begin{multline*}
        -x_3^4 y_1^4 y_2^4-x_2^4 y_1^4 y_3^4+x_1^4 y_2^4 y_3^4+x_1^4 x_2^4 x_3^4,
        -i x_1^2 x_2^4 y_1^2 y_3^4+i x_1^2 x_3^4 x_2^4 y_1^2-x_3^2 x_2^4 y_1^4 y_3^2+x_1^4 x_3^2 x_2^4 y_3^2+x_1^4 x_2^2 y_2^2 y_3^4\\+x_2^2 y_1^4 y_2^2 y_3^4+x_1^4 x_3^4 x_2^2 y_2^2+x_3^4 x_2^2 y_1^4 y_2^2+i x_1^2 x_3^4 y_1^2 y_2^4-i x_1^2 y_1^2 y_2^4 y_3^4+x_3^2 y_1^4 y_2^4 y_3^2-x_1^4 x_3^2 y_2^4 y_3^2,\\-x_2^4 x_3^4 y_1^4+x_1^4 x_3^4 y_2^4+x_1^4 x_2^4 y_3^4-y_2^4 y_3^4 y_1^4,x_1^2 x_2^2 x_3^2 y_1^2 y_2^2 y_3^2.
    \end{multline*}
    Its base locus is zero-dimensional. If $C$ does not lie in the locus of degenerate orbits and does not pass through base points of $V_{4,4,4}$, then for a general divisor $D \in V_{4,4,4}$, \[D\cdot C\geq|D \cap C|\geq96,\] but $D\cdot C = 48$, which is absurd. Hence either $C$ lies in the locus of degenerate orbits or $C$ passes through base points of $V_{4,4,4}$. The former case is already accounted for. Consider the latter case. The base locus of $V_{4,4,4}$ consists of $96$ points, which split into four orbits of length 24 each: 
    \begin{align*}
        &O_{24,1} = G_{96,3}\cdot\left([1:1],[1:0],[1:1]\right),\\ 
        &O_{24,2} = G_{96,3}\cdot\left([1:1],[0:1],[1:1]\right),\\
        &O_{24,3} = G_{96,3}\cdot\left([1:1],[0:1],[i:1]\right),\\ 
        &O_{24,4} = G_{96,3}\cdot\left([1:1],[1:0],[i:1]\right).
    \end{align*}
    Hence, $C$ must pass through one of these orbits. Assume that $C$ passes through the first one; the other cases are done similarly. Since $C$ must also lie in a divisor of tri-degree $(2,2,2)$, $C$ must lie in some $G$-invariant positive-dimensional linear subsystem of the linear system $U_1$ of divisors of tri-degree $(2,2,2)$ that pass through $O_{24,1}$. By a direct computation, this gives an $8$-dimensional linear space.
    Its decomposition into irreducible representations consists of four non-isomorphic irreducible representations: 
    \[U_1 = \langle x_1x_2x_3y_1y_2y_3\rangle\oplus W_2 \oplus W_2' \oplus W_3,\] where $W_2$ and $W_2'$ are two-dimensional, and $W_3$ is three-dimensional. So, if such $C$ exists, it must lie in the base locus of $W_2$, $W_2'$ or $W_3$. Computing base loci, we see that $\Bs(W_3)$ consists of the curve $\boldsymbol{\mathcal{S}}_1$ of tri-degree $(4,4,4)$ and eight $\mathbb{T}$-invariant points. The base loci of $W_2$ and $W_2'$ are both non-reduced curves of tri-degree $(8,8,8)$ that set-theoretically coincide with $\boldsymbol{\mathcal{S}}_1$. Thus, the only possible irreducible curve of tri-degree $(4,4,4)$ that passes through $O_{24,1}$ is $\boldsymbol{\mathcal{S}}_1$. Considering orbits $O_{24,2}$, $O_{24,3}$, and~$O_{24,4}$, we get $G$-invariant $G$-irreducible curves $\boldsymbol{\mathcal{S}}_2$, $\boldsymbol{\mathcal{S}}_3$, and~$\boldsymbol{\mathcal{S}}_4$, respectively.
\end{proof}

\begin{proposition}
\label{Gcurves4850}
Let $G=G_{48,50}$. The complete list of $G$-invariant $G$-irreducible curves of tri-degree $(a,a,a)$ with $a\leq 7$, other than $\boldsymbol{\mathcal{C}}_{\mathbb{T}}$, is as follows:
\begingroup\footnotesize
\begin{longtblr}[
theme = mytheme,
entry = none,
label = none,
]{
stretch = 1.25,
colspec = {|Q[c,m,co=0.3]|Q[c,m,co=0.5]|Q[c,m,co=1]|Q[l,m,co=5.8]|},
rowhead = 1,
}
\hline
$a$ & Curve & Components & Representative component\\
\hline
\hline
$4$ & $\boldsymbol{\mathcal{F}}_1$ & $12$ rulings & $\mathbb{P}^1\times\{[0:1]\}\times\{[1:1]\}$\\
\cline{2-4}
& $\boldsymbol{\mathcal{F}}_2$ & $12$ rulings & $\mathbb{P}^1\times\{[0:1]\}\times\{[i:1]\}$\\
\cline{2-4}
& $\boldsymbol{\mathcal{F}}_3$ & $12$ rulings & $\{[1:1]\}\times\{[0:1]\}\times\mathbb{P}^1$\\
\cline{2-4}
& $\boldsymbol{\mathcal{F}}_4$ & $12$ rulings & $\{[i:1]\}\times\{[0:1]\}\times\mathbb{P}^1$\\
\cline{2-4}
& $\boldsymbol{\mathcal{F}}_5$ & $12$ rulings & $\{[i:1]\}\times\mathbb{P}^1\times\{[1:1]\}$\\
\cline{2-4}
& $\boldsymbol{\mathcal{F}}_6$ & $12$ rulings & $\{[1:1]\}\times\mathbb{P}^1\times\{[i:1]\}$\\
\cline{2-4}
& $\boldsymbol{\mathcal{F}}'_1$ & $12$ rulings & $\{[1:1]\}\times\mathbb{P}^1\times\{[1:1]\}$\\
\cline{2-4}
& $\boldsymbol{\mathcal{F}}'_2$ & $12$ rulings & $\{[i:1]\}\times\mathbb{P}^1\times\{[i:1]\}$\\
\cline{2-4}
& $\boldsymbol{\mathcal{R}}_1$ & $4\times(1,1,1)$ & $\left\{\left([x:y],[x-iy:-x-iy],[x-y:-ix-iy]\right)\mid[x:y]\in\mathbb{P}^1\right\}$\\
\cline{2-4}
& $\boldsymbol{\mathcal{R}}_2$ & $4\times(1,1,1)$ & $\left\{\left([x:y],[x-iy:-x-iy],[x-y:ix+iy]\right)\mid[x:y]\in\mathbb{P}^1\right\}$\\
\cline{2-4}
& $\boldsymbol{\mathcal{R}}_3$ & $4\times(1,1,1)$ & $\left\{\left([x:y],[x-iy:-x-iy],[x+y:-ix+iy]\right)\mid[x:y]\in\mathbb{P}^1\right\}$\\
\cline{2-4}
& $\boldsymbol{\mathcal{R}}_4$ & $4\times(1,1,1)$ & $\left\{\left([x:y],[x-iy:-x-iy],[x+y:ix-iy]\right)\mid[x:y]\in\mathbb{P}^1\right\}$\\
\cline{2-4}
& $\boldsymbol{\mathcal{R}}_5$ & $4\times(1,1,1)$ & $\left\{\left([x:y],[x-y:-ix-iy],[x-iy:-x-iy]\right)\mid[x:y]\in\mathbb{P}^1\right\}$\\
\cline{2-4}
& $\boldsymbol{\mathcal{R}}_6$ & $4\times(1,1,1)$ & $\left\{\left([x:y],[x-y:-ix-iy],[x-iy:x+iy]\right)\mid[x:y]\in\mathbb{P}^1\right\}$\\
\cline{2-4}
& $\boldsymbol{\mathcal{R}}_7$ & $4\times(1,1,1)$ & $\left\{\left([x:y],[x-y:-ix-iy],[x+iy:-x+iy]\right)\mid[x:y]\in\mathbb{P}^1\right\}$\\
\cline{2-4}
& $\boldsymbol{\mathcal{R}}_8$ & $4\times(1,1,1)$ & $\left\{\left([x:y],[x-y:-ix-iy],[x+iy:x-iy]\right)\mid[x:y]\in\mathbb{P}^1\right\}$\\
\hline
\end{longtblr}
\endgroup
\end{proposition}

\begin{proof}
Let $C$ be a $G$-invariant $G$-irreducible curve of tri-degree $(a,a,a)$, $a\leq 7$, and $C \neq \boldsymbol{\mathcal{C}}_{\mathbb{T}}$. Assume that $C$ does not lie in either the locus of degenerate $G$-orbits or $\boldsymbol{\mathcal{D}}_{\mathbb{T}}$; the ones that do are the curves $\boldsymbol{\mathcal{F}}_i$ for $i \in \{1,2,3,4\}$ and $\boldsymbol{\mathcal{F}}_i'$ for $i\in \{1,2\}$. First, apply Lemma~\ref{rulingsstab}. We compute that $|K_i| = 4$; hence, either $C$ is contained in the union of degenerate loci of $K_i$ or every $F$-component of $C$ has tri-degree divisible by 4. The curves contained in these degenerate loci are precisely the curves $\boldsymbol{\mathcal{F}}_{1}'$ and $\boldsymbol{\mathcal{F}}_{2}'$. Thus, we can assume that $C$ is not one of these curves and satisfies the divisibility condition. We get that the tri-degree of $C$ is $(4,4,4)$, and there are the following possibilities for the decomposition of $C$ into irreducible components:
\begin{enumerate}
    \item an irreducible $(4,4,4)$-curve;
    \item two irreducible $(2,2,2)$-curves; impossible since there are no index 2 subgroups in $G_{48,50}$;
    \item four irreducible $(1,1,1)$-curves: $\boldsymbol{\mathcal{R}}_{i}$ for $i \in \{1,\dots,8\}$;
    \item four $(1,0,0)$-, four $(0,1,0)$-, and four $(0,0,1)$-curves: $\boldsymbol{\mathcal{F}}_{i}$ for $i \in \{1,\dots,6\}$ and $\boldsymbol{\mathcal{F}}_{i}'$ for $i \in \{1,2\}$.
\end{enumerate}

Assume that $C$ is an irreducible $(4,4,4)$-curve. We know, by Lemma~\ref{222sections}, that $C$ must lie in a pencil of $(2,2,2)$-divisors. Suppose that $C$ lies in a smooth $(2,2,2)$-divisor. A smooth $(2,2,2)$-divisor is a K3 surface. Then, by the Hodge index theorem, \[\det\begin{pmatrix}
H^2 & C\cdot H\\
C\cdot H & C^2 
\end{pmatrix}\leq0.\] Thus, \[C^2\leq 12.\] By the Riemann--Roch theorem for surfaces: \[p_a(C) = 1+\frac{1}{2}(C^2+K\cdot C) = 1+\frac{1}{2}C^2,\] since for K3 surfaces $K = 0$. Therefore, \[p_a(C) \leq 7.\] Consider the Riemann-Hurwitz formula for the cover $\pi: \widetilde{C} \rightarrow \widetilde{C}/G$, where $\widetilde{C}$ is the normalization of $C$: \[2g(\widetilde{C})-2 = |G|(2g(\widetilde{C}/G)-2)+R,\ R\geq0.\] Since $g(\widetilde{C})\leq7$, \[48(2g(\widetilde{C}/G)-2)\leq48(2g(\widetilde{C}/G)-2)+R\leq12.\] Thus, $g(\widetilde{C}/G) = 0$ or $1$. For $g(\widetilde{C}/G)=0$, by \cite{LMFDB}, there are no curves $\widetilde{C}$ with an action of $G_{48,50}$ such that $2\leq g(\widetilde{C})\leq7$ and $g(\widetilde{C}/G_{48,50})=0$. The genus of $\widetilde{C}$ is not 0 either, since $G_{48,50}$ cannot act effectively on $\mathbb{P}^1$. Thus, the only possibility is $g(\widetilde{C}) = 1$. This case is also not possible, since $G_{48,50}$ is not of the form~\eqref{eq:Gellipticcurve}. Hence, the case $g(\widetilde{C}/G)=0$ is impossible. Thus, the only remaining case is $g(\widetilde{C}/G) = 1$. Consider the ramification divisor's degree: \[R = \sum_{i}(|G|-|O_i|),\] where $O_i$ denote degenerate orbits in $\widetilde{C}$. The possible lengths for $O_i$ are \[\{1,2,3,4,6,8,12,24\},\] hence if $R \neq 0$, $R \geq 24$. In the latter case: \[24\leq 2g(\widetilde{C})-2 = R\leq12,\] which is absurd. Hence, $R = 0$, $g(\widetilde{C})=1$, and the action of $G$ on $\widetilde{C}$ is free. But we have already seen that $G_{48,50}$ cannot act on an elliptic curve. Hence, there can be no $G$-curves $C$ of tri-degree $(4,4,4)$ lying on a smooth K3 surface.

The only remaining case is when $C$ lies on a singular $(2,2,2)$-divisor. We know that $C$ must lie in a pencil of singular $(2,2,2)$-divisors. Note that there is a $G$-invariant linear system $V_{2,2,2}$ of $G$-invariant $(2,2,2)$-divisors, generated by: \[f_1 = x_3^2 y_1^2 y_2^2+x_2^2 y_1^2 y_3^2+x_1^2 y_2^2 y_3^2+x_1^2 x_2^2 x_3^2,\ f_2 = x_2^2 x_3^2 y_1^2+x_1^2 x_3^2 y_2^2+x_1^2 x_2^2 y_3^2+y_2^2 y_3^2 y_1^2,\ f_3 = x_1 x_2 x_3 y_1 y_2 y_3,\] whose base locus consists of $48$ points and splits into two orbits of length $24$; denote them $O_{24,1}$ and $O_{24,2}$. Let $D$ be a general divisor in $V_{2,2,2}$. If $C$ does not pass through base points of $V_{2,2,2}$ and does not lie in the locus of degenerate orbits, $D$ must intersect $C$ in a non-degenerate orbit; thus, $D\cdot C\geq 48$. But $D\cdot C = 24$, hence, if such $C$ exists, it must either lie in the locus of degenerate orbits of $G$ or pass through one of the orbits $O_{24,1}$ or $O_{24,2}$. The former case is already accounted for. Hence, assume we are in the latter case. The linear systems $U_1$ and $U_2$ of divisors of tri-degree $(2,2,2)$ passing through $O_{24,1}$ and $O_{24,2}$, respectively, are both of dimension $6$. Decomposing them into the isotypical components, we get: \[U_1 = V_{2,2,2}\oplus W_3,\ U_2 = V_{2,2,2}\oplus W_3',\] where both $W_3$ and $W_3'$ are irreducible of dimension 3. In both cases the three-dimensional linear subsystems $W_3$ and $W_3'$ have zero-dimensional base locus, so they cannot contain a curve. Thus, if we want to have a pencil of singular $(2,2,2)$-divisors, it must be a subsystem of $V_{2,2,2}$. Parametrizing the system as \[Af_1+Bf_2+Cf_3,\] we compute that there are six pencils of singular divisors, defined by the following equations on $A,B,C$: \[A=0,\ B=0,\ 4A+4B-C=0,\ 4A+4B+C=0,\ 4A-4B+iC=0,\ 4A-4B-iC=0.\] In particular, we see that one of $f_3,\ f_1+f_2,\ f_1-f_2$ must lie in the pencil. But $V(f_3) = \boldsymbol{\mathcal{D}}_{\mathbb{T}}$, $V(f_1 + f_2)$, and~$V(f_1 - f_2)$ are all reducible and $G$-irreducible; hence, since $C$ is irreducible, it must lie in the intersection of irreducible components. But in each case the intersection is empty. Therefore, $C$ cannot lie in a singular pencil. Thus, the list of $G_{48,50}$-invariant curves is complete.
\end{proof}

\begin{proposition}
\label{Gcurves324}
Let $G=G_{324,160}$ or $G_{324,160}'$. The complete list of $G$-invariant
$G$-irreducible curves of tri-degree $(a,a,a)$ with $a\leq 7$, other
than $\boldsymbol{\mathcal{C}}_{\mathbb T}$, is as follows:
\begingroup\footnotesize
\begin{longtblr}[
theme = mytheme,
entry = none,
label = none,
]{
stretch = 1.25,
colspec = {|Q[c,m,co=0.3]|Q[c,m,co=0.5]|Q[c,m,co=1]|Q[l,m,co=5.8]|},
rowhead = 1
}
\hline
$a$ & Curve & Components & Representative component\\
\hline
\hline
$6$ & $\boldsymbol{\mathcal{F}}_{3,1}$ & $18$ rulings & $\mathbb{P}^1\times\{[0:1]\}\times\{[1:1]\}$\\
\cline{2-4}
& $\boldsymbol{\mathcal{F}}_{3,2}$ & $18$ rulings & $\mathbb{P}^1\times\{[0:1]\}\times\{[-1:1]\}$\\
\cline{2-4}
& $\boldsymbol{\mathcal{F}}_{3,3}$ & $18$ rulings & $\{[1:1]\}\times\{[0:1]\}\times\mathbb{P}^1$\\
\cline{2-4}
& $\boldsymbol{\mathcal{F}}_{3,4}$ & $18$ rulings & $\{[-1:1]\}\times\{[0:1]\}\times\mathbb{P}^1$\\
\hline
\end{longtblr}
\endgroup
\end{proposition}
\begin{proof}
    Let $C$ be a $G$-invariant $G$-irreducible curve of tri-degree $(a,a,a)$, $a\leq 7$, and $C \neq \boldsymbol{\mathcal{C}}_{\mathbb{T}}$. By a direct computation, there exists a base-point-free linear system $U$ of $G$-invariant divisors of tri-degree $(d,d,d)$ with $d\leq\frac{|G|}{24}$. For a general member $D \in U$, if $C$ is a $G$-invariant curve of tri-degree $(a,a,a)$ with $a<8$ and $C$ does not lie in the locus of degenerate orbits, then we have the following inequality: \[24d>3ad=D\cdot C \geq |D\cap C|\geq|G|,\] which is a contradiction. Thus, $C$ lies in the locus of degenerate orbits. Such curves are computed directly and are exactly the curves in the list.
\end{proof}

\begin{proposition}
\label{Gcurves4849}
Let $G=G_{48,49}'$. The complete list of $G$-invariant $G$-irreducible curves of tri-degree $(a,a,a)$ with $a\leq 7$, other than $\boldsymbol{\mathcal{C}}_{\mathbb{T}}$, is as follows:
\begingroup\footnotesize
\begin{longtblr}[
theme = mytheme,
entry = none,
label = none,
]{
stretch = 1.25,
colspec = {|Q[c,m,co=0.3]|Q[c,m,co=0.5]|Q[c,m,co=1]|Q[l,m,co=5.8]|},
rowhead = 1
}
\hline
$a$ & Curve & Components & Representative component\\
\hline
\hline
$4$ & $\boldsymbol{\mathcal{F}}_1$ & $12$ rulings & $\mathbb{P}^1\times\{[0:1]\}\times\{[1:1]\}$\\
\cline{2-4}
& $\boldsymbol{\mathcal{F}}_2$ & $12$ rulings & $\mathbb{P}^1\times\{[0:1]\}\times\{[i:1]\}$\\
\cline{2-4}
& $\boldsymbol{\mathcal{F}}_3$ & $12$ rulings & $\mathbb{P}^1\times\{[1:1]\}\times\{[0:1]\}$\\
\cline{2-4}
& $\boldsymbol{\mathcal{F}}_4$ & $12$ rulings & $\mathbb{P}^1\times\{[i:1]\}\times\{[0:1]\}$\\
\cline{2-4}
& $\boldsymbol{\mathcal{F}}_5$ & $12$ rulings & $\mathbb{P}^1\times\{[1:1]\}\times\{[i:1]\}$\\
\cline{2-4}
& $\boldsymbol{\mathcal{F}}_6$ & $12$ rulings & $\mathbb{P}^1\times\{[i:1]\}\times\{[1:1]\}$\\
\cline{2-4}
& $\boldsymbol{\mathcal{F}}'_1$ & $12$ rulings & $\mathbb{P}^1\times\{[1:1]\}\times\{[1:1]\}$\\
\cline{2-4}
& $\boldsymbol{\mathcal{F}}'_2$ & $12$ rulings & $\mathbb{P}^1\times\{[i:1]\}\times\{[i:1]\}$\\
\cline{2-4}
& $\boldsymbol{\mathcal{R}}_1$ & $4\times(1,1,1)$ & $\left\{\left([x:y],[x:y],[x:y]\right)\mid[x:y]\in\mathbb{P}^1\right\}$\\
\cline{2-4}
& $\boldsymbol{\mathcal{R}}_2$ & $4\times(1,1,1)$ & $\left\{\left([x:y],[x+y:ix-iy],[x-iy:x+iy]\right)\mid[x:y]\in\mathbb{P}^1\right\}$\\
\cline{2-4}
& $\boldsymbol{\mathcal{Q}}_{2,1}$ & $2\times(2,2,2)$ & {$\displaystyle\begin{cases}x_1x_2x_3+\varepsilon_{12}^{7}y_1x_2x_3-\varepsilon_{12}^{11}x_1y_2x_3+\varepsilon_{3}^{2}y_1x_2y_3-\varepsilon_{3}x_1y_2y_3-iy_1y_2y_3=0,\\ \varepsilon_{3}^{2}y_1x_2x_3+\varepsilon_{3}x_1y_2x_3+iy_1y_2x_3+x_1x_2y_3+\varepsilon_{12}^{7}y_1x_2y_3+\varepsilon_{12}^{11}x_1y_2y_3=0.\end{cases}$}\\
\cline{2-4}
& $\boldsymbol{\mathcal{Q}}_{2,2}$ & $2\times(2,2,2)$ & {$\displaystyle\begin{cases}x_1x_2x_3+\varepsilon_{12}^{11}y_1x_2x_3-\varepsilon_{12}^{7}x_1y_2x_3+\varepsilon_{3}y_1x_2y_3-\varepsilon_{3}^{2}x_1y_2y_3-iy_1y_2y_3=0,\\ \varepsilon_{3}y_1x_2x_3+\varepsilon_{3}^{2}x_1y_2x_3+iy_1y_2x_3+x_1x_2y_3+\varepsilon_{12}^{11}y_1x_2y_3+\varepsilon_{12}^{7}x_1y_2y_3=0.\end{cases}$}\\
\cline{2-4}
& $\boldsymbol{\mathcal{Q}}_{2,3}$ & $2\times(2,2,2)$ & {$\displaystyle\begin{cases}x_1x_2x_3-\varepsilon_{3}y_1x_2x_3+\varepsilon_{3}^{2}x_1y_2x_3+\varepsilon_{3}^{2}y_1x_2y_3-\varepsilon_{3}x_1y_2y_3+y_1y_2y_3=0,\\ \varepsilon_{3}^{2}y_1x_2x_3+\varepsilon_{3}x_1y_2x_3+y_1y_2x_3+x_1x_2y_3+\varepsilon_{3}y_1x_2y_3+\varepsilon_{3}^{2}x_1y_2y_3=0.\end{cases}$}\\
\cline{2-4}
& $\boldsymbol{\mathcal{Q}}_{2,4}$ & $2\times(2,2,2)$ & {$\displaystyle\begin{cases}x_1x_2x_3-\varepsilon_{3}^{2}y_1x_2x_3+\varepsilon_{3}x_1y_2x_3+\varepsilon_{3}y_1x_2y_3-\varepsilon_{3}^{2}x_1y_2y_3+y_1y_2y_3=0,\\ \varepsilon_{3}y_1x_2x_3+\varepsilon_{3}^{2}x_1y_2x_3+y_1y_2x_3+x_1x_2y_3+\varepsilon_{3}^{2}y_1x_2y_3+\varepsilon_{3}x_1y_2y_3=0.\end{cases}$}\\
\cline{2-4}
& $\boldsymbol{\mathcal{Q}}_{2,5}$ & $2\times(2,2,2)$ & {$\displaystyle\begin{cases}x_1x_2x_3+\varepsilon_{3}^{2}y_1y_2x_3-\varepsilon_{3}y_1x_2y_3=0,\\ \varepsilon_{3}y_1y_2x_3+\varepsilon_{3}^{2}y_1x_2y_3+x_1y_2y_3=0.\end{cases}$}\\
\cline{2-4}
& $\boldsymbol{\mathcal{Q}}_{2,6}$ & $2\times(2,2,2)$ & {$\displaystyle\begin{cases}x_1x_2x_3+\varepsilon_{3}y_1y_2x_3-\varepsilon_{3}^{2}y_1x_2y_3=0,\\ \varepsilon_{3}^{2}y_1y_2x_3+\varepsilon_{3}y_1x_2y_3+x_1y_2y_3=0.\end{cases}$}\\
\hline
\end{longtblr}
\endgroup
\end{proposition}
\begin{proof}
Let $C$ be a $G$-invariant $G$-irreducible curve of tri-degree $(a,a,a)$, $a\leq 7$, and $C \neq \boldsymbol{\mathcal{C}}_{\mathbb{T}}$. Assume that $C$ does not lie in either the locus of degenerate $G$-orbits or $\boldsymbol{\mathcal{D}}_{\mathbb{T}}$; the ones that do are the curves $\boldsymbol{\mathcal{F}}_i$ for $i \in \{1,\dots,5\}$ and $\boldsymbol{\mathcal{R}}_1$. First, apply Lemma~\ref{rulingsstab}. We compute that $|K_i| = 4$; hence, either $C$ is contained in the union of degenerate loci of $K_i$ or every $F$-component of $C$ has tri-degree divisible by 4. The curves contained in these degenerate loci are precisely the curves $\boldsymbol{\mathcal{F}}_{i}$ and $\boldsymbol{\mathcal{F}}_{i}'$ for some $i$. Thus, we can assume that $C$ is not one of these curves and satisfies the divisibility condition. We get that the tri-degree of $C$ is $(4,4,4)$, and there are the following possibilities for the decomposition of $C$ into irreducible components:
        \begin{enumerate}
            \item an irreducible $(4,4,4)$-curve;
            \item two irreducible $(2,2,2)$-curves: $\boldsymbol{\mathcal{Q}}_{2,i}$ for $i \in \{1,\dots,6\}$;
            \item four irreducible $(1,1,1)$-curves: $\boldsymbol{\mathcal{R}}_{1}$ and $\boldsymbol{\mathcal{R}}_{2}$;
            \item four $(1,0,0)$-, four $(0,1,0)$-, and four $(0,0,1)$-curves: $\boldsymbol{\mathcal{F}}_{i}$ for $i \in \{1,\dots,6\}$ and $\boldsymbol{\mathcal{F}}_{i}'$ for $i \in \{1,2\}$.
        \end{enumerate}
        
        From now on assume that $C$ is an irreducible curve of tri-degree $(4,4,4)$. Consider $G$-invariant divisors of tri-degree $(2,2,2)$. They are the following:
        \begin{enumerate}
    \item $D_1=\boldsymbol{\mathcal{D}}_{\mathbb{T}}$:
    \[
        f_1=x_1y_1x_2y_2x_3y_3=0;
    \]

    \item $D_2$:
    \[
        f_2=(x_1-y_1)(x_1+y_1)(x_2-y_2)(x_2+y_2)(x_3-y_3)(x_3+y_3)=0;
    \]

    \item $D_3$:
    \[
        f_3=(x_1-iy_1)(x_1+iy_1)(x_2-iy_2)(x_2+iy_2)(x_3-iy_3)(x_3+iy_3)=0;
    \]

    \item $D_4$:
    \begin{multline*}
        f_4=x_1y_1x_2^2x_3^2+x_1^2x_2y_2x_3^2+y_1^2x_2y_2x_3^2-x_1y_1y_2^2x_3^2+x_1^2x_2^2x_3y_3-y_1^2x_2^2x_3y_3\\
        {}+x_1^2y_2^2x_3y_3-y_1^2y_2^2x_3y_3+x_1y_1x_2^2y_3^2-x_1^2x_2y_2y_3^2-y_1^2x_2y_2y_3^2-x_1y_1y_2^2y_3^2=0;
    \end{multline*}

    \item $D_5$:
    \begin{multline*}
        f_5=\ee_3x_1y_1x_2^2x_3^2+\ee_3^2x_1^2x_2y_2x_3^2+\ee_3^2y_1^2x_2y_2x_3^2-\ee_3x_1y_1y_2^2x_3^2+x_1^2x_2^2x_3y_3-y_1^2x_2^2x_3y_3\\
        {}+x_1^2y_2^2x_3y_3-y_1^2y_2^2x_3y_3+\ee_3x_1y_1x_2^2y_3^2-\ee_3^2x_1^2x_2y_2y_3^2-\ee_3^2y_1^2x_2y_2y_3^2-\ee_3x_1y_1y_2^2y_3^2=0;
    \end{multline*}

    \item $D_6$:
    \begin{multline*}
        f_6=\ee_3^2x_1y_1x_2^2x_3^2+\ee_3x_1^2x_2y_2x_3^2+\ee_3y_1^2x_2y_2x_3^2-\ee_3^2x_1y_1y_2^2x_3^2+x_1^2x_2^2x_3y_3-y_1^2x_2^2x_3y_3\\
        {}+x_1^2y_2^2x_3y_3-y_1^2y_2^2x_3y_3+\ee_3^2x_1y_1x_2^2y_3^2-\ee_3x_1^2x_2y_2y_3^2-\ee_3y_1^2x_2y_2y_3^2-\ee_3^2x_1y_1y_2^2y_3^2=0,
    \end{multline*}
\end{enumerate}
 with pairwise distinct characters. Note that $D_1$, $D_2$, and $D_3$ are reducible, $G$-irreducible, and have an empty intersection of their components. Hence, $C$ cannot lie in them. Also note that $D_4$ is smooth. We also need the $G$-orbit description for each of them. From Proposition~\ref{orbits4849}, we have the orbit decomposition. Here we list all degenerate orbits for each $G$-invariant divisor of tri-degree $(2,2,2)$:
\begingroup \footnotesize\begin{longtblr}[
theme = mytheme,
entry = none,
label = none,
]{
stretch = 1.25,
colspec = {|Q[c,m]|Q[l,m]|},
rowhead = 1,
}
\hline
Divisor & Degenerate orbits\\
\hline
\hline
$D_1$
&
$3\cdot\mathcal{O}_{8,1}$,\
$\Gamma_{24,1}$,\
$\Gamma_{24,2}$,\
$\mathcal{O}_{24}$
\\
\hline
$D_2$
&
$3\cdot\mathcal{O}_{8,2}$,\
$\Gamma_{24,1}$,\
$\Gamma_{24,3}$,\
$\mathcal{O}_{24}$
\\
\hline
$D_3$
&
$3\cdot\mathcal{O}_{8,3}$,\
$\Gamma_{24,2}$,\
$\Gamma_{24,3}$,\
$\mathcal{O}_{24}$
\\
\hline
$D_4$
&
$\mathcal{O}_{8,1}$,\
$\mathcal{O}_{8,2}$,\
$\mathcal{O}_{8,3}$,\
$\mathcal{O}_{24,1}^{(4)}$,\
$\mathcal{O}_{24,2}^{(4)}$,\
$\mathcal{O}_{24,3}^{(4)}$,\
$\mathcal{O}_{24}$
\\
\hline
$D_5$
&
$\Gamma_{16}$;
$\mathcal{O}_{24,1}^{(5)}$,\
$\mathcal{O}_{24,2}^{(5)}$,\
$\mathcal{O}_{24,3}^{(5)}$,\
$\mathcal{O}_{24}$
\\
\hline
$D_6$
&
$\Gamma_{16}$;
$\mathcal{O}_{24,1}^{(6)}$,\
$\mathcal{O}_{24,2}^{(6)}$,\
$\mathcal{O}_{24,3}^{(6)}$,\
$\mathcal{O}_{24}$
\\
\hline
\end{longtblr}\endgroup \noindent where the numbers before the orbits denote the multiplicities, and $\mathcal{O}_{24,i}^{(k)}$ denotes generic members of the corresponding families $\Gamma_{24,i}$. 

        Let us utilize the intersection numbers with $G$-invariant divisors. Since $C$ is a curve of tri-degree $(4,4,4)$, we have: \[D_i\cdot C = 24,\] for each $G$-invariant divisor of tri-degree $(2,2,2)$. As both $C$ and $D_i$ are $G$-invariant, $D_i\cdot C$ must be $G$-invariant. Moreover, since the order of $G_{48,49}'$ is 48, and the action of $G$ on each $G$-invariant divisor is faithful, $C$ cannot intersect $D_i$ in a generic orbit; else it will overcount $D_i \cdot C$. Consider $D_4 \cdot C$. We have the following cases:
        \begin{enumerate}
            \item $C$ lies in $D_4$. Then $C$ lies on a K3 surface, since $D_4$ is smooth. By the Hodge index theorem, $C^2\leq 12$ and $p_a(C)\leq 7$. Consider the Riemann-Hurwitz theorem for the cover $\pi: \widetilde{C} \rightarrow \widetilde{C}/G$, where $\widetilde{C}$ is the normalization of $C$:
        \[2g(\widetilde{C})-2 = |G|(2g(\widetilde{C}/G)-2)+R,\ R\geq 0.\] It follows then that $g(\widetilde{C}/G) = 0$ or $1$. If $g(\widetilde{C}/G) = 1$, then $g(\widetilde{C}) = 1$ and $R = 0$. But this is impossible because $G_{48,49}'$ cannot act faithfully on an elliptic curve. Hence $g(\widetilde{C}/G) = 0$. The case of $g(\widetilde{C}) = 0$ or $1$ is impossible, because $G_{48,49}'$ acts faithfully neither on $\mathbb{P}^1$ nor on an elliptic curve. Thus, $2\leq g(\widetilde{C}) \leq 7$ and $g(\widetilde{C}/G)=0$. In such a case, by \cite{LMFDB}, the only possibility is $g(\widetilde{C}) = 5$, and the action has the signature $[0;2,6,6]$. Thus, non-free orbits on $\widetilde{C}$ are of length $24$, $8$, and~$8$, respectively. On $D_4$ we have three orbits of length 8: $\mathcal{O}_{8,1}$, $\mathcal{O}_{8,2}$, and~$\mathcal{O}_{8,3}$. Hence, either $C$ passes through two of them or $C$ passes through one of them with multiplicity (at least) two. Neither of these cases is realized, which can be seen by examining the intersection of $C$ and $D_1$, $D_2$, and~$D_3$.
        \item $D_4 \cdot C = \mathcal{O}^{(4)}_{24,k}$ for some $k = 1,2,3$ or $D_4 \cdot C$ consists of two orbits of length 8 with multiplicities 2 and 1, respectively. These cases are also impossible by considering intersections of $C$ with $D_1$, $D_2$, and~$D_3$.
        \item $D_4 \cdot C = \mathcal{O}_{24}$ or $\mathcal{O}_{8,1} + \mathcal{O}_{8,2}+\mathcal{O}_{8,3}$. By Lemma~\ref{222sections}, $C$ must lie in a pencil of $(2,2,2)$-divisors. Computing the space of $(2,2,2)$-divisors passing through $D_4 \cdot C$, computing its decomposition into irreducible $G$-modules, and finding base loci of pencils, we find no $G$-invariant irreducible curves of tri-degree $(4,4,4)$.
        \item $D_4 \cdot C$ is an orbit of length 8 with multiplicity 3. For concreteness, suppose \[D_4 \cdot C = 3\cdot \mathcal{O}_{8,1}.\] Note that $C$ must be smooth at $\mathcal{O}_{8,1}$, else it will overcount $D_1 \cdot C$. It turns out that this case also is not realized. Indeed, let \[P_0=([0:1],[0:1],[0:1]),\] and let $G_{P_0}$ be its stabilizer in $G$. We compute the $G_{P_0}$-semi-invariant tangent directions to $D_4$ at $P_0$, transport them to the remaining points of $\boldsymbol{\mathcal O}_{8,1}$ via the action of $G$, and, for each resulting collection of tangent directions, decompose the space of $(2,2,2)$-divisors passing through $\boldsymbol{\mathcal O}_{8,1}$ with these prescribed tangent directions into irreducible $G$-modules and compute the base loci of its pencils. In every case, no $G$-invariant irreducible curve of tri-degree $(4,4,4)$ occurs.
        \end{enumerate}
        Thus, we have exhausted all possibilities for an irreducible $G$-invariant curve of tri-degree $(4,4,4)$. Hence, our list is complete.
\end{proof}

\section{Birational rigidity}
\label{Gbirrig}
In this section we finally prove Theorem~\ref{maintheorem}. The section is organized as follows. In Subsection~\ref{subsecGbirsupriglarge}, we prove $G$-birational superrigidity for groups with sufficiently large toric kernel. In Subsection~\ref{subsecGbirsuprigspecial}, we prove $G$-birational superrigidity for our list of special groups~\eqref{eq:small-superrigidity-groups}. In Subsection~\ref{subsecnotGsolid}, we deal with the cases where $X$ is not $G$-birationally superrigid; all of them are in fact not $G$-solid.
\subsection{$G$-birational superrigidity for sufficiently large groups}
\label{subsecGbirsupriglarge}
First, let us prove $G$-birational superrigidity for all sufficiently large groups. Recall that we always assume that $W_G$ contains $\mathfrak{A}_4$.
\begin{theorem}
\label{Gbirsupriglarge}
For any finite toric group $G$ with $W_G \supset \mathfrak{A}_4$ and $n\geq 4$, the variety $X = \mathbb{P}^1 \times \mathbb{P}^1 \times \mathbb{P}^1$ is $G$-birationally superrigid.
\end{theorem}

We will prove this theorem in a series of lemmas. The argument essentially follows \cite[7.2]{ToricGsolid}. The proof, as well as all subsequent proofs of $G$-birational superrigidity, follows the same general strategy, based on Lemma~\ref{mobilesystems}:
\begin{enumerate}
\item We argue by contradiction and assume that $X$ is not $G$-birationally superrigid.
\item By Lemma~\ref{mobilesystems}, there exists a $G$-invariant mobile linear system $\M$ such that
\[K_X+\lambda\M \sim_{\mathbb Q}0\]
for some positive rational number $\lambda$, and the log-pair $(X,\lambda\M)$ has non-canonical singularities.
\item Hence there exists a non-empty $G$-irreducible $G$-center of non-canonical singularities $Z$.
\item Since $\M$ is mobile, the center $Z$ is either a $G$-invariant $G$-irreducible curve or the $G$-orbit of a point.
\item We first rule out all possibilities in which $Z$ is a curve. Thus, $Z$ must be the $G$-orbit of a point.
\item We then show that such an orbit must lie on a $G$-invariant curve of small tri-degree.
\item Finally, we eliminate the remaining cases by ruling out the possibility that $Z$ lies on any such $G$-invariant curve.
\end{enumerate}
The main difference between our proof and \cite[7.2]{ToricGsolid} is Corollary~\ref{nocurves}, which lets us eliminate curves for much smaller groups.

Let us start with the proof of Theorem~\ref{Gbirsupriglarge}. We use the following lemma from \cite{ToricGsolid} to prove that $X$ is $G$-birationally superrigid:
\begin{lemma}[\protect{\cite[7.2]{ToricGsolid}}]
\label{sufficientGbirsuprig}
Let $X = \PPP$ and let $G$ be a finite toric group acting on $X$, and assume $W_G \supset \mathfrak{A}_4$. Then the following two conditions are sufficient for $X$ to be $G$-birationally superrigid: 
\begin{enumerate}
    \item For every non-empty $G$-invariant mobile system $\mathcal{M}$ on the threefold $X$, all non-canonical centers (if any) of the mobile log-pair $(X,\lambda\mathcal{M})$ are $\mathbb{T}$-invariant, where $\lambda$ is a positive rational number such that $\lambda\mathcal{M}\sim_{\mathbb{Q}}-K_X$.
    \item If a $\mathbb{T}$-invariant smooth point $P\in X$ is a non-canonical center of the log-pair $(X,\lambda\mathcal{M})$, then $\mult_P(\mathcal{M})>\frac{2}{\lambda}$.
\end{enumerate}
\end{lemma}

So, let $\mathcal{M}$ be some non-empty $G$-invariant mobile linear system on $X$, and let $\lambda$ be a positive rational number such that $-K_X \sim_{\mathbb{Q}}\lambda\mathcal{M}.$
\begin{lemma}
\label{faithfulforT}
    Let $G$ be a finite toric group with $W_G \supset \mathfrak{A}_4$ and $n\geq 4$. Let $P \in X$ be a $\mathbb{T}$-invariant point such that it is a non-canonical center of $(X,\lambda\mathcal{M})$. Then $\mult_P(\mathcal{M})>\frac{2}{\lambda}$.
\end{lemma}
\begin{proof}
    Let $G_P$ be the stabilizer of $P$ in $G$. Then, by Lemma~\ref{faithrep}, the action of $G_P$ on the tangent space $T_P(X)$ is faithful. Moreover, this action is irreducible. Indeed, $G_P$ contains $T_G$, and, since $n \geq 4$, the three $T_G$-irreducible submodules of $T_P(X)$ have distinct characters. But $G_P$ also contains elements that permute those submodules transitively, which proves the $G_P$-irreducibility. Hence $\mult_P(\mathcal{M})>\frac{2}{\lambda}$ by Lemma~\ref{irrepmult}.
\end{proof}

Now suppose $(X,\lambda\M)$ is non-canonical, and let $Z$ be a $G$-invariant center of non-canonical singularities. Let us show that $Z$ must be $\mathbb{T}$-invariant.
\begin{lemma}
\label{Znotacurve}
    For \textbf{\emph{any}} finite group $G$ with $\rank\Pic(X)^G = 1$, if $Z$ is a curve, then the tri-degree of $Z$ is less than $(8,8,8)$. In particular, if $G$ is a toric group with $W_G \supset \mathfrak{A}_4$ and $n\geq 4$, then $Z$ is $\mathbb{T}$-invariant.
\end{lemma}
\begin{proof}
    Let $H$ be a divisor of tri-degree $(1,1,1)$. Then, for two general surfaces $M_1,M_2$ in $\M$, we have: \[\frac{24}{\lambda^2}=H\cdot M_1 \cdot M_2\geq H\cdot (\mult_Z(\M))^2 Z>\frac{H\cdot Z}{\lambda^2},\] where the last inequality is true by Lemma~\ref{noncancurve}. Thus, $H\cdot Z<24$. But on the other hand, if $Z$ is a curve of tri-degree $(a,a,a)$, then $H\cdot Z=3a$. Hence the tri-degree of $Z$ is less than $(8,8,8)$. Moreover, for the case of $n\geq 4$, by Corollary~\ref{nocurves} there are no $G$-invariant curves of tri-degrees less than 8 such that they do not lie in the union of $\mathbb{T}$-invariant curves. Therefore we see that $Z$ must be $\mathbb{T}$-invariant.
\end{proof}

Thus, the only possibility for $Z$ to be a curve is $\boldsymbol{\mathcal{C}}_{\mathbb{T}}$. Now we assume that $Z$ is a $G$-orbit of a point in $X$.
\begin{lemma}
\label{ZonTinv}
    For \textbf{\emph{any}} finite group $G$ with $\rank\Pic(X)^G = 1$, if $Z$ is a $G$-orbit of a point, then either $|Z|\leq 8$ or $Z$ lies on an invariant curve of tri-degree less than or equal to $(4,4,4)$. In particular, if $G$ is a toric group with $W_G \supset \mathfrak{A}_4$ and $n\geq 4$, then $Z$ lies on $\boldsymbol{\mathcal{C}}_{\mathbb{T}}$.
\end{lemma}
\begin{proof}
    Suppose that $Z$ does not lie on any curve of tri-degree less than or equal to $(4,4,4)$. We will prove that $|Z|\leq 8$ in that case. Observe then that if $Z$ is a $G$-center of non-canonical singularities of $(X,\lambda\M)$ then $Z$ is a $G$-center of non-log canonical singularities of $(X,\frac{3}{2}\lambda\M)$ by \cite[Remark 3.6]{BirRigAlpha}. Let us prove that $Z$ is an isolated $G$-center. Suppose that there is a curve $C$ that is a center of non-log canonical singularities of $(X,\frac{3}{2}\lambda \M)$. Let $M_1,M_2$ be two general surfaces from $\M$. Then $M_1 \cdot M_2 = mC+\Omega$, where $m$ is a non-negative integer and $\Omega$ is an effective one-cycle whose support does not contain components of  $C$. Then $m > \frac{16}{9\lambda^2}$ by Lemma~\ref{ineqCorti}. Now let $H$ be a divisor of tri-degree $(1,1,1)$. Then \[\frac{24}{\lambda^2} = H\cdot M_1 \cdot M_2 = mH \cdot C+H\cdot \Omega\geq m H\cdot C > \frac{16H \cdot C}{9\lambda^2},\] so $\frac{27}{2}>H\cdot C$. Hence, the tri-degree of $C$ is less than or equal to $(4,4,4)$. By our assumptions, $Z$ does not lie on such curves. Thus, $Z$ is an isolated center of non-log canonical singularities of $(X,\frac{3}{2}\lambda \M)$. 
    
    Let $\mu<\lambda$ be a positive rational number such that $Z$ is an isolated $G$-center of log canonical singularities of the log-pair $(X,\frac{3}{2}\mu\M)$. Let $\mathcal{I}$ be the multiplier ideal sheaf of the log-pair $(X,\frac{3}{2}\mu\M)$. Then $\mathcal{I}$ defines a subscheme $Z'$ in $X$ whose support contains $Z$. By Nadel's vanishing theorem~\cite[Theorem 2.3.5]{IcosahedronCheltsovShramov}, we get $H^1(X,\mathcal{I}\otimes\mathcal{O}_X(-\frac{1}{2}K_X))=0$. Now, using the ideal exact sequence \[0\rightarrow \mathcal{I}\otimes\mathcal{O}_X(-\frac{1}{2}K_X)\rightarrow \mathcal{O}_X(-\frac{1}{2}K_X) \rightarrow \mathcal{O}_{Z'}\otimes\mathcal{O}_X(-\frac{1}{2}K_X)\rightarrow0,\] we get $|Z|\leq h^0(X,\mathcal{O}_X(-\frac{1}{2}K_X))=8$, and we are done.
    
    For the case of $n\geq 4$, we see that, by Corollary~\ref{nocurves}, $Z$ can only lie in $\boldsymbol{\mathcal{C}}_{\mathbb{T}}$ or have length less than or equal to 8. But looking at the orbits and using the fact that for even $n$, $G$ always contains either $G_{12\cdot(\nicefrac{n}{2})^3}$ or $G'_{12\cdot(\nicefrac{n}{2})^3}$ up to conjugation, we can see that if $Z$ does not lie in $\boldsymbol{\mathcal{C}}_{\mathbb{T}}$, there exist two possibilities: \[\begin{cases}
        |Z|\geq 3n^2, & \text{if } Z \text{ lies in } \boldsymbol{\mathcal{D}}_{\mathbb{T}} \\
        |Z|\geq \frac{1}{4}n^3, & \text{if } Z \text{ does not lie in } \boldsymbol{\mathcal{D}}_{\mathbb{T}}
    \end{cases},\] and in both cases $|Z|\geq16$ if $n \geq 4$. Hence, we get a contradiction, and $Z$ must lie in $\boldsymbol{\mathcal{C}}_{\mathbb{T}}$.
\end{proof}

\begin{lemma}
\label{ZisTinv}
 $Z$ is $\mathbb{T}$-invariant.  
\end{lemma}
\begin{proof}
    We know by Lemma~\ref{ZonTinv} that $Z$ is contained in $\boldsymbol{\mathcal{C}}_{\mathbb{T}}$. Suppose $Z$ is not $\mathbb{T}$-invariant. Let $H$ be a divisor of tri-degree $(1,1,1)$ on $X$. Then $H \cdot \boldsymbol{\mathcal{C}}_{\mathbb{T}} = 12$.

    Recall that \[G \cap \mathbb{T} = T_G \cong \boldsymbol{\mu}^3_{n},\,\boldsymbol{\mu}^2_{n}\times\boldsymbol{\mu}_{\frac{n}{2}} \text{ or } \boldsymbol{\mu}_n \times \boldsymbol{\mu}^2_{\frac{n}{2}}\] with $n \geq 4$, by assumption. Therefore, since $Z$ is not $\mathbb{T}$-invariant, we get in each case that $|Z|\geq12n$.

    Let $M_1$ and $M_2$ be two general surfaces from $\M$. Then $M_1 \cdot M_2 = m \boldsymbol{\mathcal{C}}_{\mathbb{T}}+\Delta$, where $m$ is a non-negative integer, and $\Delta$ is an effective one-cycle whose support does not contain components of $\boldsymbol{\mathcal{C}}_{\mathbb{T}}$. Then \[\frac{24}{\lambda^2} = H \cdot M_1 \cdot M_2 = H\cdot (m\boldsymbol{\mathcal{C}}_{\mathbb{T}}+\Delta) \geq m H \cdot \boldsymbol{\mathcal{C}}_{\mathbb{T}} =12m,\] hence $m \leq \frac{2}{\lambda^2}$.

    Let $\mathcal{B}$ be a linear subsystem of $|2H|$, consisting of surfaces that pass through $\boldsymbol{\mathcal{C}}_{\mathbb{T}}$. It is easy to see that $\mathcal{B}$ is generated by \[x_2 x_3 y_1^2 y_2 y_3 ,\, x_1 x_3 y_1 y_2^2 y_3 ,\, x_1 x_2 y_1 y_2 y_3^2 ,\, x_1 x_2 x_3 y_1 y_2 y_3 ,\, x_1 x_2 x_3^2 y_1 y_2 ,\, x_1 x_2^2 x_3 y_1 y_3 ,\, x_1^2 x_2 x_3 y_2 y_3. \] Hence, $\boldsymbol{\mathcal{C}}_{\mathbb{T}}$ is the base locus of the linear system $\mathcal{B}$. By Lemma~\ref{ineqPukhlikov}, $\mult_O(M_1\cdot M_2)>\frac{4}{\lambda^2}$ for any point $O \in Z$. Thus, for any point $O \in Z$ we have \[\mult_O(\Delta)>\frac{4}{\lambda^2}-m,\] since the curve $\boldsymbol{\mathcal{C}}_{\mathbb{T}}$ is smooth at $O$. Now let $S$ be a general surface in $\mathcal{B}$. Then \[\frac{48}{\lambda^2}-24m=S\cdot(M_1\cdot M_2-m\boldsymbol{\mathcal{C}}_{\mathbb{T}})=S\cdot \Delta \geq \sum_{O \in Z}\mult_{O}(\Delta)>|Z|(\frac{4}{\lambda^2}-m)\geq12n(\frac{4}{\lambda^2}-m).\] Since $m\leq \frac{2}{\lambda^2}$ and $n \geq 2$, we get \[\frac{24}{\lambda^2}n=\frac{24}{\lambda^2}(n-2)+\frac{48}{\lambda^2}\geq12m(n-2)+\frac{48}{\lambda^2}>\frac{48}{\lambda^2}n,\] which is absurd. Hence, $Z$ must be $\mathbb{T}$-invariant.
\end{proof}

\begin{proof}[Proof of Theorem~\ref{Gbirsupriglarge}]
By Lemmas~\ref{Znotacurve} and~\ref{ZisTinv}, if $Z$ is a curve, it is $\boldsymbol{\mathcal{C}}_{\mathbb{T}}$, and if $Z$ is a $G$-orbit of a point, it must be $\mathbb{T}$-invariant. Then, by Lemma~\ref{faithfulforT}, $\mult_P(\M)>\frac{2}{\lambda}$ for every point $P \in Z$. Thus, the two conditions of Lemma~\ref{sufficientGbirsuprig} are satisfied and we conclude that $\PPP$ is $G$-birationally superrigid for all toric groups $G$ with $W_G \supset \mathfrak{A}_4$ and $n\geq4$.
\end{proof}

\subsection{$G$-birational superrigidity for small groups}
\label{subsecGbirsuprigspecial}
We study the groups from the list~\eqref{eq:small-superrigidity-groups} in a similar manner. However, since for these groups there are many more $G$-invariant curves of small tri-degrees and $G$-orbits of small lengths, these groups require a separate analysis.
\begin{theorem}
\label{Gsuperbirridspecial}
    For the groups in the list~\eqref{eq:small-superrigidity-groups}, $X = \PPP$ is $G$-birationally superrigid.
\end{theorem}
We will prove all the cases simultaneously. 

Applying our strategy, we suppose that there is some $G$-center of non-canonical singularities $Z$ for a log-pair $(X,\lambda \M)$.
\begin{lemma}
\label{Znotacurvespecial}
$Z$ cannot be a $G$-invariant curve for groups in the list~\eqref{eq:small-superrigidity-groups}.
\end{lemma}
\begin{proof}
    Let the $G$-irreducible curve $Z$ be a $G$-center of non-canonical singularities of the log-pair $(X,\lambda\M)$. We know, by Lemma~\ref{Znotacurve}, that $H\cdot Z<24$; hence, $Z$ is either $\boldsymbol{\mathcal{C}}_{\mathbb{T}}$ or one of the curves from Propositions~\ref{Gcurves963}-\ref{Gcurves4849}. We divide the proof into cases according to the type of curve:
    \begin{enumerate}
        \item $Z$ is either $\boldsymbol{\mathcal{C}}_{\mathbb{T}}$ or one of the curves $\boldsymbol{\mathcal{F}}_i$, $\boldsymbol{\mathcal{F}}'_i$, or~$\boldsymbol{\mathcal{F}}_{3,i}$  for some $i$. Then it decomposes as four $(1,0,0)$-curves, four $(0,1,0)$-curves, and four $(0,0,1)$-curves or as six $(1,0,0)$-curves, six $(0,1,0)$-curves, and six $(0,0,1)$-curves. Pick a surface of tri-degree $(1,0,0)$ that contains two $(0,1,0)$- or two $(0,0,1)$-curves. Call those curves $C_1$ and $C_2$. Then $\lambda \M|_S$ is an effective $\mathbb{Q}$-linear system of bi-degree $(2,2)$. Since $Z$ is non-canonical, by Lemma~\ref{noncancurve} we see that \[\mult_{C_1}(\M) = \mult_{C_2}(\M) > \frac{1}{\lambda}.\] Thus, a general curve from $\lambda \M|_S$ contains $C_1$ and $C_2$ with coefficients greater than 1, contradicting the fact that $\lambda \M|_S$ is of bi-degree $(2,2)$.
        
        \item $Z$ is one of the smooth irreducible $(4,4,4)$-curves $\boldsymbol{\mathcal{S}}_i$. It lies on a smooth $(0,2,2)$-surface $S_{0,2,2}$, which is the product of $\mathbb{P}^1$ and a $(2,2)$-curve in $\mathbb{P}^1\times\mathbb{P}^1$. Let $\mathcal{B}\subset|2H|$ be the subsystem passing through $Z$. A direct computation gives $\Bs(\mathcal{B})=Z$, so for $\pi\colon Y=\Bl_ZX\rightarrow X$ with the exceptional divisor $E$, the divisor $-K_Y$ is base-point-free and hence nef. If $\widetilde{F}$ is the proper transform of a general fiber $F$ of $S_{0,2,2}$, then \[-K_Y\cdot\widetilde{F}=(-\pi^*K_X-E)\cdot\widetilde{F}=2-\len(Z\cap F)\leq0,\] since general $F$ intersects $Z$ in two points. Hence, $-K_Y\cdot\widetilde{F}=0$ and $\len(Z\cap F)=2$. On the other hand, \[\lambda\widetilde{\M}\sim_{\mathbb{Q}}2\pi^*H-mE,\] where $m>1$ by Lemma~\ref{noncancurve}, and therefore \[\lambda\widetilde{\M}\cdot\widetilde{F}=2-2m<0.\] This is impossible: since $\widetilde{\M}$ is mobile, it does not contain $\widetilde{S}_{0,2,2}$, while $\widetilde{F}$ moves in a positive-dimensional family on this surface, so its intersection with a general member of $\widetilde{\M}$ is non-negative.

        \item $Z$ is one of the curves $\boldsymbol{\mathcal{R}}_i$, with disjoint smooth $(1,1,1)$-components $\Gamma_1,\ \Gamma_2,\ \Gamma_3,\ \Gamma_4$. The anticanonical divisors through $Z$ form a pencil $\mathcal{B}$ with smooth general member, and $\Bs(\mathcal{B})=Z\cup\mathcal{C}$, where $\mathcal{C}$ is another such $G$-invariant $(4,4,4)$-curve, with components $R_1,\ R_2,\ R_3,\ R_4$. Direct computation shows that each $\Gamma_i$ meets each $R_j$ in two reduced points. Let $S\in\mathcal{B}$ be general. Then $S$ is a K3 surface and \[\lambda\M|_S=aZ+b\mathcal{C}+\Delta,\] where $a>1$ by Lemma~\ref{noncancurve}, $b\geq0$, and $\Delta$ is effective and contains no components of $Z$ or $\mathcal{C}$. Since all $\Gamma_i$ and $R_j$ are rational, adjunction on $S$ gives $\Gamma_i^2=R_j^2=-2$; hence, $Z^2=\mathcal{C}^2=-8$ and $Z\cdot\mathcal{C}=32$. Therefore, for a general $M\in\M|_S$, \[24=-8a+32b+\Delta\cdot Z,\ 24=32a-8b+\Delta\cdot\mathcal{C}.\] Since $\Delta\cdot Z,\ \Delta\cdot\mathcal{C}\geq0$, we obtain \[24+8a-32b\geq0,\ 24-32a+8b\geq0.\] The second inequality gives $b\geq4a-3$, and the first then gives $a\leq1$, a contradiction.
        
        \item $Z$ one of the curve $\boldsymbol{\mathcal{Q}}_{2,i}$ for some $i$. These are the curves that decompose as two smooth and disjoint $(2,2,2)$-curves. Denote the irreducible components of $Z$ by $C_1$ and $C_2$. We can see that $C_1$ and $C_2$ are complete intersections of two $(1,1,1)$-divisors. Thus, we get two pencils of $(1,1,1)$-divisors passing through $C_1$ and $C_2$, respectively. Moreover, it can be seen that those pencils are the full systems of $(1,1,1)$-divisors passing through those curves. Let $\pi$ denote the blow-up of $X$ in $Z$, which is smooth. Then the proper transform of $\lambda \M$ satisfies: \[\widetilde{\lambda \M} \sim_{\mathbb{Q}} 2\pi^*H-m(E_1+E_2),\] where $m>1$ by Lemma~\ref{noncancurve}, and $E_1$ and $E_2$ are the exceptional divisors over $C_1$ and $C_2$, respectively. Meanwhile, the proper transforms of the pencils are $\pi^*H-E_1$ and $\pi^*H-E_2$, and they are base-point-free, hence nef. Thus, since $\M$ is mobile, we get: \[(2\pi^*H-m(E_1+E_2))\cdot (\pi^*H-E_1)\cdot(\pi^*H-E_2)\geq 0.\] Computing the intersection numbers, we see: \[(\pi^*H)^3 = 6,\ (\pi^*H)^2\cdot E_i = 0,\ \pi^*H\cdot E_i^2 = -H\cdot C_i = -6.\] Hence, the inequality becomes: \[(2\pi^*H-m(E_1+E_2))\cdot (\pi^*H-E_1)\cdot(\pi^*H-E_2)=12(1-m)\geq0,\] which yields $m\leq 1$, which is a contradiction.
    \end{enumerate}
    We can see that these cases include all $G$-invariant curves of tri-degree $(a,a,a)$ with $a \leq 7$ for all groups $G$ in the list~\eqref{eq:small-superrigidity-groups}.
\end{proof}

\begin{lemma}
\label{Znotanorbitspecial}
$Z$ cannot be a $G$-orbit of a point for groups in the list~\eqref{eq:small-superrigidity-groups}.
\end{lemma}
\begin{proof}
    Suppose that $Z$ is a $G$-orbit of a point. Then by Lemma~\ref{ZonTinv} we get that either $Z$ lies on a $G$-invariant curve of tri-degree less than or equal to $(4,4,4)$ or $|Z|\leq 8$. But, from Propositions~\ref{orbitdescription}-\ref{orbits4849}, for points not lying on $G$-invariant curves of tri-degree less than or equal to $(4,4,4)$ the orbit sizes are greater than 8, which is a contradiction. Hence, $Z$ must lie on one of the $G$-invariant $G$-irreducible curves of tri-degree less than or equal to $(4,4,4)$. Call this curve $C$. Consider cases:
    \begin{enumerate}
        \item $C$ is either $\boldsymbol{\mathcal{C}}_{\mathbb{T}}$ or the curves $\boldsymbol{\mathcal{F}}'_i$, $\boldsymbol{\mathcal{S}}_i$, or~$\boldsymbol{\mathcal{Q}}_{2,i}$ for some $i$. Then assume $|Z|\geq 12$. Let $M_1$ and $M_2$ be two general surfaces from $\M$. Then $M_1\cdot M_2 = mC + \Delta$, where $m$ is a non-negative integer and $\Delta$ is an effective one-cycle whose support does not contain components of  $C$. Then \[\frac{24}{\lambda^2}=H\cdot M_1 \cdot M_2 = H \cdot (mC+\Delta)\geq 12m,\] hence $m\leq \frac{2}{\lambda^2}$. Now let $\mathcal{B}$ be a linear subsystem of $|2H|$ of surfaces that pass through $C$. We can see that $\Bs(\mathcal{B}) = C$. By Lemma~\ref{ineqPukhlikov}, $\mult_{O}(M_1 \cdot M_2) >\frac{4}{\lambda^2}$ for any point $O \in Z$; hence $\mult_O(\Delta)>\frac{4}{\lambda^2}-m$, since from the description of curves and orbits, we see that all the orbits on $C$ of length $\geq 12$ are smooth. Then let $S$ be a general surface in $\mathcal{B}$. We have: \[ \frac{48}{\lambda^2}-24m= S\cdot(M_1\cdot M_2-mC) = S\cdot \Delta \geq \sum_{O\in Z}\mult_O(\Delta)>|Z|(\frac{4}{\lambda^2}-m).\] Rearranging, we get $m(|Z|-24)>\frac{4}{\lambda^2}(|Z|-12)$. Now there are two cases. If $12\leq|Z|\leq24$, then the left-hand side is non-positive and the right-hand side is non-negative, which is a contradiction. Else if $|Z|\geq 24$, then since $m\leq \frac{2}{\lambda^2}$, we have $\frac{2}{\lambda^2}(|Z|-24)>\frac{4}{\lambda^2}(|Z|-12)$, which simplifies to $\frac{2}{\lambda^2}|Z|>\frac{4}{\lambda^2}|Z|$, which is absurd. Hence, $|Z|<12$.

        \item $C$ is $\boldsymbol{\mathcal{F}}_i$ for some $i$. Then assume $|Z|\geq 24$. Following the same logic as before but for $\mathcal{B} \subset |4H|$, we get a contradiction. Hence, $|Z|<24$.

        \item $C$ is one of the curves $\boldsymbol{\mathcal{R}}_{i}$ for some $i$. Then consider a linear subsystem $\mathcal{B} \subset |2H|$ passing through $C$. By direct computation, $\Bs(\mathcal{B}) = C \cup \mathcal{R}$, where $\mathcal{R}$ is some $\boldsymbol{\mathcal{R}}_{j}$ for $j \neq i$. Suppose $Z \not\subset C\cap \mathcal{R}$. Moreover, suppose that $|Z|\geq 12$. Let $M_1$ and $M_2$ be two general surfaces from $\M$. Then $M_1\cdot M_2 = m_1C+m_2\mathcal{R} + \Delta$, where $m_1,m_2$ are non-negative integers and $\Delta$ is an effective one-cycle whose support does not contain components of  $C$ or $\mathcal{R}$. Then \[\frac{24}{\lambda^2}=H\cdot M_1 \cdot M_2 = H \cdot (m_1C+m_2\mathcal{R} + \Delta)\geq 12(m_1+m_2),\] hence $m_1+m_2\leq \frac{2}{\lambda^2}$. By Lemma~\ref{ineqPukhlikov}, $\mult_{O}(M_1 \cdot M_2) >\frac{4}{\lambda^2}$ for any point $O \in Z$; hence, $\mult_O(\Delta)>\frac{4}{\lambda^2}-m_1$, since $Z$ does not lie on $\mathcal{R}$, and from the description of curves and orbits, we see that all the orbits on $C\cup \mathcal{R}$ that are not in the intersection are smooth. Then let $S$ be a general surface in $\mathcal{B}$. We have: \[ \frac{48}{\lambda^2}-24m_1-24m_2= S\cdot(M_1\cdot M_2-m_1C-m_2\mathcal{R}) = S\cdot \Delta \geq \sum_{O\in Z}\mult_O(\Delta)>|Z|(\frac{4}{\lambda^2}-m_1).\] Rearranging, we get \[m_1(|Z|-24)-24m_2>\frac{4}{\lambda^2}(|Z|-12).\] As in the first case, this leads to a contradiction. Thus, either $Z$ lies in $C \cap \mathcal{R}$ or $|Z|<12$. 

    \end{enumerate}

    Thus, we have eliminated all but finitely many possibilities for $G$-orbits. Namely, we are left with the following cases:
    \begin{enumerate}
        \item $|Z|<12$. Then, from the orbit description, we see that $Z$ is one of the orbits of length 4 or 8. Consider cases:
        \begin{enumerate}
            \item $Z$ is one of the orbits of length $4$. Then, by a direct computation, the tangent space of $X$ at each of those points is an irreducible representation of the stabilizer of the point. By Lemma~\ref{irrepmult} we see that $\mult_{O}(\mathcal{M})>\frac{2}{\lambda}$, for any point $O \in Z$. Pick a surface $S$ of tri-degree $(1,0,0)$ which contains at least two of the points in $Z$; call those points $P_1$ and $P_2$. Let $C$ be a general curve of tri-degree $(1,1)$ which passes through $P_1$ and $P_2$. Such curves span $S$, hence $C$ is not contained in the base locus of $\mathcal{M}|_S$. Thus, for a general member $M \in \mathcal{M}$, we get: \[4 = \lambda M\cdot C\geq\lambda(\mult_{P_1}(\M)+\mult_{P_2}(\M))>4,\] which is a contradiction.
            \item $Z$ is one of the orbits of length $8$ for $G_{48,49}'$. The induced representation on the tangent space is reducible in this case, so we employ another strategy. We will use the alpha-invariant to rule those orbits out. Let $S \cong \mathbb{P}^1 \times \mathbb{P}^1$ be a $(1,0,0)$-surface that contains at least two points of $Z$. Let $K$ be the stabilizer of $S$ in $G$. We can see that it acts faithfully on $S$. By direct computation, there are no $K$-invariant curves of bi-degree $(1,0),\ (0,1)$ and $(1,1)$ in $S$. Then, it follows, from Theorem~\ref{alphaPP}, that $\alpha_{K}(S)\geq1$. This is a contradiction, since the log-pair $(S, \lambda\M|_S)$ is not log-canonical by~\cite[Theorem 5.50]{KollarMori}. Thus, $Z$ cannot be a center of non-canonical singularities of $(X, \lambda \M)$. 
        \end{enumerate}
        \item Either $Z$ lies on $\boldsymbol{\mathcal{F}}_i$ for some $i$ and $|Z|<24$, or $Z$ lies in $\boldsymbol{\mathcal{R}}_{i} \cap \boldsymbol{\mathcal{R}}_{j}$ for some $i$ and $j$ such that $\boldsymbol{\mathcal{R}}_{j}$ is the residual curve of the base locus of the system of divisors of tri-degree $(2,2,2)$ passing through $\boldsymbol{\mathcal{R}}_{i}$. In the latter case, each such $Z$ is a $G$-orbit of length 16. In the former case, we can assume that $Z$ does not lie on any other $G$-invariant curve of tri-degree less than $(4,4,4)$, since those cases are already eliminated. Then $Z$ is one of the orbits of length 12 for $G_{48,49}'$. In either case, let $\mathcal{B}$ be a linear subsystem of $|2H|$ passing through $Z$. Then, by a direct computation, $\Bs(\mathcal{B})$ is zero-dimensional. Hence, for two general surfaces $M_1,M_2$ in $\mathcal{M}$ and a general surface $S$ in $\mathcal{B}$, Lemma~\ref{ineqPukhlikov} implies that \[\frac{48}{\lambda^2}=S \cdot M_1 \cdot M_2\geq \sum_{O \in Z}\mult_{O}(M_1\cdot M_2)>|Z|\cdot \frac{4}{\lambda^2} \geq \frac{48}{\lambda^2},\] which is absurd.
    \end{enumerate}
\end{proof}

\begin{proof}[Proof of Theorem~\ref{Gsuperbirridspecial}]
   By Lemmas~\ref{Znotacurvespecial} and \ref{Znotanorbitspecial}, for any $G$-invariant mobile linear system $\M$, there are no $G$-centers of non-canonical singularities for the log-pair $(X, \lambda \M)$. Therefore, $X$ is $G$-birationally superrigid by Lemma~\ref{mobilesystems}.
\end{proof}

\begin{corollary}
\label{Gbirsuprigrest}
    $\PPP$ is $G$-birationally superrigid for the following groups: 
    \begingroup
    \setlength{\multicolsep}{3pt}
    \setlength{\premulticols}{0pt}
    \setlength{\postmulticols}{0pt}
    \begin{multicols}{3}
    \begin{enumerate}[itemsep=1.5pt, topsep=0pt, parsep=0pt, partopsep=0pt]
        \item $G_{96,229}$
        \item $\widebar{G}_{192,1538}$
        \item $\widebar{G}_{192,1495}$
        \item $\widebar{G}_{96,227}$
        \item $\widebar{G}_{96,227}'$
        \item $G_{192,1538}$
        \item $G_{192,1538}'$
        \item $G_{192,181}$
        \item $G_{96,227}$
        \columnbreak
        \item $G_{96,227}'$
        \item $G_{192,1540}$
        \item $G_{192,197}$
        \item $G_{192,1000}$
        \item $G_{96,229}'$
        \item $G_{96,70}$
        \item $G_{384,20163}$
        \item $G_{384,17948}$
        \item $G_{384,5657}$
        \columnbreak
        \item $G_{384,20100}$
        \item $G_{192,1538}''$
        \item $G_{192,1538}'''$
        \item $G_{192,955}$
        \item $G_{192,955}'$
        \item $G_{96,195}$
        \item toric groups with $n = 3$ and $W_G \supsetneq \mathfrak{A}_4$.
    \end{enumerate}
    \end{multicols}
    \endgroup
\end{corollary}
\begin{proof}
    This follows immediately from Lemma~\ref{overgroupsGbirsuperrig}, since each of the groups in the corollary contains (up to conjugation) a subgroup $F$, for which we have already proved that $\PPP$ is $F$-birationally superrigid. Namely, for $G = G_{96,229}$, $F = G_{48,50}$, and for $G = G_{96,195}$, $F = G_{48,49}'$. All other groups in this list are toric groups with $T_G = \boldsymbol{\mu}_2^3$, $\boldsymbol{\mu}_2^2$ or $\boldsymbol{\mu}_3^3$ and $W_G \neq \mathfrak{A}_4$. For such a group, we can take $F = G \cap (\mathbb{T}\rtimes \mathfrak{A}_4)$. Since the toric part remains the same, $F$ is one of $G_{96,229}$, $G_{96,3}$, $G_{48,50}$, $G_{324,160}$ or $G_{324,160}'$. For $F = G_{96,229}$ we can replace it by $G_{48,50}$ by the previous discussion. For each remaining case, we know that $X$ is $F$-birationally superrigid by Theorem~\ref{Gsuperbirridspecial}. 
\end{proof}

\subsection{Non-$G$-birationally superrigid cases}
\label{subsecnotGsolid}
We now turn our attention to non-$G$-birationally superrigid cases. It turns out that each group gives us non-$G$-solidity of $\PPP$. 

The following lemma is a direct consequence of the relative $G$-MMP.
\begin{lemma}
\label{nonGsolidcriterion}
Let $X = \PPP$, and let $\varphi$ be a dominant $G$-equivariant map from $X$ of one of the following kinds: 
\begin{enumerate}
    \item $\varphi$ is a map to a surface with a rational curve as a general fiber.
    \item $\varphi$ is a map to a curve with a rational surface as a general fiber.
\end{enumerate}
Then $X$ is not $G$-solid.
\end{lemma}
By Lemma~\ref{pencilsnotGsolid}, one of the simplest ways to construct a non-$G$-solid link for $\PPP$ is to utilize $G$-invariant pencils of $|\mathcal{O}_X(1,1,1))|$.
\begin{lemma}
\label{invpencils}
   For every group $G$ in the following list, there exists a $G$-invariant pencil of divisors of tri-degree $(1,1,1)$ such that its base locus is a $G$-invariant curve of tri-degree $(2,2,2)$:
   \begingroup
    \setlength{\multicolsep}{3pt}
    \setlength{\premulticols}{0pt}
    \setlength{\postmulticols}{0pt}
    \begin{multicols}{3}
    \begin{enumerate}[itemsep=1.5pt, topsep=0pt, parsep=0pt, partopsep=0pt]
        \item $G_{12,3}$
        \item $G_{12,3}'$
        \item $G_{24,13}$
        \item $\widebar{G}_{24,12}$
        \item $\widebar{G}_{24,12}'$
        \item $\widebar{G}_{48,48}$
        \item $G_{24,12}$
        \columnbreak
        \item $G_{24,12}'$
        \item $G_{48,48}$
        \item $G_{24,13}'$ 
        \item $G_{24,13}''$
        \item $G_{24,13}'''$ 
        \item $G_{24,13}^{IV}$
        \item $G_{48,49}$
        \columnbreak
        \item $G_{48,48}''$
        \item $G_{48,48}'''$
        \item $G_{48,48}^{IV}$
        \item $G_{48,48}^{V}$
        \item $G_{96,226}$
    \end{enumerate}
    \end{multicols}
    \endgroup
    There are two possible types of curves of tri-degree $(2,2,2)$:
    \begin{enumerate}
        \item Smooth irreducible curve of tri-degree $(2,2,2)$, which is an elliptic curve.
        \item A union of two $(1,0,0)$-, two $(0,1,0)$-, and two $(0,0,1)$-curves, intersecting in a hexagon.
    \end{enumerate}
\end{lemma}
\begin{proof}
    Computed directly by decomposing the representation of $G$ on $H^0(\mathcal{O}(1,1,1))$ into irreducibles.
\end{proof}

\begin{corollary}
\label{notGsolidtoric}
    For each group in Lemma~\ref{invpencils}, $\PPP$ is not $G$-solid.
\end{corollary}
\begin{proof}
    Follows directly from Lemmas~\ref{pencilsnotGsolid} and \ref{invpencils}.
\end{proof}

\begin{theorem}
    \label{notGsolid2412}
        $\PPP$ is not $G_{24,12}''$-solid.
    \end{theorem}
    \begin{proof}
        Consider the curve \[C = G_{24,12}'' \cdot  \left\{\left([x:y], [x-iy : ix-y], [x+y : -ix+iy]\right)\mid [x:y]\in \mathbb{P}^1 \right\}.\] This is a $G$-invariant curve of tri-degree $(2,2,2)$ that decomposes as a union of two disjoint smooth $(1,1,1)$-curves; call its components $C_1$ and $C_2$. Consider the linear system $4H-3C$ of divisors of tri-degree $(4,4,4)$ passing through $C$ with multiplicity 3. We can compute that the system $4H-3C$ can be decomposed as a sum of linear systems of tri-degrees $(2,1,1)$, $(1,2,1)$, $(1,1,2)$ passing through $C$. Call them $\mathcal{B}_{2,1,1}$, $\mathcal{B}_{1,2,1}$, and $\mathcal{B}_{1,1,2}$, respectively. Their descriptions are the following: \begin{enumerate}
        \item $\mathcal{B}_{2,1,1}$ is generated by 
        \begin{align*}
            &x_1^2x_2x_3-x_1^2y_2y_3-(1+i)x_1y_1x_2y_3+(1+i)x_1y_1y_2x_3+iy_1^2x_2x_3-iy_1^2y_2y_3,\\
            &x_1^2x_2y_3+x_1^2y_2x_3+(1-i)x_1y_1x_2x_3+(1-i)x_1y_1y_2y_3-iy_1^2x_2y_3-iy_1^2y_2x_3.
        \end{align*}
        \item $\mathcal{B}_{1,2,1}$ is generated by 
        \begin{align*}
             &x_1x_2^2x_3-(1+i)x_1x_2y_2y_3-ix_1y_2^2x_3-iy_1x_2^2y_3+(1-i)y_1x_2y_2x_3-y_1y_2^2y_3,\\
             &x_1x_2^2y_3+(1-i)x_1x_2y_2x_3+ix_1y_2^2y_3-iy_1x_2^2x_3+(1+i)y_1x_2y_2y_3+y_1y_2^2x_3.
        \end{align*}
        \item $\mathcal{B}_{1,1,2}$ is generated by 
        \begin{align*}
            &x_1x_2x_3^2+ix_1x_2y_3^2-(1-i)x_1y_2x_3y_3+(1-i)y_1x_2x_3y_3+y_1y_2x_3^2+iy_1y_2y_3^2,\\
            &-(1-i)x_1x_2x_3y_3+ix_1y_2x_3^2+x_1y_2y_3^2+iy_1x_2x_3^2+y_1x_2y_3^2+(1-i)y_1y_2x_3y_3.
\end{align*}
        \end{enumerate}
    Each of them is a pencil. Thus, these three linear systems give a rational map $\varphi$ \[\varphi\colon \PPP \dashrightarrow \mathbb{P}_{[u_1:v_1]}^1 \times \mathbb{P}_{[u_2:v_2]}^1 \times \mathbb{P}_{[u_3:v_3]}^1.\] By a direct computation, the image $S$ of $\varphi$ is a $(1,1,1)$-divisor given by the equation \[u_1u_2v_3 - iu_1v_2u_3 + v_1u_2u_3 + iv_1v_2v_3 = 0.\] This defines a smooth surface in $\PPP$, hence it is a del Pezzo surface of degree 6. Moreover, a general fiber of $\varphi$ is rational by a direct computation of a parametrization in coordinates. Hence, $\PPP$ is not $G_{24,12}''$-solid by Lemma~\ref{nonGsolidcriterion}.
\end{proof}
    
\begin{corollary}
\label{notGsolid4848}
    $\PPP$ is not $G_{48,48}'$-solid.
\end{corollary}
\begin{proof}
    We use the notation from the previous theorem. Note that $G_{24,12}''$
    is a subgroup of $G_{48,48}'$. Also note that the curve $C$ from the
    previous theorem is $G_{48,48}'$-invariant. Hence, the map $\varphi$
    is $G_{48,48}'$-equivariant. Thus, $\PPP$ is not $G_{48,48}'$-solid, by Lemma~\ref{nonGsolidcriterion}.
\end{proof}

\begin{theorem}
    \label{notGsolid4830}
    $\PPP$ is not $\widebar{G}_{48,30}$-solid.
\end{theorem}
\begin{proof}
Let $\Theta$ be a twisted cubic in $\mathbb{P}^3$ defined as an intersection of three quadrics: 
\begingroup
\begin{align*}
q_{1}
={}&2z_0^2+i(z_3^2-z_2^2)
-\frac{1}{3}\bigl(\ee_8^{-1}-(1+i)\bigr)(z_0z_3-3z_1z_2)\\
&+\frac{1}{3}\bigl(\ee_8-(i-1)\bigr)(z_0z_2-3z_1z_3)
+\frac{2}{3}\bigl(-1+(\ee_8-\ee_8^{-1})\bigr)z_2z_3,\\
q_{2}
={}&2z_0z_1+\ee_8(z_0z_3+z_1z_2)
-\ee_8^{-1}(z_0z_2+z_1z_3)\\
&+\bigl(-1-(\ee_8-\ee_8^{-1})\bigr)(z_2^2+z_3^2),\\
q_{3}
={}&2z_1^2+i(z_2^2-z_3^2)
+\frac{1}{3}\bigl(\ee_8^{-1}-(1+i)\bigr)(3z_0z_3-z_1z_2)\\
&+\frac{1}{3}\bigl(\ee_8-(i-1)\bigr)(z_1z_3-3z_0z_2)
+\frac{2}{3}\bigl(-1+(\ee_8-\ee_8^{-1})\bigr)z_2z_3.
\end{align*}
\endgroup Consider the curve $C$ of tri-degree $(6,6,6)$ defined as a proper inverse image of the twisted cubic $\Theta$ via \[\varphi\colon\PPP\dashrightarrow\mathbb{P}^3,\ \left([x_1:y_1],[x_2:y_2],[x_3:y_3]\right)\mapsto\left[x_1x_2x_3:x_1y_2y_3:y_1x_2y_3:y_1y_2x_3\right].\] This is a $G$-invariant irreducible curve with four ordinary triple points.

Consider the $G$-invariant linear system $\mathcal{D}$ of divisors of
tri-degree $(9,9,9)$ having multiplicity at least $5$ at a general point of $C$.
By a direct computation, this is a pencil. We describe its associated
rational map in the affine coordinates $x=x_1/y_1$, $y=x_2/y_2$, $z=x_3/y_3$.
Put \[\rho=\ee_8-\ee_8^3,\ \eta=\ee_8+\ee_8^3,\] and set
\[
 U=q_{1}(xyz,x,y,z),\
 V=q_{2}(xyz,x,y,z),\
 W=q_{3}(xyz,x,y,z),
\] which are the equations of $C$ pulled back from the equations of the twisted cubic and written in affine coordinates.
Write \[g(u,v,w)=R_+(u,v,w)R_-(u,v,w),\] where
\[
 R_\pm(u,v,w)=3u^2+(1+\eta)uw+\eta v^2
                 \pm\bigl((i-2\rho)uv+3ivw\bigr).
\]
Put $G_0=g(U,V,W)$ and $G_1=g(W,V,U)$. The map defined by $\mathcal{D}$ is
\[
 \begin{aligned}
 \chi\colon\PPP&\dashrightarrow\mathbb{P}^1_{[\lambda_0:\lambda_1]},\\
 ([x:1],[y:1],[z:1])&\mapsto
 [WG_0-yzVG_1:VG_0-yzUG_1].
 \end{aligned}
\]
Use the affine coordinate $\lambda=\lambda_0/\lambda_1$.
To parametrize a general fiber, put $\delta=uw-v^2$ and
\begin{align*}
 A&=\lambda\bigl((2+\eta)u+(4-\eta)w\bigr)+\tfrac23(1+2\eta)v,\\
 B&=(\rho+4i)u+(2i-\rho)w-\tfrac23(2\rho-i)\lambda v,\\
 \Delta&=(1+2\eta)\left(u^2+\tfrac{4-2\eta}{3}uw
                       +\tfrac{2-2\eta}{9}v^2+w^2\right).
\end{align*}
For general parameters $a,b$, let $[u:v:w]\in\mathbb{P}^2$ be the
intersection of the two lines
\[
 v-\lambda u=a(w-\lambda v),\ b^2(A-B)=a(A+B).
\]
Thus $[u:v:w]$ depends rationally on $\lambda,a,b$.

The parametrization sends $(a,b)$ to $P=\left([x:1],[y:1],[z:1]\right)$, where
\[
 x=\frac{\rho a(w-\lambda v)\Delta}{b\delta(A-B)},\
 y=\frac{R_-(u,v,w)}{\ee_8 bR_-(w,v,u)},\
 z=\frac{\ee_8 bR_+(u,v,w)}{aR_+(w,v,u)}.
\]
By direct substitution, $[U:V:W]=[u:v:w]$ and
$yz=g(U,V,W)/(ag(W,V,U))$. Hence
$\chi(P)=[aW-V:aV-U]=[\lambda:1]$, and the parameters are recovered by
\[
 a=\frac{V-\lambda U}{W-\lambda V},\
 b=\frac{R_-(U,V,W)}{\ee_8 yR_-(W,V,U)}.
\]
Allowing $\lambda$ to vary, these formulas give mutually inverse
birational maps between $\PPP$ and $\mathbb{P}^1\times\mathbb{A}^2$
over $\mathbb{P}^1$. Thus, a general fiber of the $G$-equivariant map $\chi$
is rational, and the result follows from Lemma~\ref{nonGsolidcriterion}.
\end{proof}

\section{Non-toric groups}
\label{nontoricgroups}
\subsection{Preliminaries}
Consider a subgroup $G$ of $\Aut(X) \cong (\PGL_2)^3 \rtimes \mathfrak{S}_3$. There is an exact sequence \[1\rightarrow (\PGL_2)^3 \rightarrow \Aut(X) \xrightarrow{\pi} \mathfrak{S}_3 \rightarrow 1;\] hence, \[1 \rightarrow K \rightarrow G \xrightarrow{\pi |_G} Q \rightarrow 1,\] where $K = \ker (\pi|_G)$ and $Q = \im(\pi|_G)$. We assume $\rank\Pic(X)^G = 1$; hence, $\pi(G) = \boldsymbol{\mu}_3 \text{ or }\mathfrak{S}_3$. Consider the projections $p_i\colon K \rightarrow \PGL_2$ to the respective factors, and call $F_i$ the images. Since $\boldsymbol{\mu}_3 \subset Q$, the group $Q$ acts transitively on the three factors. Since $K$ is normal in $G$, this transitive action sends the projections $F_i$ of $K$ to one another up to conjugation in $\PGL_2$. Hence, $F_i$ are conjugate subgroups of $\PGL_2$. Thus, after conjugation, we can assume $F_1 = F_2 = F_3 = F$, and $K$ is a subdirect product of $F \times F \times F$.

Notice that, since $K\lhd G$, we have $G \leq \Norm_{\Aut(X)}(K)$. Pick an element $(n_1,n_2,n_3;\sigma) \in \Norm_{\Aut(X)}(K)$. Then, for any element $(k_1,k_2,k_3;\id)$ in $K$, we get \[(n_1,n_2,n_3;\sigma)(k_1,k_2,k_3;\id)(n_1,n_2,n_3;\sigma)^{-1} = (n_1k_{\sigma^{-1}(1)}n_1^{-1},n_2k_{\sigma^{-1}(2)}n_2^{-1},n_3k_{\sigma^{-1}(3)}n_3^{-1};\id) \in K.\] Thus, taking the projections, we get $n_iFn_i^{-1} = F$; hence, $n_i$ lies in $\Norm_{\PGL_2}(F)$, and \[G \leq \Norm_{\Aut(X)}(K) \leq (\Norm_{\PGL_2}(F))^3 \rtimes \mathfrak{S}_3.\]

Since $F \leq \PGL_2$, we have, by the classification of finite subgroups of $\PGL_2$, the following possibilities: \[F = \boldsymbol{\mu}_n,\ D_n,\ \mathfrak{A}_4,\ \mathfrak{S}_4, \text{ or } \mathfrak{A}_5.\]
\begin{lemma}
    Let $G$ be a finite subgroup of $\Aut(X)$ such that $\rank\Pic(X)^G = 1$. If $G$ is not toric, then $F = \mathfrak{A}_4$, $\mathfrak{S}_4$, or~$\mathfrak{A}_5$.
\end{lemma}
\begin{proof}
    Consider cases:
    \begin{enumerate}
        \item Suppose $F = \boldsymbol{\mu}_n$ for $n\geq2$. By conjugation, we can assume that $F$ is \[F_0 = \left\langle
        \begin{pmatrix}
            \ee_n & 0\\
            0 & 1
        \end{pmatrix}\right\rangle = \langle f \rangle.\] Then, \[T_0 = \left\{\begin{pmatrix}
            \lambda & 0\\
            0 & 1
        \end{pmatrix} \mid \lambda \in \mathbb{C}^{*}\right\},\] and $F_0$ clearly lies in $T_0$. Now conjugate $G$ such that each $F_i$ coincides with $F_0$. Then we claim that $G$ normalizes the torus $\mathbb{T} = T_0^3$. Indeed, this follows from the fact that $\Norm_{\PGL_2}(F_0) = \Norm_{\PGL_2}(T_0)$, for $n\geq 2$, since we know that $G$ lies in \[(\Norm_{\PGL_2}(F_0))^3 \rtimes \mathfrak{S}_3 = (\Norm_{\PGL_2}(T_0))^3 \rtimes \mathfrak{S}_3 = \Norm_{\Aut(X)}(\mathbb{T}).\]
        \item $F = \{\id\}$. Then $G = \boldsymbol{\mu}_3$ or $\mathfrak{S}_3$. Consider cases:
        \begin{enumerate}
            \item $G \cong \boldsymbol{\mu}_3$. Then we can conjugate $G$ to $\langle (1,1,1;(1,2,3)) \rangle$. This group normalizes the standard torus $\mathbb{T}$.
            \item $G \cong \mathfrak{S}_3$. Then we can conjugate $G$ to either \[\langle (1,1,1;\,(1,2,3)), (1,1,1;\,(1,2))\rangle \text{ or } \langle (1,1,1;\,(1,2,3)), (-1,-1,-1;\,(1,2))\rangle.\] These groups also normalize the standard torus $\mathbb{T}$.
        \end{enumerate}
        Thus, in both cases for $F = \{\id\}$, $G$ normalizes a maximal torus.
        \item $F \cong D_n$ with $n\geq 3$. Then there is a unique normal subgroup $\boldsymbol{\mu}_n$ in $D_n$. Hence, if an element normalizes $D_n$, it must also normalize this subgroup, and $\Norm_{\PGL_2}(D_n) \subset \Norm_{\PGL_2}(\boldsymbol{\mu}_n)$, for $n\geq 3$. Thus, $G$ normalizes the same maximal torus as in the cyclic case.
        \item $F \cong D_2 \cong V_4$. For each factor write \[\mathcal{T}_i = \{T \subset \PGL_2 \mid T \text{ is a maximal torus and } F_i \leq \Norm_{\PGL_2}(T)\}.\] One can readily see that $|\mathcal{T}_i| =3$. Now choose tori $T_i \in \mathcal{T}_i$. Consider cases: 
        \begin{enumerate}
            \item $Q = \boldsymbol{\mu}_3$. Then pick a lift $g = (g_1,g_2,g_3;(1,2,3))$ of $(1,2,3)$. Choose arbitrary $T_3$ in $\mathcal{T}_3$ and define $T_1 = g_1 T_3 g_1^{-1}, T_2 = g_2 T_1 g_2^{-1}$.
            \item $Q = \mathfrak{S}_3$. Then pick lifts $g,h$ of $(1,2,3)$ and $(1,2)$, respectively: \[g = (g_1,g_2,g_3;(1,2,3)),\ h = (h_1,h_2,h_3;(1,2)).\] Then $h_3$ acts on $\mathcal{T}_3$. Since $h^2 = 1$ modulo $K$, this action has order at most 2; hence, it has a fixed point on a three-element set. Pick $T_3$ as this fixed point, and put $T_1 = g_1T_3g_1^{-1},\ T_2 = g_2T_1g_2^{-1}$.
        \end{enumerate}
       We can see then that $G$ normalizes the torus $T_1 \times T_2 \times T_3$.
    \end{enumerate}
Therefore, for any cyclic or dihedral group $F$, the group $G$ is toric. Hence, if $G$ is not toric, then $F$ must be one of three remaining possibilities for a finite subgroup in $\PGL_2$.
\end{proof}
Thus, to understand non-toric subgroups, it is enough to classify subdirect products of $F \times F \times F$ for $F = \mathfrak{A}_4$, $\mathfrak{S}_4$, and~$\mathfrak{A}_5$, and then classify all lifts of $Q = \boldsymbol{\mu}_3$ or~$\mathfrak{S}_3$.

Recall Goursat's lemma for the subdirect products.
\begin{lemma}
    Let $G$ be a subdirect product of $G_1 \times G_2$, that is, the two projections of $G$ onto the factors are surjective. Then there exist normal subgroups $M_i \lhd G_i$ and an isomorphism \[\varphi\colon G_1/M_1 \xrightarrow{\sim} G_2/M_2\] such that $G$ is of the form \[G = \{(g_1,g_2)\in G_1 \times G_2 \mid g_2M_2 =\varphi(g_1M_1)\}.\]
\end{lemma}
Note that if there are two subdirect products of $G_1 \times G_2$ with the same normal subgroups such that the corresponding automorphisms differ by a inner automorphism of $G_1/M_1$ or $G_2/M_2$, then such subdirect products are conjugate in $G_1 \times G_2$.

Denote $N_1 = \ker p_{23}|_{K}$, where $p_{23}$ is the projection \[\PGL_2 \times \PGL_2 \times \PGL_2 \rightarrow \PGL_2 \times \PGL_2\] to the second and the third factors, and likewise $N_2 = \ker p_{13}|_{K},\ N_3 = \ker p_{12}|_{K}$. Then $p_i(N_i) \lhd F$. Now, since $Q$ is transitive, we see that $p_i(N_i)$ are conjugate by $\Norm_{\Aut(X)}(F)$. Since each of $p_i(N_i)$ is a normal subgroup in $F$, and for $F = \mathfrak{A}_4$, $\mathfrak{S}_4$, and~$\mathfrak{A}_5$, all normal subgroups are characteristic, we get that \[p_1(N_1) = p_2(N_2) = p_3(N_3) = N.\] In particular, \[N \times N \times N \lhd K.\]

 We require the following form of Goursat's lemma for the triple subdirect product.
 \begin{lemma}
 \label{threeGoursat}
     Let $R$ be a finite group and let $L$ be a subdirect product of $R^3$ such that \[L \cap (R \times 1 \times 1) = L \cap (1 \times R \times 1) = L \cap (1 \times 1 \times R) = 1.\] Then the projection $p_{12}|_L\colon L \rightarrow R^2$ on the first two factors is injective. Its image $P$ is a subdirect product of $R^2$, and \[L = \{(x,y,f(x,y)) \in R^3 \mid (x,y) \in P\},\] for a surjective homomorphism $f\colon P \rightarrow R$. Moreover, the restrictions of $f$ to the kernels of two projections $P \rightarrow R$ are injective.
 \end{lemma}
 \begin{proof}
First, note that $p_{12}|_{L}$ is indeed injective, since \[\ker p_{12}|_L = L \cap(1 \times 1 \times R) = 1.\] Then, define $f$ as $f = p_3 \circ (p_{12}|_L)^{-1}$. Hence, we only need to prove the last statement. Denote by $\pi_i$ the projections in question. Consider $\pi_1$; the proof for $\pi_2$ is identical. We compute $\ker( f|_{\ker\pi_1})$. Suppose $(1,y)$ lies in $\ker( f|_{\ker\pi_1})$. Then, $f(1,y) = 1$; hence, $(1,y,1)$ lies in $L$. But $L \cap (1 \times R \times 1) = 1$. Thus, $(1,y) = (1,1)$, and $f|_{\ker\pi_1}$ is injective.
 \end{proof}

Let us now talk about how we classify the lifts. Since we know that $G \leq \Norm_{\Aut(X)}(K)$, it is enough to classify subgroups of $\Norm_{\Aut(X)}(K)$ such that $K$ lies in them. This is equivalent to the classification of the subgroups of $\Norm_{\Aut(X)}(K)/K$. Thus, the general algorithm is:
\begin{enumerate}
    \item For each group $F = \mathfrak{A}_4$, $\mathfrak{S}_4$, and $\mathfrak{A}_5$, classify all subdirect products $K$ of $F \times F \times F$ that are normalized by some element of the form $(g_1,g_2,g_3;\,(1,2,3))$, up to conjugation by elements of $\Norm_{\PGL_2}(F)^3 \rtimes \mathfrak{S}_3$. To do so:
    \begin{enumerate}
        \item Use Lemma~\ref{threeGoursat} with $R = F/N$ for all possible normal subgroups $N$ of $F$.
        \item Classify all subdirect products $P$ of $R^2$, up to conjugation by elements of $\Norm_{\PGL_2}(F)^2 \rtimes \boldsymbol{\mu}_2$, via Goursat's lemma.
        \item Classify all homomorphisms $f\colon P \rightarrow R$ that satisfy Lemma~\ref{threeGoursat}, and construct all possible groups $L$, up to conjugation by elements of $\Norm_{\PGL_2}(F)^3 \rtimes \mathfrak{S}_3$. All groups $K$ are then obtained as the inverse images of these groups $L$.
        \item Finally, keep precisely those groups $K$ that are normalized by some element of the form $(g_1,g_2,g_3;\,(1,2,3))$ in $\Norm_{\PGL_2}(F)^3 \rtimes \mathfrak{S}_3$.
    \end{enumerate}
    \item For each group $K$, compute its normalizer in $\Aut(X)$, which is the same as the normalizer of $K$ in $\Norm_{\PGL_2}(F)^3 \rtimes \mathfrak{S}_3$.
    \item For each group $K$, classify, up to conjugation, all subgroups of $\Norm_{\Aut(X)}(K)/K$ that project isomorphically onto the subgroups $\boldsymbol{\mu}_3$ or $\mathfrak{S}_3$ of the permutation part $\mathfrak{S}_3$ of $\Aut(X)$.
\end{enumerate}
We now only list the results of the classification, omitting the details for brevity.
\subsection{Classification of non-toric groups}
Denote the following elements in $\PGL_2$: \[a = 
 \begin{pmatrix}
    -1 & 0\\
    0 & 1
 \end{pmatrix},\ b = 
 \begin{pmatrix}
    1 & -i\\
    1 & i
 \end{pmatrix},\ w =
 \begin{pmatrix}
    0 & i\\
    1 & 0
 \end{pmatrix},\ d = 
 \begin{pmatrix}
    0 & -1\\
    1 & 0
 \end{pmatrix},\ e = 
 \begin{pmatrix}
    \ee_{6} & \varphi\\
    0 & \ee^{-1}_{6}
 \end{pmatrix},\ c =bab^{-1},\] where \[\varphi = \frac{1+\sqrt{5}}{2}.\] Then $\mathfrak{A}_4 = \langle a,b \rangle,\ \mathfrak{S}_4=\langle a,b,w\rangle,\ \mathfrak{A}_5 = \langle d,e \rangle$, with the following relations: \[a^2 = b^3 = (ab)^3 = 1,\ waw^{-1} = a,\ wbw^{-1} = b^{-1},\ d^2 = e^3 = (de)^5=1.\] We know that \[\Norm_{\PGL_2}(\mathfrak{A}_4) = \Norm_{\PGL_2}(\mathfrak{S}_4)= \mathfrak{S}_4,\ \Norm_{\PGL_2}(\mathfrak{A}_5) = \mathfrak{A}_5.\] Also put $V_4 = \langle a,c \rangle \leq \mathfrak{A}_4$.
 
 Denote \[R= (1,1,1;\,(1,2,3)),\ T = (1,1,1;\,(1,2)),\ R_b = (1,1,b;\,(1,2,3)),\] \[T_w = (w,w,w;\,(1,2)),\ T_{bw} = (w,w,bw;\, (1,2)).\]
\begin{proposition}
For $F=\mathfrak{A}_4$, the kernels $K$, the normalizers $\Norm_{\Aut(X)}(K)$, and the corresponding groups $G$ are the following:
\begingroup\footnotesize
\begin{longtblr}[
theme = mytheme,
entry = none,
label = none,
]{
stretch = 1.25,
colspec = {|Q[l,m,co=1]|Q[l,m,co=1]|Q[l,m,co=1]|}, 
rowhead = 1,
}
\hline
$K$ & $\Norm_{\Aut(X)}(K)$ & $G$\\
\hline
\hline
$K_{12}=\Delta\mathfrak{A}_4$
& $\Delta\mathfrak{S}_4 \times \mathfrak{S}_3$
& {$G_{12,R}=\langle K_{12},R\rangle$\\
$G_{12,R,T}=\langle K_{12},R,T\rangle$\\
$G_{12,R,T_w}=\langle K_{12},R,T_w\rangle$}\\
\hline
{$K_{48}=\langle (a,1,a),(1,a,a),(b,b,b)\rangle$}
& $\langle K_{48},(w,w,w)\rangle\rtimes\mathfrak{S}_3$
& {$G_{48,R}=\langle K_{48},R\rangle$\\
$G_{48,R,T}=\langle K_{48},R,T\rangle$\\
$G_{48,R,T_w}=\langle K_{48},R,T_w\rangle$}\\
\hline
$K_{192}=\langle V_4^3,(b,b,b)\rangle$
& $\langle \mathfrak{A}_4^3,(w,w,w)\rangle\rtimes\mathfrak{S}_3$
& {$G_{192,R}=\langle K_{192},R\rangle$\\
$G_{192,R_b}=\langle K_{192},R_b\rangle$\\
$G_{192,R,T}=\langle K_{192},R,T\rangle$\\
$G_{192,R,T_w}=\langle K_{192},R,T_w\rangle$\\
$G_{192,R_b,T_{bw}}=\langle K_{192},R_b,T_{bw}\rangle$}\\
\hline
{$K_{576}=\langle V_4^3,(b,b^{-1},1),(1,b,b^{-1})\rangle$}
& $\langle \mathfrak{A}_4^3,(w,w,w)\rangle\rtimes\mathfrak{S}_3$
& {$G_{576,R}=\langle K_{576},R\rangle$\\
$G_{576,R_b}=\langle K_{576},R_b\rangle$\\
$G_{576,R,T}=\langle K_{576},R,T\rangle$\\
$G_{576,R,T_w}=\langle K_{576},R,T_w\rangle$\\
$G_{576,R_b,T_w}=\langle K_{576},R_b,T_w\rangle$}\\
\hline
$K_{1728}=\mathfrak{A}_4^3$
& $\mathfrak{S}_4^3\rtimes\mathfrak{S}_3$
& {$G_{1728,R}=\langle K_{1728},R\rangle$\\
$G_{1728,R,T}=\langle K_{1728},R,T\rangle$\\
$G_{1728,R,T_w}=\langle K_{1728},R,T_w\rangle$}\\
\hline
\end{longtblr}
\endgroup
\end{proposition}

\begin{proposition}
For $F=\mathfrak{S}_4$, the kernels $K$, the normalizers $\Norm_{\Aut(X)}(K)$, and the corresponding groups $G$ are the following:
\begingroup\footnotesize
\begin{longtblr}[
theme = mytheme,
entry = none,
label = none,
]{
stretch = 1.25,
colspec = {|Q[l,m,co=1]|Q[l,m,co=1]|Q[l,m,co=1]|}, 
rowhead = 1,
}
\hline
$K$ & $\Norm_{\Aut(X)}(K)$ & $G$\\
\hline
\hline
$K_{24}=\Delta\mathfrak{S}_4$
& $K_{24}\rtimes\mathfrak{S}_3$
& {$G_{24,R}=\langle K_{24},R\rangle$\\
$G_{24,R,T}=\langle K_{24},R,T\rangle$}\\
\hline
{$K_{96}=\langle (a,a,1),(1,a,a),(c,c,1),$\\
$(1,c,c),(b,b,b),(w,w,w)\rangle$}
& $K_{96}\rtimes\mathfrak{S}_3$
& {$G_{96,R}=\langle K_{96},R\rangle$\\
$G_{96,R,T}=\langle K_{96},R,T\rangle$}\\
\hline
$K_{384}=\langle V_4^3,(b,b,b),(w,w,w)\rangle$
& $K_{384}\rtimes\mathfrak{S}_3$
& {$G_{384,R}=\langle K_{384},R\rangle$\\
$G_{384,R,T}=\langle K_{384},R,T\rangle$}\\
\hline
{$K_{1152}=\langle V_4^3,(b,b^{-1},1),(1,b,b^{-1}),$\\$(w,w,w)\rangle$}
& $K_{1152}\rtimes\mathfrak{S}_3$
& {$G_{1152,R}=\langle K_{1152},R\rangle$\\
$G_{1152,R,T}=\langle K_{1152},R,T\rangle$}\\
\hline
$K_{3456}=\langle \mathfrak{A}_4^3,(w,w,w)\rangle$
& $\mathfrak{S}_4^3\rtimes\mathfrak{S}_3$
& {$G_{3456,R}=\langle K_{3456},R\rangle$\\
$G_{3456,R,T}=\langle K_{3456},R,T\rangle$}\\
\hline
{$K_{6912}=\langle \mathfrak{A}_4^3,(w,w,1),(1,w,w)\rangle$}
& $\mathfrak{S}_4^3\rtimes\mathfrak{S}_3$
& {$G_{6912,R}=\langle K_{6912},R\rangle$\\
$G_{6912,R,T}=\langle K_{6912},R,T\rangle$\\
$G_{6912,R,T_w}=\langle K_{6912},R,T_w\rangle$}\\
\hline
$K_{13824}=\mathfrak{S}_4^3$
& $\mathfrak{S}_4^3\rtimes\mathfrak{S}_3$
& {$G_{13824,R}=\langle K_{13824},R\rangle$\\
$G_{13824,R,T}=\langle K_{13824},R,T\rangle$}\\
\hline
\end{longtblr}
\endgroup
\end{proposition}

\begin{proposition}
For $F=\mathfrak{A}_5$, the kernels $K$, the normalizers $\Norm_{\Aut(X)}(K)$, and the corresponding groups $G$ are the following:
\begingroup\footnotesize
\begin{longtblr}[
theme = mytheme,
entry = none,
label = none,
]{
stretch = 1.25,
colspec = {|Q[l,m,co=1]|Q[l,m,co=1]|Q[l,m,co=1]|}, 
rowhead = 1,
}
\hline
$K$ & $\Norm_{\Aut(X)}(K)$ & $G$\\
\hline
\hline
$K_{60}=\Delta\mathfrak{A}_5$
& $K_{60}\rtimes\mathfrak{S}_3$
& {$G_{60,R}=\langle K_{60},R\rangle$\\
$G_{60,R,T}=\langle K_{60},R,T\rangle$}\\
\hline
$K_{216000}=\mathfrak{A}_5^3$
& $K_{216000}\rtimes\mathfrak{S}_3$
& {$G_{216000,R}=\langle K_{216000},R\rangle$\\
$G_{216000,R,T}=\langle K_{216000},R,T\rangle$}\\
\hline
\end{longtblr}
\endgroup
\end{proposition}

\subsection{$G$-birational behaviour for non-toric groups}
Now let us consider the $G$-birational behaviour for non-toric groups. First, we can utilize toric subgroups of non-toric groups for which we have already proved $G$-birational superrigidity.
\begin{theorem}
\label{Gbirsuprignontoric}
For the groups in the following list, $\PPP$ is $G$-birationally superrigid:
\begingroup
\setlength{\multicolsep}{3pt}
\setlength{\premulticols}{0pt}
\setlength{\postmulticols}{0pt}
\begin{multicols}{3}
\begin{enumerate}[itemsep=1.5pt, topsep=0pt, parsep=0pt, partopsep=0pt]
    \item $G_{48,R}$;
    \item $G_{48,R,T}$;
    \item $G_{48,R,T_w}$;
    \item $G_{192,R}$;
    \item $G_{192,R,T}$;
    \item $G_{192,R,T_w}$;
    \item $G_{576,R}$;
    \item $G_{576,R,T}$;
    \item $G_{576,R,T_w}$;
    \columnbreak
    \item $G_{1728,R}$;
    \item $G_{1728,R,T}$;
    \item $G_{1728,R,T_w}$;
    \item $G_{96,R}$;
    \item $G_{96,R,T}$;
    \item $G_{384,R}$;
    \item $G_{384,R,T}$;
    \item $G_{1152,R}$;
    \item $G_{1152,R,T}$;
    \columnbreak
    \item $G_{3456,R}$;
    \item $G_{3456,R,T}$;
    \item $G_{6912,R}$;
    \item $G_{6912,R,T}$;
    \item $G_{6912,R,T_w}$;
    \item $G_{13824,R}$;
    \item $G_{13824,R,T}$;
    \item $G_{216000,R}$;
    \item $G_{216000,R,T}$.
\end{enumerate}
\end{multicols}
\endgroup
\end{theorem}
\begin{proof}
    Indeed, note first that the toric group $G_{48,50}$ lies in $G_{48,R}$. Thus, $X = \PPP$ is $G_{48,R}$-birationally superrigid by Lemma~\ref{overgroupsGbirsuperrig} and Theorem~\ref{Gsuperbirridspecial}. Now, each subgroup in the list, except $G_{216000,R}$ and $G_{216000,R,T}$, contains $G_{48,R}$. For $G = G_{216000,R}$ and $G_{216000,R,T}$, since $\mathfrak{A}_5$ contains $\mathfrak{A}_4$, and since any two subgroups in $\PGL_2$ isomorphic to $\mathfrak{A}_4$ are conjugate in $\PGL_2$, there is a subgroup in $G$ conjugate to $G_{48,R}$. Thus, for each group in the list, there is a subgroup that is conjugate to $G_{48,R}$. Since we have already established that $X$ is $G_{48,R}$-birationally superrigid, $X$ is $G$-birationally superrigid for any remaining group $G$ in the list by Lemma~\ref{overgroupsGbirsuperrig}.
\end{proof}
Therefore, we are left with the following non-toric groups, which we divide into two sets: 
\begin{enumerate}
    \item Groups in the list
    \begin{equation}
    \label{diagonalgroups}
        \left\{G_{12,R},\ G_{12,R,T},\ G_{12,R,T_w},\ G_{24,R},\ G_{24,R,T},\ G_{60,R},\ G_{60,R,T}\right\},
    \end{equation} which we will call \say{diagonal}, since they all lie in $\Delta\PGL_2 \times \mathfrak{S}_3$.
    \item Groups in the list
    \begin{equation}
    \label{192groups}
    \left\{G_{192,R_b},\ G_{192,R_b,T_{bw}},\ G_{576,R_b},\ G_{576,R_b,T_{w}}\right\},
    \end{equation} which all contain $G_{192,R_b}$ as a subgroup.
\end{enumerate}
\begin{theorem}
\label{notGsoliddiagonal}
    For the groups in the list~\eqref{diagonalgroups}, $X = \PPP$ is not $G$-solid.
\end{theorem}
\begin{proof}
Consider the diagonal curve \[C = \{([x:y],[x:y],[x:y])\mid [x:y] \in \PP\}.\] It is invariant under $\Delta \PGL_2 \times \mathfrak{S}_3$; hence, it is invariant for every group in the list~\eqref{diagonalgroups}. Consider the linear system $U$ of divisors of tri-degree $(1,1,1)$ passing through the curve $C$. We can choose the following basis in $U$: \[x_3 y_1 y_2-x_1 y_2 y_3,\ x_2 y_1 y_3-x_1 y_2 y_3,\ x_2 x_3 y_1-x_1 x_2 y_3,\ x_1 x_3 y_2-x_1 x_2 y_3.\] This defines a rational map \[\varphi\colon \PPP \dashrightarrow \mathbb{P}^3.\] This is a classical link; its inverse is the blow-up of three skew lines followed by a contraction of the strict transform of the unique quadric containing those lines. Denote the coordinates in $\mathbb{P}^3$ as \[[z_0:z_1:z_2:z_3].\] Then consider the following two skew lines: \[L_1 \colon \{z_0 + \ee_3^2z_1= z_2 + \ee_3z_3=0\},\ L_2 \colon \{z_0 + \ee_3z_1= z_2 + \ee_3^2z_3=0\}.\] We can check that those lines are invariant under the action of $\Delta\PGL_2 \times \boldsymbol{\mu}_3$ and are permuted by the element $(1,1,1;(1,2))$. Hence, each group in the list~\eqref{diagonalgroups} has either an invariant line or an invariant pair of skew lines. For such groups, $\mathbb{P}^3$ is always non-$G$-solid. Indeed, for the fixed line, we have the link 
\[\begin{tikzcd}
	& Y & \\
	{\mathbb{P}^3} && {\mathbb{P}^1}
	\arrow["\pi"', from=1-2, to=2-1]
	\arrow["\xi", from=1-2, to=2-3]
	\arrow["\psi", dashed, from=2-1, to=2-3]
\end{tikzcd}\] where $\pi$ is the blow-up of the line, $\xi$ is a $\mathbb{P}^2$-bundle, and $\psi$ is the projection from the line. For the pair of skew lines, we have a link 
\[\begin{tikzcd}
	& Y & \\
	{\mathbb{P}^3} && {\mathbb{P}^1 \times \mathbb{P}^1}
	\arrow["\pi"', from=1-2, to=2-1]
	\arrow["\xi", from=1-2, to=2-3]
	\arrow["\psi", dashed, from=2-1, to=2-3]
\end{tikzcd}\] where $\pi$ is the blow-up of the skew lines, $\xi$ is a $\mathbb{P}^1$-bundle, and $\psi$ is given by the two linear systems of planes passing through each of the lines. Thus, for each group $G$ in the list~\eqref{diagonalgroups}, $\PPP$ is not $G$-solid.
\end{proof}
\begin{theorem}
\label{Gbirsuprig192}
    For the groups in the list~\eqref{192groups}, $X = \PPP$ is $G$-birationally superrigid.
\end{theorem}
\begin{proof}
Let $G = G_{192,R_b}$. Suppose that $Z$ is a $G$-irreducible center of non-canonical singularities for the pair $(X, \lambda \mathcal{M})$, where $\mathcal{M}$ is a mobile $G$-invariant linear system and $\lambda$ is a positive number such that $-K_X \sim_{\mathbb{Q}}\lambda \mathcal{M}$. Then $Z$ is either a $G$-irreducible curve or a $G$-orbit of a point. In the first case, we know, by Lemma~\ref{Znotacurve}, that $H\cdot Z<24$, hence $Z$ is a curve of tri-degree $(a,a,a)$ with $a \leq 7$. Consider a general divisor $S$ of tri-degree $(1,0,0)$. Its stabilizer $B$ in $G$ is a group of order 16, generated by \[(1,-1,1),\ (1,1,-1),\ \left(
   1,\iota(1),1\right),\ \left(
   1,1,\iota(1)\right),\] and acts on $S$ faithfully. For any $G$-invariant curve $Z$ of tri-degree $(a,a,a)$ with $a\leq 7$ that does not lie in $S$, we have $S \cdot Z = a \leq 7$. Hence, since $|B| = 16$ and since the linear system $|S|$ is base-point-free, $Z$ must lie in the locus $\Sigma$ of degenerate orbits of $B$. A computation shows that there are no such curves in $\Sigma$. 
    
    Therefore, there are no $G$-invariant curves of tri-degree $(a,a,a)$ with $a \leq 7$, and $Z$ cannot be a curve. So suppose now that $Z$ is a $G$-orbit. Then by Lemma~\ref{ZonTinv} we get that either $Z$ lies on a $G$-invariant curve of tri-degree less than or equal to $(4,4,4)$ or $|Z|\leq 8$. Since there are no invariant curves of the required length, $Z$ must be of length 8 or less. However, there are no orbits of length $\leq 8$ on $X$. Indeed, the group $G_{192,R_b}$ has index $[576,8661]$ in the SmallGroups library. This group has a unique subgroup of index 3, which corresponds to the group $K_{192}$, and there are no other subgroups of index less than or equal to 8, apart from $G_{192,R_b}$ and $K_{192}$. Hence, the only possible orbits of length $\leq8$ are of length 1 or 3; both must come from fixed points of $K_{192}$. But we can see that $K_{192}$ does not have fixed points on $X$. Thus, there are no orbits of length $\leq 8$, and there are no other possibilities for $Z$. Therefore, $X$ is $G_{192,R_b}$-birationally superrigid.

    For other groups in the list~\eqref{192groups}, $G$-birational superrigidity follows from Lemma~\ref{overgroupsGbirsuperrig}, since each of these groups contains $G_{192,R_b}$ as a subgroup.
\end{proof}
\begin{proof}[Proof of Theorem~\ref{maintheorem}]
By Theorem~\ref{GminimalGsolid}, if $G$ is toric and $\PPP$ is $G$-solid, then
$W_G\supset\mathfrak{A}_4$. Hence, if $G$ is toric and
$W_G\not\supset\mathfrak{A}_4$, then $\PPP$ is not $G$-solid.

It remains to consider toric groups with $W_G\supset\mathfrak{A}_4$ and
non-toric groups. The results of Sections~5 and~6 exhaust all such groups:
each of them is either not $G$-solid or $G$-birationally superrigid.
Among the groups for which $\PPP$ is not $G$-solid, every group either
has a $G$-invariant curve of one of the types listed in the statement or a $G$-invariant pencil in $|\mathcal{O}_X(1,1,1)|$ or
is isomorphic to
\[
\widebar{G}_{48,30}\cong\mathfrak{A}_4\rtimes\boldsymbol{\mu}_4.
\]
On the other hand, none of the groups for which $\PPP$ is
$G$-birationally superrigid has a $G$-invariant curve of tri-degree less
than $(4,4,4)$, and none of them is isomorphic to
$\mathfrak{A}_4\rtimes\boldsymbol{\mu}_4$. As they are $G$-birationally superrigid, they also don't admit $G$-invariant pencils in $|\mathcal{O}_X(1,1,1)|$. Since every curve appearing in the statement has tri-degree strictly less than $(4,4,4)$, this gives
both implications and proves the theorem.
\end{proof}
\appendix
\section*{Appendix A. The final tables}
\label{appendixA}
Here we list the results of the analysis and where they are proved. In the table, \say{superrigid} means $G$-birationally superrigid, and similarly for $G$-birational rigidity and $G$-solidity. For all toric groups we assume that $W_G$ contains $\mathfrak{A}_4$.
\begingroup \scriptsize
\begin{longtblr}[
theme = mytheme,
entry = none,
label = none,
]{
stretch = 1.25,
colspec = {|Q[c,m,co=1.5]|Q[c,m,co=1.5]|Q[c,m,co=1]|Q[c,m,co=1.5]|Q[c,m,co=1.5]|Q[c,m,co=1]|Q[c,m,co=1.5]|Q[c,m,co=1.5]|Q[c,m,co=1]|},
rowhead = 1,
row{even} = {myblue},
row{odd} = {white},
row{1} = {white},
}
\hline
Group & Result & Proof & Group & Result & Proof & Group & Result & Proof\\
\hline
\hline
$G_{12,3}$ & not solid & Cor~\ref{notGsolidtoric} & $G_{12,3}'$ & not solid & Cor~\ref{notGsolidtoric} & $G_{96,229}$ & superrigid & Cor~\ref{Gbirsuprigrest}\\
\hline
$G_{96,3}$ & superrigid & Thm~\ref{Gsuperbirridspecial} & $G_{48,50}$ & superrigid & Thm~\ref{Gsuperbirridspecial} & $G_{24,13}$ & not solid & Cor~\ref{notGsolidtoric}\\
\hline
$G_{24,13}'$ & not solid & Cor~\ref{notGsolidtoric} & $G_{324,160}$ & superrigid & Thm~\ref{Gsuperbirridspecial} & $G_{324,160}'$ & superrigid & Thm~\ref{Gsuperbirridspecial}\\
\hline
$\widebar{G}_{24,12}$ & not solid & Cor~\ref{notGsolidtoric} & $\widebar{G}_{24,12}'$ & not solid & Cor~\ref{notGsolidtoric} & $\widebar{G}_{192,1538}$ & superrigid & Cor~\ref{Gbirsuprigrest}\\
\hline
$\widebar{G}_{192,1495}$ & superrigid & Cor~\ref{Gbirsuprigrest} & $\widebar{G}_{96,227}$ & superrigid & Cor~\ref{Gbirsuprigrest} & $\widebar{G}_{96,227}'$ & superrigid & Cor~\ref{Gbirsuprigrest}\\
\hline
$\widebar{G}_{48,48}$ & not solid & Cor~\ref{notGsolidtoric} & $\widebar{G}_{48,30}$ & {not solid} & {Thm~\ref{notGsolid4830}} & $G_{24,12}$ & not solid & Cor~\ref{notGsolidtoric}\\
\hline
$G_{24,12}'$ & not solid & Cor~\ref{notGsolidtoric} & $G_{24,12}''$ & not solid & Thm~\ref{notGsolid2412} & $G_{192,1538}$ & superrigid & Cor~\ref{Gbirsuprigrest}\\
\hline
$G_{192,1538}'$ & superrigid & Cor~\ref{Gbirsuprigrest} & $G_{192,181}$ & superrigid & Cor~\ref{Gbirsuprigrest} & $G_{96,227}$ & superrigid & Cor~\ref{Gbirsuprigrest}\\
\hline
$G_{96,227}'$ & superrigid & Cor~\ref{Gbirsuprigrest} & $G_{48,48}$ & not solid & Cor~\ref{notGsolidtoric} & $G_{48,48}'$ & not solid & Cor~\ref{notGsolid4848}\\
\hline
$G_{24,13}''$ & not solid & Cor~\ref{notGsolidtoric} & $G_{24,13}'''$ & not solid & Cor~\ref{notGsolidtoric} & $G_{24,13}^{IV}$ & not solid & Cor~\ref{notGsolidtoric}\\
\hline
$G_{192,1540}$ & superrigid & Cor~\ref{Gbirsuprigrest} & $G_{192,197}$ & superrigid & Cor~\ref{Gbirsuprigrest} & $G_{192,1000}$ & superrigid & Cor~\ref{Gbirsuprigrest}\\
\hline
$G_{96,229}'$ & superrigid & Cor~\ref{Gbirsuprigrest} & $G_{96,70}$ & superrigid & Cor~\ref{Gbirsuprigrest} & $G_{48,49}$ & not solid & Cor~\ref{notGsolidtoric}\\
\hline
$G_{48,49}'$ & superrigid & Thm~\ref{Gsuperbirridspecial} & $G_{48,48}''$ & not solid & Cor~\ref{notGsolidtoric} & $G_{48,48}'''$ & not solid & Cor~\ref{notGsolidtoric}\\
\hline
$G_{48,48}^{IV}$ & not solid & Cor~\ref{notGsolidtoric} & $G_{48,48}^{V}$ & not solid & Cor~\ref{notGsolidtoric} & $G_{384,20163}$ & superrigid & Cor~\ref{Gbirsuprigrest}\\
\hline
$G_{384,17948}$ & superrigid & Cor~\ref{Gbirsuprigrest} & $G_{384,5657}$ & superrigid & Cor~\ref{Gbirsuprigrest} & $G_{384,20100}$ & superrigid & Cor~\ref{Gbirsuprigrest}\\
\hline
$G_{192,1538}''$ & superrigid & Cor~\ref{Gbirsuprigrest} & $G_{192,1538}'''$ & superrigid & Cor~\ref{Gbirsuprigrest} & $G_{192,955}$ & superrigid & Cor~\ref{Gbirsuprigrest}\\
\hline
$G_{192,955}'$ & superrigid & Cor~\ref{Gbirsuprigrest} & $G_{96,226}$ & not solid & Cor~\ref{notGsolidtoric} & $G_{96,195}$ & superrigid & Cor~\ref{Gbirsuprigrest}\\
\hline
{toric groups\\ with $n = 3$\\ and $W_G \supsetneq \mathfrak{A}_4$} & superrigid & Cor~\ref{Gbirsuprigrest} & {toric groups\\ with $n\geq 4$} & superrigid & Thm~\ref{Gbirsupriglarge} & & &\\
\hline
\end{longtblr}
\endgroup

Here is the analogous table for non-toric groups. We assume that $\rank\Pic(X)^G = 1$.
\begingroup\scriptsize
\begin{longtblr}[
theme = mytheme,
entry = none,
label = none,
]{
stretch = 1.25,
colspec = {|Q[c,m,co=1.5]|Q[c,m,co=1.5]|Q[c,m,co=1]|Q[c,m,co=1.5]|Q[c,m,co=1.5]|Q[c,m,co=1]|Q[c,m,co=1.5]|Q[c,m,co=1.5]|Q[c,m,co=1]|},
rowhead = 1,
row{even} = {myblue},
row{odd} = {white},
row{1} = {white},
}
\hline
Group & Result & Proof & Group & Result & Proof & Group & Result & Proof\\
\hline
\hline
$G_{12,R}$ & not solid & Thm~\ref{notGsoliddiagonal} & $G_{12,R,T}$ & not solid & Thm~\ref{notGsoliddiagonal} & $G_{12,R,T_w}$ & not solid & Thm~\ref{notGsoliddiagonal}\\
\hline
$G_{48,R}$ & superrigid & Thm~\ref{Gbirsuprignontoric} & $G_{48,R,T}$ & superrigid & Thm~\ref{Gbirsuprignontoric} & $G_{48,R,T_w}$ & superrigid & Thm~\ref{Gbirsuprignontoric}\\
\hline
$G_{192,R}$ & superrigid & Thm~\ref{Gbirsuprignontoric} & $G_{192,R_b}$ & superrigid & Thm~\ref{Gbirsuprig192} & $G_{192,R,T}$ & superrigid & Thm~\ref{Gbirsuprignontoric}\\
\hline
$G_{192,R,T_w}$ & superrigid & Thm~\ref{Gbirsuprignontoric} & $G_{192,R_b,T_{bw}}$ & superrigid & Thm~\ref{Gbirsuprig192} & $G_{576,R}$ & superrigid & Thm~\ref{Gbirsuprignontoric}\\
\hline
$G_{576,R_b}$ & superrigid & Thm~\ref{Gbirsuprig192} & $G_{576,R,T}$ & superrigid & Thm~\ref{Gbirsuprignontoric} & $G_{576,R,T_w}$ & superrigid & Thm~\ref{Gbirsuprignontoric}\\
\hline
$G_{576,R_b,T_{w}}$ & superrigid & Thm~\ref{Gbirsuprig192} & $G_{1728,R}$ & superrigid & Thm~\ref{Gbirsuprignontoric} & $G_{1728,R,T}$ & superrigid & Thm~\ref{Gbirsuprignontoric}\\
\hline
$G_{1728,R,T_w}$ & superrigid & Thm~\ref{Gbirsuprignontoric} & $G_{24,R}$ & not solid & Thm~\ref{notGsoliddiagonal} & $G_{24,R,T}$ & not solid & Thm~\ref{notGsoliddiagonal}\\
\hline
$G_{96,R}$ & superrigid & Thm~\ref{Gbirsuprignontoric} & $G_{96,R,T}$ & superrigid & Thm~\ref{Gbirsuprignontoric} & $G_{384,R}$ & superrigid & Thm~\ref{Gbirsuprignontoric}\\
\hline
$G_{384,R,T}$ & superrigid & Thm~\ref{Gbirsuprignontoric} & $G_{1152,R}$ & superrigid & Thm~\ref{Gbirsuprignontoric} & $G_{1152,R,T}$ & superrigid & Thm~\ref{Gbirsuprignontoric}\\
\hline
$G_{3456,R}$ & superrigid & Thm~\ref{Gbirsuprignontoric} & $G_{3456,R,T}$ & superrigid & Thm~\ref{Gbirsuprignontoric} & $G_{6912,R}$ & superrigid & Thm~\ref{Gbirsuprignontoric}\\
\hline
$G_{6912,R,T}$ & superrigid & Thm~\ref{Gbirsuprignontoric} & $G_{6912,R,T_w}$ & superrigid & Thm~\ref{Gbirsuprignontoric} & $G_{13824,R}$ & superrigid & Thm~\ref{Gbirsuprignontoric}\\
\hline
$G_{13824,R,T}$ & superrigid & Thm~\ref{Gbirsuprignontoric} & $G_{60,R}$ & not solid & Thm~\ref{notGsoliddiagonal} & $G_{60,R,T}$ & not solid & Thm~\ref{notGsoliddiagonal}\\
\hline
$G_{216000,R}$ & superrigid & Thm~\ref{Gbirsuprignontoric} & $G_{216000,R,T}$ & superrigid & Thm~\ref{Gbirsuprignontoric} & & &\\
\hline
\end{longtblr}
\endgroup
\phantomsection

\addcontentsline{toc}{section}{References}

\bibliographystyle{alpha} 
\bibliography{refs}

\end{document}